\documentclass{amsart}
\usepackage{verbatim}
\usepackage[textsize=scriptsize]{todonotes}
\usepackage{longtable}

\usepackage{etoolbox}
\newtoggle{draft}
\newtoggle{final}

\usepackage{tikz}
\usetikzlibrary{matrix,arrows}
\usepackage{tikz-cd}
\usepackage{stmaryrd}
\usepackage{listings}

\toggletrue{final}

\usepackage{etoolbox}
\iftoggle{draft} {
\usepackage[margin=1.5in]{geometry}
}{ % else
\usepackage[margin=1in]{geometry}
\setlength{\marginparwidth}{0.75in}
}
\usepackage{amsmath}
\usepackage{amssymb}
\usepackage{amsthm}
\usepackage{amscd}
\usepackage{enumerate}
\usepackage[pdfusetitle,unicode,hidelinks]{hyperref}
\usepackage{bbm}
\usepackage{etoolbox}

\usepackage[utf8]{inputenc}
\usepackage[T1]{fontenc}

\newcommand{\tomemail}{\href{mailto:tom.bachmann@zoho.com}{tom.bachmann@zoho.com}}

\newtheorem{proposition}{Proposition}
\newtheorem{corollary}[proposition]{Corollary}
\newtheorem{lemma}[proposition]{Lemma}
\newtheorem{theorem}[proposition]{Theorem}

\newtheorem*{conjecture*}{Conjecture}
\newtheorem*{theorem*}{Theorem}
\newtheorem*{corollary*}{Corollary}
\newtheorem*{proposition*}{Proposition}
\newtheorem*{lemma*}{Lemma}
\theoremstyle{definition}
\newtheorem{definition}[proposition]{Definition}

\newtheorem*{definition*}{Definition}
\newtheorem*{construction*}{Construction}
\theoremstyle{remark}
\newtheorem{remark}[proposition]{Remark}
\newtheorem*{remark*}{Remark}

\newtheorem{example}[proposition]{Example}
\newtheorem*{example*}{Example}

\newcommand{\id}{\operatorname{id}}
\newcommand{\Z}{\mathbb{Z}}
\def\C{\mathbb C}

\newcommand{\Q}{\mathbb{Q}}
\newcommand{\F}{\mathbb{F}}

\let\scr=\mathcal
\let\bb=\mathbb
\newcommand{\Gm}{{\mathbb{G}_m}}
\newcommand{\Gmp}[1]{{\mathbb{G}_m^{\wedge #1}}}
\def\A{\bb A}
\def\P{\bb P}
\def\R{\bb R}
\newcommand{\1}{\mathbbm{1}}

\newcommand{\eff}{{\text{eff}}}
\newcommand{\veff}{{\text{veff}}}

\newcommand{\SH}{\mathcal{SH}}

\DeclareMathOperator*{\colim}{colim}

\let\lim=\relax
\DeclareMathOperator*{\lim}{lim}
\def\Map{\mathrm{Map}}

\def\map{\mathrm{map}}
\def\CAlg{\mathrm{CAlg}}
\def\CMon{\mathrm{CMon}}

\def\PSh{\mathcal{P}}

\def\Cat{\mathcal{C}\mathrm{at}{}}
\def\Spc{\mathcal{S}\mathrm{pc}{}}
\def\Fin{\cat F\mathrm{in}}
\def\Fun{\mathrm{Fun}}
\def\red{\mathrm{red}}
\newcommand{\Spec}{\mathrm{Spec}}
\newcommand{\gp}{\mathrm{gp}}

\newcommand{\wequi}{\simeq}
\newcommand{\Mod}{\text{-}\mathcal{M}\mathrm{od}}
\def\adj{\leftrightarrows}

\DeclareRobustCommand{\ul}{\underline}
\newcommand{\heart}{\heartsuit}

\newcommand{\Hom}{\operatorname{Hom}}
\newcommand{\iHom}{\ul{\operatorname{Hom}}}
\def\op{\mathrm{op}}

\let\cat=\mathrm
\def\Sm{{\cat{S}\mathrm{m}}}
\def\Sch{\cat{S}\mathrm{ch}{}}

\def\Nis{\mathrm{Nis}}
\def\Zar{\mathrm{Zar}}

\def\mot{\mathrm{mot}}

\newcommand{\ret}{{r\acute{e}t}}

\newcommand{\lra}[1]{\langle #1 \rangle}
\def\ph{\mathord-}

\numberwithin{proposition}{section}
\numberwithin{equation}{section}

\iftoggle{final} {
\renewcommand{\todo}[1]{}
\newcommand{\NB}[1]{}
\newcommand{\discuss}[1]{}
}{ % else

\newcommand{\NB}[1]{\todo[color=gray!40]{#1}}
\newcommand{\discuss}[1]{\todo[color=green]{#1}}
}

\newcommand{\fpsr}[1]{\llbracket #1 \rrbracket}

\newcommand{\Ab}{\mathrm{Ab}}
\newcommand{\Ext}{\mathrm{Ext}}
\newcommand{\Vect}{\mathrm{Vect}}
\newcommand{\Bil}{\mathrm{Bil}}
\newcommand{\Grpd}{\mathcal{G}\mathrm{rpd}}
\newcommand{\Sper}{\mathrm{Sper}}
\newcommand{\Sym}{\mathrm{Sym}}

\newcommand{\KO}{\mathrm{KO}}
\newcommand{\KSp}{\mathrm{KSp}}
\newcommand{\ko}{\mathrm{ko}}
\newcommand{\ku}{\mathrm{ku}}
\newcommand{\KU}{\mathrm{KU}}
\newcommand{\ksp}{\mathrm{ksp}}
\newcommand{\KGL}{\mathrm{KGL}}
\newcommand{\kgl}{\mathrm{kgl}}
\newcommand{\kw}{\mathrm{kw}}
\newcommand{\KW}{\mathrm{KW}}
\newcommand{\HW}{\mathrm{HW}}
\newcommand{\HZ}{\mathrm{H}\Z}
\newcommand{\MSL}{\mathrm{MSL}}
\newcommand{\MSO}{\mathrm{MSO}}
\newcommand{\MSp}{\mathrm{MSp}}
\newcommand{\MU}{\mathrm{MU}}
\newcommand{\BSU}{\mathrm{BSU}}
\newcommand{\BU}{\mathrm{BU}}

\newcommand{\BSp}{\mathrm{BSp}}
\newcommand{\SL}{\mathrm{SL}}
\newcommand{\Sp}{\mathrm{Sp}}
\newcommand{\MGL}{\mathrm{MGL}}
\newcommand{\jo}{\mathrm{j}_o}
\newcommand{\W}{\mathrm{W}}
\newcommand{\I}{\mathrm{I}}
\newcommand{\GW}{\mathrm{GW}}
\newcommand{\K}{\mathrm{K}}
\newcommand{\HP}{\mathbb{HP}}
\newcommand{\Gr}{\mathrm{Gr}}
\newcommand{\HGr}{\mathrm{HGr}}
\newcommand{\SGr}{\mathrm{SGr}}
\newcommand{\RP}{\mathbb{RP}}
\newcommand{\GL}{\mathrm{GL}}
\newcommand{\Perf}{\mathrm{Perf}}
\newcommand{\Sq}{\mathrm{Sq}}
\newcommand{\gr}{\mathrm{gr}}
\newcommand{\Th}{\mathrm{Th}}
\newcommand{\Pic}{\mathrm{Pic}}
\newcommand{\limone}{\mathrm{lim}^1}

\def\h{\mathrm h}
\def\im{\mathrm{im}}

\newcommand{\comp}{{{\kern -.5pt}\wedge}}
\newcommand{\vcd}{\mathrm{vcd}}
\newcommand{\cd}{\mathrm{cd}}
\newcommand{\scomp}{\mathrm{sc}}
\newcommand{\cof}{\mathrm{cof}}
\newcommand{\fib}{\mathrm{fib}}

\newcommand{\dual}{*}
\newcommand{\rk}{\mathrm{rk}}
\newcommand{\topsup}{\mathrm{top}}
\newcommand{\chara}{\mathrm{char}}

\newcommand{\Snaith}{\mathrm{S}}
\newcommand{\OSnaith}{{h\mathrm{S}}}
\newcommand{\Riou}{\mathrm{Riou}}
\newcommand{\FH}{\mathrm{FH}}
\newcommand{\geo}{\mathrm{geo}}

\newcommand{\adamsphi}{\varphi}
\newcommand{\adamspsi}{\psi}

\newtheorem{warning}[proposition]{Warning}

\title{\texorpdfstring{$\eta$}{eta}-periodic motivic stable homotopy theory over fields}
\date{\today}

\author{Tom Bachmann}
\address{Department of Mathematics, Massachusetts Institute of Technology,
Cambridge, MA, USA}
\email{\tomemail}
\author{Michael J. Hopkins}

\begin{document}

\maketitle

\begin{abstract}
Over any field of characteristic $\ne 2$, we establish a 2-term resolution of the $\eta$-periodic, $2$-local motivic sphere spectrum by shifts of the connective $2$-local Witt $K$-theory spectrum.
This is curiously similar to the resolution of the $K(1)$-local sphere in classical stable homotopy theory.
As applications we determine the $\eta$-periodized motivic stable stems and the $\eta$-periodized algebraic symplectic and $\SL$-cobordism groups.
Along the way we construct Adams operations on the motivic spectrum representing Hermitian $K$-theory and establish new completeness results for certain motivic spectra over fields of finite virtual $2$-cohomological dimension.
In an appendix, we supply a new proof of the homotopy fixed point theorem for the Hermitian $K$-theory of fields.
\end{abstract}

\tableofcontents

\section{Introduction} \label{sec:intro}

\subsection{Inverting the motivic Hopf map}
Let $k$ be a field.
We are interested in studying the motivic stable homotopy category $\SH(k)$ \cite[\S5]{morel-trieste}.
Recall that this category is obtained from the category $\Sm_k$ of smooth $k$-varieties by (1) freely adjoining (homotopy) colimits, (2) enforcing Nisnevich descent, (3) making the affine line $\A^1$ contractible, (4) passing to pointed objects and (5) making the operation $\wedge \P^1$ invertible: \[ \SH(k) = L_{\A^1,\Nis} \PSh(\Sm_k)_*[(\P^1)^{-1}]. \]

Of particular relevance for us is the \emph{motivic Hopf} map.
Recall that the geometric Hopf map is the morphism of pointed varieties \begin{equation} \label{eq:geometric-hopf} \A^2 \setminus 0 \to \P^1 \end{equation} obtained from the construction of $\P^1$ as a quotient.
There are $\A^1$-equivalences \cite[\S3 2.15--2.20]{A1-homotopy-theory} \[ \P^1 \wequi S^1 \wedge \Gm \quad\text{and}\quad \A^2 \setminus 0 \wequi \P^1 \wedge \Gm; \] here $\Gm := (\A^1 \setminus 0, 1)$.
It follows in particular that all of these objects become invertible in $\SH(k)$ and so are ``spheres''.
We may thus desuspend \eqref{eq:geometric-hopf} by $\P^1$ to obtain the motivic Hopf map \[ \eta: \Gm \to \1 \in \SH(k); \] here we write $\1$ for the sphere spectrum.
This is completely analogous to the classical stable Hopf map $\eta^\topsup: S^1 \to \1 \in \SH$.
However, while classically we have $\eta^4 = 0$, in $\SH(k)$ no power of $\eta$ is null-homotopic; this is a theorem of Morel \cite[Corollary 6.4.5]{morel2004motivic-pi0}.
In slightly more sophisticated language, denote by \[ \SH(k)[\eta^{-1}] \subset \SH(k) \] the subcategory of \emph{$\eta$-periodic spectra}; in other words given $E \in \SH(k)$ we have $E \in \SH(k)[\eta^{-1}]$ if and only if \[ \Gm \wedge E \xrightarrow{\eta \wedge \id} \1 \wedge E \] is an equivalence.
The inclusion $\SH(k)[\eta^{-1}] \subset \SH(k)$ admits a left adjoint, given by the \emph{$\eta$-periodization} \[ E \mapsto E[\eta^{-1}] = \colim(E \xrightarrow{\eta} \Gmp{-1} \wedge E \xrightarrow{\eta} \dots). \]
Morel's result shows that $\SH(k)[\eta^{-1}]$ is not the terminal category; in fact he proves that \[ \pi_0(\1[\eta^{-1}]) \wequi \W(k) \ne 0, \] where $\W(k)$ is the \emph{Witt ring of symmetric bilinear forms}.

\subsection{Main results}
Our aim is to study the category $\SH(k)[\eta^{-1}]$; for example we would like to determine the higher homotopy groups $\pi_*(\1[\eta^{-1}])$ for $*>0$.
It turns out (see \S\ref{subsec:rho}) that the category $\SH(k)[1/\eta, 1/2]$ is quite easy to understand.
For many purposes it thus suffices to study the category $\SH(k)[\eta^{-1}]_{(2)} \subset \SH(k)[\eta^{-1}]$ of those $\eta$-periodic spectra $E$ such that $E \xrightarrow{p} E$ is an equivalence for all odd integers $p$.
Our main result is the following.
Here \[ \kw = \KO[\eta^{-1}]_{\ge 0} \] is the spectrum of \emph{connective Balmer--Witt $K$-theory}; see \S\ref{subsec:spectra} for a more detailed definition.

\begin{theorem}[see Theorem \ref{thm:main}] \label{thm:intro-main}
Let $k$ be a field of characteristic $\ne 2$.
In $\SH(k)$ there exists a fiber sequence \[ \1[\eta^{-1}]_{(2)} \to \kw_{(2)} \xrightarrow{\adamsphi} \Sigma^4 \kw_{(2)}. \]
Here $\adamsphi$ is the unique (up to homotopy) map making the following diagram commute
\begin{equation*}
\begin{tikzcd}
 & \Sigma^4 \kw_{(2)} \ar[d, "\beta"] \\
\kw_{(2)} \ar[r, "\adamspsi^3-1"] \ar[ru, dotted, "\adamsphi"]  & \kw_{(2)}
\end{tikzcd},
\end{equation*}
where $\adamspsi^3$ is the third Adams operation and $\beta$ is the Bott element.
\end{theorem}
It is well-known (see \S\ref{subsec:KW} for a recollection) that $\pi_* \kw = \W(k)[\beta]$.
On the other hand the construction of well-behaved Adams operations on $\kw$ is one of the major technical challenges of our work (see \S\ref{sec:adams}).
We observe that this result is curiously similar to the resolution of the $K(1)$-local sphere in classical stable homotopy theory \cite[Lemma 8.7]{ravenel1984localization}; the reader is invited to speculate on the significance of this observation.
\begin{remark}
As an immediate consequence, if (for example) $k=\C$ then $\beta$ lifts to $\pi_4 \1[\eta^{-1}]_{(2)}$.
Hornbostel has shown \cite[Theorem 3.2]{hornbostel2018some} that in this case \[ \pi_*(\cof(\1[\eta^{-1}]_{(2)} \to \kw_{(2)})[\beta^{-1}]) \wequi \pi_* \Sigma^4\kw_{(2)}[\beta^{-1}]. \]
Ormsby--Röndigs \cite[p. 13]{ormsby2019homotopy} construct (still over $\C$) a map $\Sigma^3 \kw \to \1[\eta^{-1}]_{(2)}$ inducing \[ \pi_*(\cof(\1[\eta^{-1}]_{(2)} \to \kw_{(2)})) \wequi \pi_* \Sigma^4 \kw_{(2)}. \]
These are weak forms of our main result.
In fact Ormsby--Röndigs \cite[Conjecture 4.11]{ormsby2019homotopy} explicitly ask a question equivalent to our main theorem.
\end{remark}

The following are some consequences which can be obtained by more or less immediate computation.
\begin{corollary}[see Theorems \ref{thm:stable-stems}, \ref{thm:MSp-htpy} and \ref{thm:homotopy-msl}]
Let $k$ be a field of characteristic $\ne 2$.
\begin{enumerate}
\item \[ \pi_*(\1_k[\eta^{-1}]_{(2)}) \wequi \begin{cases} \W(k)_{(2)} & *=0 \\ \mathrm{coker}(8n: \W(k)_{(2)} \to \W(k)_{(2)}) & *=4n-1>0 \\ \ker(8n: \W_{(2)} \to \W_{(2)}) & *=4n>0 \\ 0 & \text{else} \end{cases} \]
\item $\pi_* \MSp[\eta^{-1}] \wequi \W(k)[y_1, y_2, \dots]$, where $|y_i| = 2i$
\item $\pi_* \MSL[\eta^{-1}] \wequi \W(k)[y_2, y_4, \dots]$
\end{enumerate}
Here $\MSp$ and $\MSL$ denote the algebraic symplectic and special linear cobordism spectra \cite{panin2010algebraic}, respectively.
\end{corollary}
Result (1) above (i.e. the computation of $\pi_*(\1[\eta^{-1}])$) has attracted substantial attention before; see \S\ref{subsec:stable-stems} for a review of previous work.
To the best of our knowledge, results (2) and (3) are new in all cases.

As another application, we obtain the following cellularity results.
\begin{corollary}[see Proposition \ref{prop:cellularity}]
Let $k$ have exponential characteristic $e \ne 2$.
Then the spectra $\ko[1/e], H\tilde\Z[1/e] \in \SH(k)$ are cellular.
\end{corollary}

See \S\ref{sec:applications} for further applications, among them the determination of the $\HW_{(2)}$-cooperations and $\kw_{(2)}$-operations.

\subsection{Proof sketch}
To orient the reader, we provide a rapid sketch of our proof of Theorem \ref{thm:intro-main}.
Granting the construction of $\adamsphi$, we proceed as follows.
Write $F$ for the fiber of $\adamsphi$.
By a connectivity argument, it suffices to show that the canonical map $\HW_{(2)} \xrightarrow{\alpha} F \wedge \HW$ is an equivalence, where $\HW = \1[\eta^{-1}]_{\le 0}$ is the $\eta$-periodic Eilenberg-MacLane spectrum.
By base change, it suffices to prove the result for prime fields; in particular we may assume that $\vcd_2(k) < \infty$.
It suffices to show that $\alpha[1/2]$ and $\alpha/2$ are equivalences, and since we are working over an arbitrary field (of $\vcd_2 < \infty$ and characteristic $\ne 2$) it suffices to check that we have an isomorphism on homotopy groups.
The homotopy groups of $\kw$ are given by $\W(k)[\beta]$, where $\W(k)$ is the Witt ring and $\beta \in \pi_4 \kw$ is the Bott element.

For $\alpha[1/2]$ we are dealing with a rational problem; in particular $\HW \otimes \Q \wequi \1[\eta^{-1}] \otimes \Q$ and so $\pi_* (\HW \wedge \kw) \otimes \Q \wequi \W(k)[\beta] \otimes \Q$.
It follows that for $n>0$ we have $\adamsphi(\beta^n) = a_n\beta^{n-1}$ for some $a_n \in \W(k) \otimes \Q$, which we need to show is a unit.
One of the basic properties of the construction of $\adamsphi$ (related to \eqref{eq:adams-formula} below) is that $\adamsphi(\beta^n) = (9^n-1) \beta^{n-1}$, so $\alpha[1/2]$ is indeed an equivalence.

The case of $\alpha/2$ is more difficult.
We may prove the result for the $2$-adic completion $\alpha_2^\comp$ instead.
Note that under our assumption that $\vcd_2(k) < \infty$ we have $\W(k)_2^\comp \wequi \W(k)_\I^\comp$, where $\I \subset \W(k)$ is the fundamental ideal.
Our major intermediate result is as follows.
\begin{lemma}[see Theorem \ref{thm:HW-kw-summary}] \label{lemm:intro-HW-kw}
We have \[ \pi_*((\kw \wedge \HW)_2^\comp) \wequi \begin{cases} \W(k)_\I^\comp & 0 \le * \equiv 0 \pmod{4} \\ 0 &\text{else} \end{cases}; \] moreover if $x_i \in \pi_{4i}((\kw \wedge \HW)_2^\comp)$ is a generator (of the free $\W(k)_\I^\comp$-module $\pi_{4i}((\kw \wedge \HW)_2^\comp)$) then so is its base change to any larger field (of finite $\vcd_2$).
\end{lemma}
Consequently for $n>0$ we have $\adamsphi(x_n) = b_nx_{n-1}$ for some $b_n \in \W(k)_\I^\comp$, which we need to show is a unit.
Since $\W(k)_\I^\comp$ is a local ring with residue field $\F_2$ independent of $k$, and the generators are compatible with base change, we may extend $k$ arbitrarily.
We may thus assume that $k$ is quadratically closed, so that $\W(k) = \F_2$.

We now employ the motivic special linear cobordism spectrum $\MSL$.
Since $\kw$ is $\SL$-oriented and $\eta$-periodic, by work of Ananyveskiy \cite{ananyevskiy2015special} (see also \S\ref{sec:cobordism}) we have \[ \pi_*(\kw \wedge \MSL) \wequi \kw_*[e_2, e_4, \dots], \] with $|e_{2i}| = 4i$.
In the case when $\W(k) = \F_2$, we have partial information on the action of $\adamsphi$ on $\kw_* \MSL$.
\begin{lemma}[see Lemma \ref{lemm:phi-MSL}]
Suppose that $\W(k) = \F_2$.
Then $\adamsphi^{\circ i}(e_{2i}) = 1$.
\end{lemma}
Now consider the canonical ring map $\gamma: \pi_*(\kw \wedge \MSL) \to \pi_*((\kw \wedge \HW)_2^\comp)$.
We get \[ \adamsphi^{\circ i}\gamma(e_{2i}) = \gamma(\adamsphi^{\circ i}(e_{2i})) = \gamma(1) = 1 \ne 0. \]
This implies that $\adamsphi(\gamma(e_{2i})) \ne 0$, and so $b_i \ne 0$ as needed.

\subsection{Organization}
We begin in \S\ref{sec:preliminaries} by recalling well-known results and setting out notation.
It can probably be skipped and referred to only as needed.

\S\S\ref{sec:adams}--\ref{sec:HW-kw} contain preparations for the proof of our main theorem.
They all begin with a subsection called ``summary'', in which the main results of that section are listed.
Other sections will only refer to the results stated in the summary subsections.
The reader willing to take on trust the results stated in the summary subsection may thus immediately skip to the next section.

In \S\ref{sec:adams} we construct Adams operations for the motivic ring spectrum $\KO$, and we prove the important formula \begin{equation}\label{eq:adams-formula} \adamspsi^n(\beta) = n^2n^2_\epsilon \beta. \end{equation}
We construct the Adams operations as $\scr E_\infty$-ring maps \[ \adamspsi^n: \KO[1/n] \to \KO[1/n] \in \SH(S), \] for $n$ odd.
Our construction begins by making the Adams operations on $\KGL$ into $C_2$-equivariant maps of $\scr E_\infty$-rings, by using the motivic Snaith theorem \cite{gepner2009motivic,spitzweck2009bott}\footnote{This result was proved independently by Gepner--Snaith and Spitzweck--{\O}stv{\ae}r. Our terminology reflects the analogy with the classical result, which was established by Snaith \cite{snaith1979algebraic}.}.
Then we take $2$-adically completed homotopy fixed points which, by the homotopy fixed point theorem for Hermitian $K$-theory, gives us operations on $\KO_2^\comp$.\footnote{At least under some finiteness assumptions, which we ignore for the purposes of this introduction.}
By means of a fracture square, we combine this with a more naive operation on $\KO[1/2n]$ to yield the Adams operation on $\KO[1/n]$.
Since our definition is rather indirect, establishing \eqref{eq:adams-formula} is a fairly delicate problem.
Our proof eventually boils down to discreteness of the space of $\scr E_\infty$-endomorphisms of the classical spectrum $\ku$, which is a well-known consequence of Goerss--Hopkins obstruction theory \cite{goerss2004moduli}.\footnote{In fact we only need discreteness of the space of $\scr E_1$-endomorphisms of $\KU$, which was established in unpublished work of McLure.}

In \S\ref{sec:cobordism} we collect some results about the motivic cobordism spectra $\MSL$ and $\MSp$.
Firstly we dualize the well-known computations of the $A$-cohomology of $\MSL$ and $\MSp$ if $A$ is $\SL$-oriented and $\eta$-periodic, to obtain the $A$-homology of $\MSL$ and $\MSp$.
Secondly, under the additional assumption that $\W(k) = \F_2$, we obtain some information about the action of $\adamspsi^3$ on $\kw_* \MSL$ and $\kw_*\MSp$.

In \S\ref{sec:completeness} we establish some new completeness results.
Specifically, if $\vcd_2(k) < \infty$ and $E \in \SH(k)^\veff$, then we show that $E_2^\comp$ is $\eta$-complete.
This is deduced from an improved version of Levine's slice convergence theorem, which was recently established \cite[\S5]{bachmann-bott}.

We leverage this in \S\ref{sec:HW-kw} to compute $\pi_*((\kw \wedge \HW)_2^\comp)$, i.e. prove Lemma \ref{lemm:intro-HW-kw}.
This employs the equivalences \[ (\kw \wedge \HW)_2^\comp \wequi (\ko \wedge \ul{K}^W)[\eta^{-1}]_2^\comp \wequi (\ko \wedge \ul{K}^W)_2^\comp[\eta^{-1}]_2^\comp \wequi (\ko \wedge \ul{K}^W)_{\eta,2}^\comp[\eta^{-1}]_2^\comp. \]
Here the most important (and nontrivial) step is the last one, which uses our new completeness result (and the fact that $\Sigma^{2,1}\ul{K}^W \in \SH(k)^\veff$, which was essentially established in \cite{bachmann-very-effective}).
Using that \[ \ul{K}^W/\eta \wequi \ul{k}^M \wequi (\HZ/2)/\tau, \] we compute $\pi_{**}(\ko \wedge \ul{K}^W)_{\eta}^\comp$ by employing a Bockstein spectral sequence, together with the known computation of $\HZ/2_{**}\ko$ \cite{ananyevskiy2017very}.

With all this preparation out of the way, in \S\ref{sec:main} we prove our main result, following essentially the argument sketched above.
Finally in \S\ref{sec:applications} we establish various applications.

In appendix \S\ref{sec:htpy-fixed} we provide an alternative proof of the homotopy fixed point theorem for Hermitian $K$-theory (which was first established by Hu--Kriz--Ormsby \cite{hu2011homotopy}).
It utilizes the improved version of Levine's slice convergence theorem mentioned above, together with the computation by R{\"o}ndigs--{\O}stv{\ae}r of the slice spectral sequence for $\KW$ \cite{rondigs2013slices}.

\subsection{Acknowledgements}
It is our pleasure to thank Robert Burklund and Zhouli Xu for extensive discussions about the homotopy groups of $\MSL[\eta^{-1}]$ over $\C$.
We would further like to thank Alexey Ananyevskiy, Marc Hoyois, Dan Isaksen, Tyler Lawson, Denis Nardin, Oliver Röndigs and Dylan Wilson for helpful comments.

\subsection{Conventions}
All our statements directly or indirectly involving hermitian $K$-theory require the assumption that $2$ is invertible in the base scheme.
We may omit specifying this explicitly to avoid tedious repetition.

Given a map of spectra $\alpha: E \to F$, we denote by $\alpha$ also the induced map $\pi_* E \to \pi_* F$.

We denote the category of motivic spectra over a scheme $S$ by $\SH(S)$, and assume basic familiarity with its construction and properties.
See e.g. \cite[\S\S2.2,4.1]{bachmann-norms}.
We view this as a presentably symmetric monoidal, stable $\infty$-category \cite{lurie-ha,lurie-htt}, and we assume some familiarity with the theory of such categories (in particular in \S\ref{sec:adams}).

\subsection{Table of notation}\NB{more reasonable order?} \label{subsec:table-of-notation}
\begin{longtable}{lll}
$\nu_2$ & $2$-adic valuation \\
$\vcd_2(k)$ & $2$-étale cohomological dimension of $k[\sqrt{-1}]$ \\
$A[p^n]$ & $p^n$-torsion in abelian group $A$ \\
$A_p^\comp$ & classical $p$-completion of abelian group \\
$L_p^\comp A$ & derived $p$-completion of abelian group & \S\ref{subsec:derived-completion} \\
$E_p^\comp$ & $p$-completion of spectrum & \S\ref{subsec:completion} \\
$E[1/n], E_{(p)}$ & localization of spectrum & \S\ref{subsec:completion} \\
$\1$ & motivic sphere spectrum \\
$\KO, \KGL$ & hermitian and algebraic $K$-theory motivic spectra & \S\ref{subsec:KO-KGL} \\
$\KW, \kw, \ko$ & variants of the above & \S\ref{subsec:spectra} \\
$\KO^\topsup$ & classical orthogonal $K$-theory spectrum \\
$\K^\circ$ & rank $0$ summand of $K$-theory space \\
$\beta, \beta_\KGL$ & Bott element in $\KO, \KGL$ & \S\ref{subsec:KO-KGL} \\
$\adamspsi^n$ & Adams operation & \S\ref{sec:adams} \\
$\adamspsi^n_\Snaith, \adamspsi^n_\OSnaith$ & Adams operations & \S\ref{subsec:GS-htpy} \\
$\adamsphi$ & modified Adams operation & \S\ref{sec:main} \\
$\ul{K}^M$ & homotopy module of (unramified) Milnor $K$-theory \\
$\ul{k}^M$ & homotopy module of mod $2$ (unramified) Milnor $K$-theory \\
$\ul{K}^W$ & homotopy module of (unramified) Witt $K$-theory & \S\ref{subsec:f0HW} \\
$\HZ$ & motivic cohomology spectrum & \S\ref{subsec:steenrod} \\
$\tau$ & Bott element in $\HZ/2$ \\
$\HW$ & $\eta$-periodic Witt cohomology spectrum & \S\ref{subsec:KW} \\
$\MSL, \MSp$ & algebraic $\SL$-, $\Sp$-cobordism motivic spectra & \S\ref{sec:cobordism} \\
$\HP^\infty$ & infinite quaternionic projective space \\
$\HGr, \SGr, \Gr$ & quaternionic, special linear, and ordinary Grassmannians & \S\ref{subsec:real-realization} \\
$\boxplus$ & external product of vector bundles (on a product of varieties) \\
$\Th(V)$ & Thom space of a vector bundle, $V/V \setminus 0$ \\
$\eta$ & motivic Hopf map & \S\ref{sec:intro} \\
$\rho$ & standard endomorphisms of motivic sphere spectrum & \S\ref{subsec:rho} \\
$n_\epsilon$ & Milnor-Witt integer $1 + \lra{-1} + \dots + \lra{\pm 1}$ \\
$\W, \I$ & Witt ring and fundamental ideal & \S\ref{subsec:witt-groups} \\
$\ul{W}, \ul{I}^n$ & sheaf of Witt rings and fundamental ideals \\
$E[1/2]^+, E[1/2]^-$ & plus and minus part of a $2$-periodic spectrum & \S\ref{subsub:plus-minus} \\
$E^+, E^-$ & generalized plus and minus part & \S\ref{subsec:2-fracture} \\
$\pi_{i, j}(E), \pi_i(E)_j$ & bigraded homotopy groups & \S\ref{subsec:t-structure} \\
$\ul{\pi}_{i, j}(E), \ul{\pi}_i(E)_j$ & bigraded homotopy sheaves & \S\ref{subsec:t-structure} \\
$E_{\ge 0}, E_{\le 0}$ & truncations in the homotopy $t$-structure & \S\ref{subsec:t-structure} \\
$S^{p, q}$ & motivic sphere $S^{p-q} \wedge \Gmp{q}$ \\
$\Sigma^{p, q}$ & bigraded suspension (smashing with $S^{p,q}$) \\
$\SH(S)$ & category of motivic spectra over $S$ \\
$\SH(k)^\veff$ & category of very effective spectra & \S\ref{subsec:very-effective} \\
$\Spc(S)$ & category of motivic spaces over $S$ \\
$f_i$ & effective cover functor & \cite{voevodsky-slice-filtration} \\
$s_i$ & slice functor & \cite{voevodsky-slice-filtration} \\
$\tilde f_i$ & very effective cover functor & \cite{bachmann-very-effective} \\
$\tilde s_i$ & generalized slice functor & \cite{bachmann-very-effective} \\
$r_\R$ & real realization functor & \S\ref{subsec:rho} \\
$\PSh(\scr C)$ & category of presheaves of spaces on $\scr C$ \\
$\h\scr C$ & homotopy $1$-category of $\infty$-category \\
$\Map(\ph, \ph)$ & mapping space in an $\infty$-category \\
$\map(\ph, \ph)$ & mapping spectrum in a stable $\infty$-category $\scr C$ \\
$[\ph,\ph]$ & homotopy classes of maps, i.e. $\pi_0\Map(\ph,\ph)$ \\
$\Spc$ & category of spaces \\
$\SH$ & category of spectra \\
$\Grpd$ & category of $\infty$-groupoids (so $\Grpd \wequi \Spc$) \\
$\Sch, \Sch_S$ & category of qcqs schemes (over $S$) \\
$\Sm_S$ & category of smooth, qcqs schemes over $S$ \\
$\Vect_S$ & $1$-groupoid of vector bundles on $S$ \\
$\Perf_S$ & category of perfect complexes on $S$
\end{longtable}

\section{Preliminaries and Recollections} \label{sec:preliminaries}
In this section we collect various well-known results which will be used throughout the sequel.
About half of them are specific to motivic homotopy theory (bigraded homotopy sheaves \S\ref{subsec:t-structure}, completion and localization of motivic spectra \S\ref{subsec:completion}, very effective spectra \S\ref{subsec:very-effective}, and the relationship between inverting $\rho$ and real realization \S\ref{subsec:rho}), whereas the other half is about more general algebra and homotopy theory (filtered modules \S\ref{subsec:filtered}, derived completion of abelian groups \S\ref{subsec:derived-completion}, and Witt groups \S\ref{subsec:witt-groups}).
We encourage the reader to skip this section and refer back as needed.

\subsection{Filtered modules} \label{subsec:filtered}
\subsubsection{}
We follow \cite[\S2]{boardman1999conditionally}.
Thus by a filtered abelian group $G$ we mean a family of subgroups \[ G \supset \cdots \supset F^s G \supset F^{s+1} G \supset \cdots. \]
We denote by \[ \gr^n(G) = F^nG/F^{n+1}G \] the associated graded.
\begin{definition}[\cite{boardman1999conditionally}, Proposition 2.2]
Let $F^\bullet G$ be a filtered abelian group.
\begin{enumerate}
\item The filtration is \emph{exhaustive} if $G = \colim_s F^sG$.
\item The filtration is \emph{Hausdorff} if $0 = \lim_s F^s G$.
\item The filtration is \emph{complete} if $0 = \limone_s F^s G$.
\end{enumerate}
\end{definition}
More generally, we may be working with ((bi)graded) filtered modules over a ((bi)graded) filtered ring.
Since the forgetful functor from modules to abelian groups preserves limits and colimits, there is no ambiguity in the definition of exhaustive/Hausdorff/complete.
When working with graded objects this is no longer the case, and the definitions have to be applied degreewise.

\subsubsection{}
\begin{lemma}[\cite{boardman1999conditionally}, Theorem 2.6] \label{lemm:filtration-comparison}
Let $\alpha: G \to G'$ be a morphism of filtered groups.
Assume that both filtrations are Hausdorff and exhaustive, and that the filtration on $G$ is complete.
Then $\alpha$ is an isomorphism of filtered groups (i.e. $\alpha$ induces $G \wequi G'$ and $F^s G \wequi F^s G'$) if and only if $\gr^\bullet(\alpha)$ an isomorphism.
\end{lemma}

\begin{corollary} \label{corr:filtered-free-lifting}
Let $R$ be a graded filtered ring.
Assume that $R$ is concentrated in non-negative degrees and that the filtration is exhaustive, Hausdorff and complete.

\begin{enumerate}
\item Let $M$ be a graded filtered $R$-module, such that the filtration is exhaustive and Hausdorff.
  If $\gr^\bullet(M)$ is a free (bigraded) $\gr^\bullet(R)$-module with only finitely many generators in external degrees (i.e. degree coming from the grading on $M$) $\le n$ for every $n$, then $M$ is a free $R$-module on corresponding generators.
\item Let $A$ be a graded filtered $R$-algebra, such that the filtration is exhaustive and Hausdorff.
  If $\gr^\bullet(A)$ is a polynomial $\gr^\bullet(R)$-algebra on generators in positive external degrees, only finitely many of which lie in any degree, then $A$ is a polynomial $R$-algebra on corresponding generators.
\end{enumerate}
\end{corollary}
\begin{proof}
(1) Choose generators $\{\bar x_i\}_{i \in I}$ of $\gr^\bullet(M)$ in bidegree $(s_i,t_i)$, where $t_i$ corresponds to the original grading on $M$ and $s_i$ to the filtration.
Lift them to $x_i \in F^{s_i} M_{t_i}$.
Consider \[ M' = \bigoplus_i R[t_i]\{s_i\}, \] where $R[t_i]\{s_i\}$ denotes a free $R$-module on a generator in degree $t_i$ and filtration $s_i$.
There is thus a canonical map $M' \to M$ inducing an isomorphism on $\gr^\bullet$.
It remains to show that $M'$ is complete, exhaustive and Hausdorff.
By assumption, for each $n$, \[ M'_n = \bigoplus_{i: t_i \le n} R_{n-t_i}\{s_i\} \] is a finite sum of complete exhaustive and Hausdorff filtered abelian groups, and so has the same properties.

(2) Choose polynomial generators $\{\bar y_i\}_{i \in I}$ of $\gr^\bullet(A)$ in bidegree $(s_i, t_i)$.
Lift them to $y_i \in F^{s_i} A_{t_i}$.
Then monomials in the $\bar y_i$ are module generators of $\gr^\bullet(A)$ and our assumption implies that there are only finitely many monomials in degrees $\le n$ for any $n$.
Thus (1) applies and $A$ is the free $R$-module on monomials in the $y_i$.
Let $A'$ be the polynomial $R$-algebra on generators $\{y_i\}$.
There is a canonical ring map $A' \to A$, and it is an isomorphism since the underlying module map is.
\end{proof}

\begin{lemma} \label{lemm:filtered-surjection}
Let $\alpha: G \to G'$ be a morphism of filtered (possibly graded) groups.
Suppose that the filtration on $G'$ is exhaustive and Hausdorff, and the filtration on $G$ is complete.
Suppose furthermore that $\gr^\bullet(\alpha)$ is surjective.
Then $\alpha$ is surjective, and also each $F^p\alpha$ is surjective.

Moreover in this situation \[ \gr^\bullet \ker(\alpha) \wequi \ker(\gr^\bullet(\alpha)). \]
\end{lemma}
\begin{proof}
By applying the argument degreewise, we may assume that we are in the ungraded situation.

First observe that if $x \in F^N G'$ then there exists $y_0 \in F^N G$ with $\alpha(y_0) \in x + F^{N+1}G$; this is just surjectivity of $\gr^N(\alpha)$.
Applying this to $x - \alpha(y_0) \in F^{N+1}G'$ we obtain $y_1 \in F^{N+2} G$ with $\alpha(y_0 + y_1) \in x + F^{N+2}G'$.
Iterating we find $y_i \in F^{N+i}G$ (for all $i \ge 0$) with $\alpha(y_0 + \dots + y_n) \in x + F^{N+n+1}G'$ (for any $n\ge 0$).
By completeness of $G$, the series $\sum_i y_i$ converges to (a possibly non-unique element) $y \in F^NG$ with $\alpha(y) - x \in F^{N+n+1} G'$ for every $n$; hence by Hausdorffness of $G'$ we get $\alpha(y) = x$.
Thus $F^N\alpha$ is surjective.

Since $F^\bullet G'$ is exhaustive, every $x \in G'$ satisfies $x \in F^N G'$ for some $N$; hence the above argument shows that $x$ is in the image of $\alpha$ and hence $\alpha$ is surjective.

For the claim about kernels, apply the snake lemma \cite[Lemma 1.3.2]{weibel-hom-alg} to the diagram of exact sequences
\begin{equation*}
\begin{CD}
0 @>>> F^{p+1} \ker(\alpha) @>>> F^{p+1} G @>>> F^{p+1} G' @>>> 0 \\
@.      @VVV                    @VVV               @VVV \\
0 @>>> F^{p} \ker(\alpha) @>>> F^{p} G @>>> F^{p} G' @>>> 0
\end{CD}
\end{equation*}
to get the exact sequence \[ 0 = \ker(F^{p+1} G' \to F^p G') \to \gr^{p} \ker(\alpha) \to \gr^pG \xrightarrow{\gr^p \alpha} \gr^p G' \to 0. \]
This is the desired result.
\end{proof}

\begin{lemma} \label{lemm:gr-injective}
Let $\alpha: G \to G'$ be a morphisms of filtered (possibly graded) groups.
Assume that $G$ is Hausdorff and exhaustive and $\gr^\bullet(\alpha)$ is injective.
Then $\alpha$ is injective.
\end{lemma}
\begin{proof}
By applying the argument degreewise, we may assume that we are in the ungraded situation.
Let $0 \ne x \in G$.
By Hausdorffness and exhaustiveness there exists $n$ with $x \in F^n G \setminus F^{n+1} G$, whence $0 \ne [x] \in \gr^n G$.
It follows that $\alpha(x) \in F^n G'$ and $0 \ne [\alpha(x)] = \alpha([x]) \in \gr^n G'$, and so $\alpha(x) \ne 0$.
\end{proof}

\subsubsection{}
Given filtered $R$-modules $M, M'$, we can put a filtration on $M \otimes M'$ by \[ F^i(M \otimes M') = \sum_{a+b=i} (F^aM)(F^bM') \subset M \otimes M'. \]
\begin{lemma} \label{lemm:filtered-tensor} \NB{derived version is more natural; version for filtered spectra is \cite[Proposition 3.2.1]{lurie2015rotation}}
Let $R=k$ be a field.
Then \[ \gr^\bullet(M \otimes M') \wequi \gr^\bullet(M) \otimes \gr^\bullet(M'). \]
\end{lemma}
\begin{proof}\todo{reference?!}
We may assume that $M, N$ are exhaustively filtered.
If $M=k(i)$ and $M'=k(j)$ (i.e. $F^a M = k$ if $a \le i$ and $F^a M = 0$ else), then $\gr^\bullet(M) = k[i]$ and $M \otimes M' = k(i+j)$; the result in this case follows.
If $M$, $N$ are finite dimensional we can write $M = \bigoplus_{\alpha \in I} k(i_\alpha)$ (e.g. use Corollary \ref{corr:filtered-free-lifting}(1)) and similarly for $M'$.
The result follows since $\otimes$ and $\gr^p$ commute with sums.
In general write $M, N$ as filtered colimits of their finite dimensional subspaces; the result follows since $\otimes$ and $\gr^\bullet$ commute with filtered colimits.
\end{proof}

\subsubsection{}
Given any (graded) ring $R$ and (homogeneous) ideal $J$, we can give $R$ the filtration by powers of $J$, i.e. $F^n R = J^n$.
In particular if $S$ is a (possibly infinite) set of variables (possibly with some assigned degrees) and $A$ is a base ring, we can consider the polynomial ring $R=A[S]$ or the exterior algebra $R=\Lambda_A[S]$ (which is graded if the variables are), and filter it by powers of the \emph{augmentation ideal} $\ker(R \to A)$.\NB{If $\chara(A) \ne 2$ then $\Lambda_A[S]$ is not commutative...}
In this case $\gr^1(R)$ is called the \emph{indecomposable quotient}.
\begin{lemma} \label{lemm:indecomposables}
\begin{enumerate}
\item Let $S, A$ as above.
  The natural maps $A \to \gr^0$ and $(\gr^1)^{\otimes n} \to \gr^n$ induce canonical isomorphisms \[ \gr^n(A[S]) \wequi \Sym^n_A(\gr^1(S)) \quad\text{and}\quad \gr^n(\Lambda_A[S]) \wequi \Lambda^n_A(\gr^1(S)). \]
  Here $\gr^1(A[S])$ is a free $A$-module on a basis in bijection with $S$.

\item Let $\alpha: R_1 \to R_2$ be a morphism of graded $A$-algebras, where $R_i$ are either both polynomial or both exterior algebras, on generators in positive degrees.
  If $\gr^1(\alpha): \gr^1(R_1) \to \gr^1(R_2)$ is surjective (respectively split injective, e.g. injective and $A$ a field, respectively an isomorphism) then so is $\alpha$.
\end{enumerate}
\end{lemma}
\begin{proof}
(1) Let $R = A[S]$ or $R=\Lambda_A[S]$, respectively.
We can give $R$ the grading where all variables have degree $1$; then $J^n = R_{\ge n}$.
This implies that $\gr^n(R) \wequi R_n$.
All claims are checked easily.

(2) If $\beta: M \to N$ is a surjective (respectively split injective, respectively bijective) morphism of $A$-modules, then so are $\Sym^n(\beta)$ and $\Lambda^n(\beta)$\NB{Injectivity needs some assumptions. For $\Lambda^n$, actually injections of free modules are preserved (embed into tensor power). What about symmetric case?}.
Hence under our assumption $\gr^\bullet(\alpha)$ is surjective (respectively injective, respectively an isomorphism) by (1).
The filtrations are complete, exhaustive and Hausdorff for degree reasons.
The claim thus follows from Lemmas \ref{lemm:gr-injective} and \ref{lemm:filtered-surjection}.
\end{proof}

\subsection{Derived completion of abelian groups} \label{subsec:derived-completion}
For an abelian group $A$, we write \[ L_p^\comp A = \lim_n \cof(A \xrightarrow{p^n} A) \in D(\Ab) \] for the derived $p$-completion.
Then there is a short exact sequence \[ 0 \to \limone A[p^n] \to \pi_0 L_p^\wedge A \to A_p^\wedge := \lim_n A/p^n \to 0. \]
Moreover\NB{also $\pi_0 L_p^\wedge A \wequi \Ext(\Z/p^\infty, A)$ and $\pi_1 L_p^\comp A \wequi \Hom(\Z/p^\infty, A)$} \[ \pi_1 L_p^\comp A \wequi \lim A[p^n] \text{ and } \pi_i L_p^\comp A = 0 \text{ for } i \ne 0,1. \]
See \cite[Tag 0BKG]{stacks-project}.

\begin{lemma} \label{lemm:derived-p-comp}
If the $p$-torsion in $A$ is of bounded order, then $A_p^\comp \wequi L_p^\comp A$ (so in particular $\pi_1 L_p^\comp A = 0$).
\end{lemma}
\begin{proof}
We need to prove that $\lim A[p^n] = 0 = \limone A[p^n]$.
By assumption the sequence of sets $A[p^n]$ is eventually stationary, and a sufficiently high composite of the transition maps is zero (since the transition maps are given by multiplication by $p$).
The result follows since (derived) limits can be computed by restriction to final subcategories of the indexing category.
\end{proof}

\subsection{Witt groups} \label{subsec:witt-groups}
For a ring $A$, we denote by $\GW(A)$ (respectively $\W(A)$) its Grothendieck--Witt ring  (respectively its Witt ring) \cite[\S\S I.4, I.5]{knebusch-bilinear}.
We write $\I(A) \subset \W(A)$ for the fundamental ideal, i.e. the kernel of the rank homomorphism \cite[\S I.7]{knebusch-bilinear}
\begin{lemma} \label{lemm:witt-injection}
Let $A$ be a Dedekind domain with $Pic(A) = 0$.
Write $K$ for the fraction field of $A$.
Then in the following commutative diagram (in which $\GW(\ph) \to \Z$ is the rank map), all maps are injective
\begin{equation*}
\begin{CD}
\GW(A) @>>> \GW(K) \\
@VVV        @VVV   \\
\W(A) \times \Z @>>> \W(K) \times \Z.
\end{CD}
\end{equation*}
\end{lemma}
\begin{proof}
Since $\W(A) \hookrightarrow \W(K)$ \cite[Corollary IV.3.3]{milnor1973symmetric}, it suffices to show that $\GW(A) \hookrightarrow \W(A) \times \Z$.
By \cite[\S I.5, Proposition]{knebusch-bilinear} the kernel of $\GW(A) \to \W(A)$ is generated by ``metabolic spaces'' of the form $H(V)$ for $V$ a vector bundle on $A$.
By \cite[\S I.4, Proposition 2]{knebusch-bilinear} the map $H: \Vect(A) \to \GW(A)$ factors through $K_0(A)$, and so it suffices to show that $K_0(A) \xrightarrow{2\rk} \Z$ is injective.
But $K_0(A) \wequi \Z \oplus Pic(A)$ \cite[Corollary 1.11]{milnor1971introduction}, whence the claim.
\end{proof}

\begin{lemma} \label{lemm:witt-torsion}
Let $A$ be a Dedekind domain with $\vcd_2(Frac(A)) < \infty$ and $Pic(A) = 0$.
Then all torsion in $\GW(A)$ and $\W(A)$ is $2$-primary, and of bounded order.\NB{Bounded by $2^{\vcd_2(k)}$.}
\end{lemma}
\begin{proof}
It follows from Lemma \ref{lemm:witt-injection} that it suffices to establish the claims about $\W(k)$, where $k$ is a field.
For this case see e.g. \cite[Lemma 29]{bachmann-gwtimes}.
\end{proof}

\begin{corollary} \label{lemm:GW-inj}
Assumptions as in Lemma \ref{lemm:witt-torsion}.
The map $\GW(A) \to \GW(A)_2^\comp$ is injective, and similarly for $\W(A)$.
\end{corollary}
\begin{proof}
For any abelian group $G$, elements in the kernel of $G \to G_2^\comp$ must be in the kernel of $G \to G/2^n$ for all $n$, and hence be infinitely $2$-divisible.
We show the only such element in $\GW(A), \W(A)$ is $0$.
Since $\Z$ has no infinitely $2$-divisible elements (other than $0$), it suffices to prove the claim about $\W(A)$.
As above we may assume that $A=k$ is a field.
Since $2 \in \I(k)$, any infinitely $2$-divisible element lies in $\cap_n \I(k)^n$.
This group is zero \cite[Theorem III.5.1]{milnor1973symmetric}, whence the result.
\end{proof}

\begin{lemma} \label{lemm:witt-completion}
Let $\vcd_2(k) < \infty$.
Then the $\I$-adic and $2$-adic filtrations on $\W(k)$ induce the same topology.
In particular $\W(k)_2^\comp \wequi \W(k)_\I^\comp$.
\end{lemma}
\begin{proof}
We have $2\W(k) \subset \I(k)$ and $\I(k)^{\vcd_2(k)+1} \subset 2\W(k)$ \cite[last Theorem]{elman-lum-2cohom}, which implies the claim.
\end{proof}

\begin{lemma} \label{lemm:witt-units}
Let $k$ be a field.
Then $\W(k)_{(2)}$ is a local ring with maximal ideal $\I$, and similarly $\GW(k)_{(2)}$ is a local ring (with maximal ideal $\I+(2)$).
In particular $x \in \W(k)_{(2)}$ is a unit if and only if $\rk(x) \ne 0$.
\end{lemma}
\begin{proof}
The claim about $\GW(k)_{(2)}$ follows by inspection of the prime spectrum, depicted in \cite[Figure 1]{balmer2010spectra} (observe that this reference denotes by $\I$ what we denote by $\I+(2)$).
This implies the claim about $\W(k) = \GW(k)/1+\lra{-1}$.

\begin{comment}
There is a certain (possibly empty, possibly infinite) set $\Omega$ and a homomorphism $\sigma: \W(k) \to \Z^\Omega$ inducing $\Spec(\W(k)) \wequi (\Omega \times \Spec(\Z))/\Omega \times \{(2)\}$ \cite[Remark after Lemma III.3.5]{milnor1973symmetric}.
Since $\Spec(\W(k)_{(2)})$ identifies with the subspace of prime ideals not containing an odd integer it is given by $(\Omega \times \Spec(\Z_{(2)}))/\Omega \times \{(2)\}$.
This has $[\Omega \times \{(2)\}]$ as unique maximal element, whence the result for $\W(k)$.
For $\GW(k)$ we use the pullback $\GW(k) \wequi \W(k) \times_{\Z/2} \Z$.
One deduces easily that an element of $\GW(k)_{(2)}$ is a unit if and only if its images in $\W(k)_{(2)}$ and $\Z_{(2)}$ are; in particular the set of non-units in an ideal and so $\GW(k)_{(2)}$ is local.\NB{More generally: local rings and local homs are closed under limits.}
\end{comment}
\end{proof}

\subsection{The homotopy $t$-structure} \label{subsec:t-structure}
\subsubsection{}
The category $\SH(k)$ admits a $t$-structure called the \emph{homotopy $t$-structure} \cite[\S5.2]{morel-trieste}.
Its non-negative part generated (in the sense of e.g. \cite[Proposition 1.4.4.11]{lurie-ha}) by spectra of the form \[ \Sigma^\infty_+ X \wedge \Gmp{n}, \] for $X \in \Sm_k$ and $n \in \Z$.
We denote by \[ E \mapsto E_{\ge 0} \quad\text{and}\quad E \mapsto E_{\le 0} \] the truncation functors.

\subsubsection{}
For $E \in \SH(k)$, denote by $\ul{\pi}_i(E)_j$ the Nisnevich sheaf on $\Sm_k$ associated with the presheaf \[ X \mapsto [\Sigma^i \Sigma^\infty_+ X, E \wedge \Gmp{j}]. \]
One may prove that \cite[Theorem 2.3]{hoyois-algebraic-cobordism} \[ \SH(k)_{\ge 0} = \{E \in \SH(k) \mid \ul{\pi}_i(E)_j = 0 \text{ for all } i < 0, j \in \Z \}, \] and similarly \[ \SH(k)_{\le 0} = \{E \in \SH(k) \mid \ul{\pi}_i(E)_j = 0 \text{ for all } i > 0, j \in \Z \}. \]

\subsubsection{}
The homotopy $t$-structure is non-degenerate \cite[Corollary 2.4, Remark 2.5]{hoyois-algebraic-cobordism}, i.e. \[ \cap_i \SH(k)_{\ge i} = 0 = \cap_i \SH(k)_{\le i}. \]

\subsubsection{}
Each $\ul{\pi}_i(E)_j =: F$ is an \emph{unramified sheaf} \cite[Lemma 6.4.4]{morel2005stable}.
This means in particular that $F = 0$ if and only if for every finitely generated separable field extension $K/k$, which we may view as the fraction field of a smooth $k$-variety, the generic stalk $F(K)$ is zero.

\subsubsection{}
The heart $\SH(k)^\heart$ can be identified with the category of \emph{homotopy modules} \cite[Theorem 5.2.6]{morel-trieste}.
Namely for $E \in \SH(k)^\heart$ there are canonical isomorphism \[ \ul{\pi}_0(E)_{j-1} \wequi (\ul{\pi}_0(E)_{j})_{-1}, \] where on the right hand side the $(\ph)_{-1}$ means Voevodsky's \emph{contraction operation} \cite[Definition 4.3.10]{morel-trieste}.
Moreover, any sequence of Nisnevich sheaves $F_i$ together with isomorphisms $F_{i-1} \wequi (F_i)_{-1}$ such that the cohomology of each $F_i$ is homotopy invariant (i.e. $F_i$ is ``strictly homotopy invariant'') gives rise to an object of $\SH(k)^\heart$, and this is an equivalence of categories.

\subsubsection{} \label{subsub:pi0-sphere}
It turns out that \cite[Theorem 6.2.1]{morel2004motivic-pi0} \[ \ul{\pi}_0(\1)_* \wequi \ul{K}_*^{MW}. \]
This implies that \cite[Lemma 3.10]{A1-alg-top} \[ \ul{\pi}_0(\1[\eta^{-1}])_* \wequi \ul{W}[\eta^{\pm 1}]. \]
Here $\ul{K}_*^{MW}$ is the homotopy module of \emph{Milnor-Witt $K$-theory} \cite[\S3.2]{A1-alg-top} and $\ul{W}$ is the sheaf of unramified Witt rings, i.e. the Zariski sheafification of $X \mapsto \W(X)$.

\subsubsection{}
Another common indexing convention is \[ \ul{\pi}_{i,j}(E) = \ul{\pi}_{i-j}(E)_{-j}. \]
We also use the common abbreviations \[ \pi_i(E)_j = \ul{\pi}_i(E)_j(k) \quad\text{and}\quad \pi_{i,j}(E) = \ul{\pi}_{i,j}(E)(k). \]
The definition of $\pi_{i,j}$ extends to general base schemes: for $E \in \SH(S)$ we put \[ \pi_{i,j}(E) = [S^{i,j}, E]_{\SH(S)}. \]
We will not use $\ul{\pi}_{i,j}$ unless the base $S$ is the spectrum of a field.

\subsection{Completion and localization} \label{subsec:completion}
\subsubsection{}
For any set of numbers $\scr S$ (or more generally $\scr S \subset \pi_{**}(\1)$) and $E \in \SH(S)$ there exists an initial $\scr S$-periodic spectrum $E'$ under $E$.
In other words for every $x \in \scr S$ the endomorphism of $E'$ induced by $x$ is an equivalence.
For the sets $\scr S = \{x\}$ and (given a prime $p$) $\scr S = \{n \mid (n,p) = 1\}$ we respectively denote $E'$ by $E[1/x]$ and $E_{(p)}$.
These localizations may be computed as the mapping telescope of appropriate endomorphisms, i.e. ``in the expected way''.
Analogous results for $\SH$ are even more well-known, and their standard proofs adapt immediately to $\SH(S)$; see e.g. \cite[\S12.1]{bachmann-norms} for details.
We write $E \otimes \Q$ for the result of inverting all primes (i.e. $\scr S = \Z \setminus 0$).

We write \[ \SH(S)[x^{-1}] \subset \SH(S) \] for the category of $x$-periodic spectra.

\subsubsection{}
Given a prime $p$ and $E \in \SH(S)$, there exists an initial $p$-complete spectrum $E_p^\comp$ under $E$; in other words whenever $F$ is $p$-periodic (i.e. $F \wequi F[1/p]$) then $[F, E_p^\comp] = 0$.
Moreover \[ E_p^\comp \wequi \lim_n E/p^n, \] where $E/p^n$ denotes the cofiber of the endomorphism of multiplication by $p^n$.
Equivalently, $E \mapsto E_p^\comp$ is the localization at those maps $\alpha$ such that $\alpha/p$ is an equivalence.
Again the standard proofs of these claims for $\SH$ in place of $\SH(S)$ immediately generalize; see e.g. \cite[\S2.1]{bachmann-SHet} for details.
A similar discussion holds for other graded endomorphisms of $\1$; e.g. \[ E_\eta^\comp \wequi \lim_n E/\eta^n. \]

\subsubsection{}
We record some basic facts about these constructions.
Recall the derived $p$-completion of abelian groups $L_p^\comp A \in \SH$ from \S\ref{subsec:derived-completion}.

\begin{lemma} \label{lemm:completion-homotopy}
\begin{enumerate}
\item Let $E \in \SH(k)$. Then $\ul{\pi}_i(E_{(p)})_j \wequi (\ul{\pi}_i(E)_j)_{(p)}$.
  Similarly if $E \in \SH(S)$ then $\pi_{i,j}(E_{(p)}) \wequi \pi_{i,j}(E)_{(p)}$.
\item Let $E \in \SH(S)$. There is a split short exact sequence \[ 0 \to \pi_0 L_p^\wedge \pi_i(E)_j \to \pi_i(E_p^\comp)_j \to \pi_1 L_p^\wedge \pi_{i-1}(E)_j \to 0. \]
  The map $\pi_i(E)_j \to \pi_i(E_p^\comp)_j$ factors as \[ \pi_i(E)_j \to \pi_0 L_p^\wedge \pi_i(E)_j \to \pi_i(E_p^\comp)_j, \] all maps being the canonical ones.
\end{enumerate}
\end{lemma}
Beware the absence of underlines in (2)!
\begin{proof}
(1) Immediate.

(2) Since $\map(\Sigma^{i,j} \1, \ph)$ preserves $p$-completions, this follows from the analogous statement for ordinary spectra (see e.g. \cite[Proposition VI.5.1]{bousfield1987homotopy}).
For last assertion, apply the result to $E_{\ge i} \to E$.
\end{proof}

\begin{corollary} \label{corr:completion-inj}
Let $E \in \SH(S)$.
If $\pi_{i,j}(E) \to \pi_{i,j}(E)_p^\comp$ is injective then so is $\pi_{i,j}(E) \to \pi_{i,j}(E_p^\comp)$.
\end{corollary}
\begin{proof}
By Lemma \ref{lemm:completion-homotopy} the map $\pi_{i,j}(E) \to \pi_{i,j}(E_p^\comp)$ factors as $\pi_{i,j}(E) \xrightarrow{\alpha} \pi_0 L_p^\comp \pi_{i,j}(E) \hookrightarrow \pi_{i,j}(E_p^\comp)$ so it suffices to show that $\alpha$ is injective.
As reviewed in \S\ref{subsec:derived-completion}, the injection $\pi_{i,j}(E) \hookrightarrow \pi_{i,j}(E)_p^\comp$ factors as $\pi_{i,j}(E) \xrightarrow{\alpha} \pi_0 L_p^\comp \pi_{i,j}(E) \to \pi_{i,j}(E)_p^\comp$.
The result follows.
\end{proof}

\subsubsection{}
We often employ localization and completion together, using the following well-known result.
\begin{lemma} \label{lemm:inv-compl}
Let $\scr C$ be a presentable stable $\infty$-category, $E \in \scr C$ and $p$ an integer.
Then $E \wequi 0$ if and only if $E[1/p] \wequi 0$ and $E_p^\comp \wequi 0$ (equivalently, $E/p \wequi 0$; a fortiori this holds if $E_{(p)} \wequi 0$).

Similarly if $\scr C$ is presentably symmetric monoidal, $L \in \scr C$ is invertible and $x \in \pi_0(L)$ then $E \wequi 0$ if and only if $0 \wequi E_x^\comp$ and $0 \wequi E[1/x]$.
\end{lemma}
For $\scr C = \SH$ this result goes back to at least \cite[Lemma 1.34]{ravenel1984localization}, and essentially the same proof works in general.
We include it for the convenience of the reader.
\begin{proof}
We give the proof of the second statement, the proof of the first one is obtained by replacing $x$ by $p$.
Note that by the definition of $E_x^\comp$ as a localization, we have $E_x^\comp \wequi 0$ if and only if $E/x \wequi 0$, i.e. $x: E \to E \wedge L$ is an equivalence, i.e. $E \wequi E[1/x]$.
We thus need to show that $E \wequi 0$ if and only if ($E \wequi E[1/x]$ and $E[1/x] \wequi 0$), which is clear.
\end{proof}

\subsection{Very effective spectra} \label{subsec:very-effective}
We write \[ \SH(k)^\veff \subset \SH(k) \] for the subcategory generated under colimits (equivalently, colimits and extensions \cite[Remark after Proposition 4]{bachmann-very-effective}) by spectra of the form $\Sigma^\infty_+ X$, with $X \in \Sm_k$.
This is the non-negative part of a $t$-structure \cite[Proposition 1.4.4.11]{lurie-ha}.
If $E \in \SH(k)^\eff$, then $E \in \SH(k)^\eff_{\ge 0} = \SH(k)^\veff$ (respectively $E \in \SH(k)^\eff_{\le 0}$) if and only if $\ul{\pi}_i(E)_0 = 0$ for $i < 0$ (respectively $i > 0$) \cite[\S3]{bachmann-very-effective}.

\subsection{Inverting \texorpdfstring{$\rho$}{rho}} \label{subsec:rho}
\subsubsection{}
We denote by \[ \rho: \1 \to \Gm \] the map corresponding to $-1 \in \scr O^\times$.
Then there is a canonical equivalence \cite{bachmann-real-etale} \[ \SH(S)[\rho^{-1}] \wequi \SH(RS), \] where $RS$ denotes the \emph{real space associated with $S$} \cite[(0.4.2)]{real-and-etale-cohomology} and $\SH(RS)$ means sheaves of spectra on the topological space $RS$.
One puts $\Sper(A) := R\Spec(A)$.
\begin{corollary} \label{cor:rho-cons-1}
Let $\{k_\alpha/k\}$ be the set of real closures of $k$.
Then the canonical functor \[ \SH(k)[\rho^{-1}] \to \prod_\alpha \SH(k_\alpha)[\rho^{-1}] \wequi \prod_\alpha \SH \] is conservative.
\end{corollary}
\begin{proof}
Immediate from knowing the stalks of the small real étale site of $k$ \cite[Proposition 3.7.2]{real-and-etale-cohomology}.
The last equivalence comes from $\Sper(k) \wequi \{*\}$ for $k$ real closed.
\end{proof}

Given topological spaces $X, Y$, write $C(X, Y)$ for the set of continuous maps from $X$ to $Y$.
\begin{corollary} \label{cor:rho-periodic-stems}
For a local ring $k$, we have \[ \pi_* \1_k[\rho^{-1}] \wequi C(\Sper(k), \pi_*^s) \wequi C(\Sper(k), \Z) \otimes \pi_*^s. \]
Here $\Z$ and $\pi_*^s$ are given the discrete topologies.
\end{corollary}
\begin{proof} 
We need to compute $[\Sigma^* \1, \1]_{\SH(\Sper(k))}$.
We can do this using the descent spectral sequence\NB{ref?} \[ H^p_\ret(k, \pi_q^s) \Rightarrow [\Sigma^{q-p}\1, \1]_{\SH(\Sper(k))}. \]
Since the real étale cohomological dimension of local rings is $0$ \cite[Theorem 7.6]{real-and-etale-cohomology}, the spectral sequence collapses and yields \[ [\Sigma^* \1, \1]_{\SH(\Sper(k))} \wequi H^0_\ret(k, \pi_*^s) \wequi C(\Sper(k), \pi_*^s). \]
This proves the first equivalence.
For the second, it suffices to prove that if $X$ is a topological space and $A$ is a finitely generated abelian group, then $C(X, A) \wequi C(X, \Z) \otimes A$.\NB{ref?}
For this it is enough to show that the functor $C(X, \ph)$ is exact (on the category of discrete abelian groups).
Since finite sums of abelian groups are products, their preservation is clear.
Given an exact sequence \[0 \to A \to B \to C \to 0 \] of abelian groups, exactness of \[ 0 \to C(X, A) \to C(X, B) \to C(X, C) \to 0 \] is clear at all but the last place; i.e. we need to show that $C(X, \ph)$ preserves surjections.
This is true since if $\alpha: A \to B$ is a surjection of abelian groups, then there exists a set-theoretic (and so continuous) section $s: B \to A$, and then $C(X, s)$ is a set-theoretic section of $C(X, \alpha)$, whence the latter is surjective.
\end{proof}

\subsubsection{} \label{subsec:r-R}
If $RS = \{*\}$ (such as if $S = Spec(R)$, with $R=\Z, \Z[1/2], \Q, \R$, and so on) then $\SH(S)[\rho^{-1}] \wequi \SH$.
In this situation we denote by $r_\R$ the composite \[ r_\R: \SH(S) \to \SH(S)[\rho^{-1}] \wequi \SH(RS) \wequi \SH. \]
By construction, the equivalence is given by taking global sections.
If $S = \Spec(\R)$, then $r_\R$ is equivalent to the functor induced from the one sending a smooth variety $X/\R$ to its topological space of real points \cite[\S10]{bachmann-real-etale}.

\subsubsection{} \label{subsub:plus-minus}
The above will arise for us mainly as follows.
For $E \in \SH(S)$ there is a functorial splitting \[ E[1/2] \wequi E[1/2, 1/\eta] \vee E[1/2]_\eta^\comp =: E[1/2]^- \vee E[1/2]^+. \]
If $S=\Spec(k)$, then an equivalent formulation is that the ring $\GW(k)[1/2] \wequi [\1[1/2], \1[1/2]]_{\SH(k)}$ splits as \[ \GW(k)[1/2] \wequi \W(k)[1/2] \times \Z[1/2]. \]
We obtain a splitting at the level of categories \[ \SH(S)[1/2] \wequi \SH(S)[1/2]^- \times \SH(S)[1/2]^+. \]
For any $E \in \SH(S)$, $\eta$ acts as an isomorphism on $E[1/2]^-$ and as zero on $E[1/2]^+$.
On the other hand \begin{equation}\label{eq:rho-eta-periodic} \eta\rho = -2 \text{ on } \SH(S)[\eta^{-1}], \end{equation} and $\rho$ is nilpotent on $E[1/2]^+$.
See \cite[Lemma 39]{bachmann-real-etale} for all of this.
We deduce the following.
\begin{corollary} \label{cor:rR-eta-2-per}
Let $k$ be uniquely orderable.
Then \[ \SH(k)[1/2, 1/\eta] \xrightarrow{r_\R} \SH[1/2] \] is an equivalence.

For any field $k$, with set of real closures $\{k_\alpha\}$, the functor \[ \SH(k)[1/2,1/\eta] \xrightarrow{(r_\R^\alpha)_\alpha} \prod_\alpha \SH[1/2] \] is conservative.
\end{corollary}
\begin{proof}
The first claim is clear from the above discussion.
The second is an immediate consequence of Corollary \ref{cor:rho-cons-1}.
\end{proof}

One easily deduces the motivic Serre finiteness theorem \cite{levine2015witt} \begin{equation} \label{eq:serre-finiteness} \ul{\pi}_*(\1[\eta^{-1}]) \otimes \Q \wequi \ul{W} \otimes \Q. \end{equation}

\begin{remark} \label{rmk:eta-realization}
Since $\eta$ is the geometric Hopf map, $r_\R(\eta)$ comes from the Hopf map $S^1 \to \P^1(\R) \wequi S^1$.
This is just the squaring map, and so $r_\R(\eta) = 2$.
This observation is closely related to \eqref{eq:rho-eta-periodic}.
\end{remark}

\section{Adams operations for the motivic spectrum $\KO$} \label{sec:adams}
Throughout this section we will work with base schemes $S$ such that $1/2 \in S$.

\subsection{Summary}
See \S\ref{subsec:KO-KGL} for the definition of the motivic spectra $\KO, \KGL$ and the Bott element $\beta$ (and the table of notation in \S\ref{subsec:table-of-notation} for the meaning of $n_\epsilon$).
\begin{theorem} \label{thm:ko-adams}
Let $n$ be odd.
For every scheme with $1/2 \in S$ we construct a map $\adamspsi^n: \KO_S[1/n] \to \KO_S[1/n] \in \SH(S)$.
These maps satisfy the following properties.
\begin{enumerate}
\item They are compatible with base change (up to homotopy).
\item We have $\adamspsi^n(\beta) = n^2 \cdot n^2_\epsilon \cdot \beta$.
\item $\adamspsi^n$ is a morphism of $\scr E_\infty$-ring spectra.
\item The following diagram commutes (up to homotopy)
\begin{equation*}
\begin{CD}
\KO[1/n] @>>> \KGL[1/n] \\
@V{\adamspsi^n}VV @V{\adamspsi^n}VV \\
\KO[1/n] @>>> \KGL[1/n],
\end{CD}
\end{equation*}
  where $\adamspsi^n: \KGL[1/n] \to \KGL[1/n]$ is the usual Adams operation; see e.g. \cite[Definition 5.3.2 and sentences thereafter]{riou2010algebraic}.
\end{enumerate}
\end{theorem}

\begin{remark}
Shortly after this article was written, Fasel--Haution supplied a much simpler construction of operations $\adamspsi^n_\FH: \KO[1/n] \to \KO[1/n]$ \cite{FH-adams}.
These operations satisfy (1) and (4) by construction, and (2) is an immediate consequence of the construction (see Lemma \ref{lemm:adamspsi-FH-action} for details).
Their operations are only constructed as homotopy ring maps \cite[Theorem 5.2.4]{FH-adams} rather than $\scr E_\infty$-ring maps, so they do not (on the nose) ``satisfy'' (3).
However in the sequel we never use that our operations are $\scr E_\infty$ (rather than just homotopy ring maps), so the article \cite{FH-adams} can be used as a drop-in replacement for the entirety of this section.
\end{remark}

\begin{remark}
For a scheme $X$, exterior powers of vector bundles induce a special $\lambda$-ring structure on $\GW(X)$ \cite{zibrowius2018gamma,FH-adams}.
As in any special lambda ring, there are thus induced Adams operations.
From our construction, it is unclear if these ``geometric'' operations coincide with the ones induced by the spectrum maps $\adamspsi^n$.
On the other hand the spectrum maps $\adamspsi^n_\FH$ are by construction closely related to the geometric operations (see Lemma \ref{lemm:adamspsi-FH-action}).
\end{remark}

\begin{remark}
We show in \S\ref{subsec:geometric-ops} that $\adamspsi^n \wequi \adamspsi^n_\FH$ (see Proposition \ref{prop:FH-comparison}).
One can view this as either identifying $\adamspsi^n$ with the geometric operation, or lifting $\adamspsi^n_\FH$ to an $\scr E_\infty$-operation.
\end{remark}

\begin{remark} \label{rmk:obstruction-even}
One may show that for any $n$ (even or odd), $\adamspsi^n(\beta) = n^2 \cdot n^2_\epsilon \cdot \beta$, where by $\adamspsi^n$ we mean the ``geometric'' operation mentioned above.
This implies that any spectrum map realizing this geometric operation must invert $n^2$ and $n^2_\epsilon$ (since $\beta$ is a unit).
If $n$ is odd, then $n^2_\epsilon$ maps to a unit in $\Z[1/n]$ (namely $n^2$) and also in $\W(k)[1/n]$ (namely $1$), whence is a unit in $\GW(k)[1/n]$\NB{Use the pullback $\GW(k) \wequi \W(k) \times_{\F_2} \Z$.} and so $\KO[1/(n^2\cdot n^2_\epsilon)] \wequi \KO[1/n]$.
One the other hand if $n$ is even then the image of $n^2_\epsilon$ in $\W(k)$ is $0$ and so \[ \KO[1/(n^2 \cdot n^2_\epsilon)] \wequi \KO[1/2]^+[1/(n^2\cdot n^2_\epsilon)] \vee \KO[1/2]^-[1/(n^2\cdot n^2_\epsilon)] \wequi \KO[1/n]^+. \]
An Adams operation $\adamspsi^n: \KO[1/n]^+ \to \KO[1/n]^+$ for $n$ even satisfying all expected properties exists but is not very interesting; see Remark \ref{rmk:construction-even}.
\end{remark}

\begin{remark} \label{rmk:adams-kw}
The endofunctors of $\SH(k)$ given by $E \mapsto E_{\ge 0}$ and $E \mapsto E[\eta^{-1}]$ are lax symmetric monoidal (the former being the composite of the strong symmetric monoidal functor $\SH(k)_{\ge 0} \hookrightarrow \SH(k)$ and its hence lax symmetric monoidal right adjoint, and the latter being a smashing localization).
It follows that the functor \[ \SH(k) \to \SH(k), E \mapsto E_{\ge 0}[\eta^{-1}] \] is lax symmetric monoidal.
Applying it to the ring map $\adamspsi^n: \KO[1/n] \to \KO[1/n]$ we obtain \[ \adamspsi^n: \kw[1/n] \to \kw[1/n] \in \CAlg(\SH(k)). \]
\end{remark}

\begin{example} \label{ex:adamspsi-pi0}
We will use many times the following fact: if $E \to F \in \SH(S)$ is any morphism, then the induced map $\pi_{**}(E) \to \pi_{**}(F)$ is a $\pi_{**}(\1)$-module morphism.
For example, suppose that $\KO^0(S)$ is generated by elements of the form $\lra{a}$, for $a \in \scr O(S)^\times$ (e.g. $S$ the spectrum of a field or $\Z[1/d!]$ \cite[Lemma 5.5]{bachmann-euler}).
Then $\pi_{0,0}(\1) \to \pi_{0,0}(\KO)$ is surjective, and hence any ring map $\adamspsi: \KO \to \KO$ acts trivially on $\pi_{0,0}(\KO)$.
More generally, $\adamspsi$ acts trivially on the image of $\pi_{**}(\1) \to \pi_{**}(\KO)$, e.g. on $\eta$.
\end{example}

Here is a sketch of our construction.
We view $\KGL$ as a motivic spectrum with (homotopy coherent) $C_2$-action coming from passage to duals.
Using the motivic Snaith theorem, we construct a $C_2$-equivariant $\scr E_\infty$-endomorphism $\adamspsi^n_\Snaith$ of $\KGL[1/n]$.
By the homotopy fixed point theorem, over sufficiently nice bases, like $\Z[1/2]$, $\KGL^{hC_2}$ is $2$-adically equivalent to $\KO$.
We obtain for $n$ odd an $\scr E_\infty$-endomorphism $\adamspsi^n_\OSnaith$ on $\KO^+ := \KGL^{hC_2}$.
There is a fracture square for $\KO$ involving $\KO^+$, $\KO[1/2]^-$ and $\KO_2^\comp[1/2]^-$.
To build our Adams operation on $\KO[1/n]$ we will choose a compatible operation $\adamspsi^n_-$ on $\KO[1/2]^-$.
We have $\SH(\Z[1/2])[1/2]^- \wequi \SH[1/2]$ (see \S\ref{subsec:rho}), and $\KO[1/2]^-$ corresponds to the topological orthogonal $K$-theory spectrum $\KO^\topsup[1/2]$ under this equivalence.
It thus seems natural to let $\adamspsi^n_-$ be the \emph{topological} Adams operation on $\KO^\topsup[1/2n]$.
Since $\KO_2^\comp[1/2]^- \in \SH(\Z[1/2]) \otimes \Q^- \wequi D(\Q)$, compatibility of $\adamspsi^n_\OSnaith$ and $\adamspsi^n_-$ is purely a question about homotopy groups.
The operation $\adamspsi^n_-$ on $\KO^\topsup[1/2n]$ acts on a generator of $\pi_4$ by multiplication by $n^2$.
We can write $\adamspsi^n_\OSnaith(\beta) = a \beta$, for some $a \in \GW(\Z[1/2])_2^\comp$, and the compatibility is equivalent to requiring that the image of $a$ in $\W(\Z[1/2])_2^\comp[1/2] \wequi \Q_2$ be $n^2$.
Our proof of this ultimately relies on the fact that the space of $\scr E_\infty$-endomorphisms of $\KU^\topsup$ is discrete, i.e. Goerss--Hopkins obstruction theory.

\begin{remark}
This proof seems very unsatisfying to the authors.
We believe that the Grothendieck--Witt space of any scheme should admit $\scr E_\infty$ Adams operations after inverting the relevant integer, functorially in the scheme.
Presumably this should be related to the spectral $\lambda$-ring theory of Barwick--Glasman--Mathew--Nikolaus.
This operation would coincide with Zibrowius' one essentially by construction, and the fact that $\adamspsi^n(\beta) = n^2 \cdot n^2_\epsilon \cdot \beta$ would be a fairly straightforward computation.
Taking suspension spectra and inverting the Bott element (see \S\ref{subsec:KO-KGL}), we would immediately obtain an $\scr E_\infty$ Adams operation on $\KO[1/n]$.

Unfortunately we have been unable to implement this idea, forcing us to perform the contortions sketched above instead.
\end{remark}

Along the way, we also determine the real realization of $\KO$.
Recall the real realization functor $r_\R$ from \S\ref{subsec:r-R}.
Note that the periodicity generator $\beta: S^{8,4} \to \KO$ gives a periodicity of degree $4$ in $r_\R(\KO_S)$, so this spectrum cannot possibly be the topological variant $\KO^\topsup$.
\begin{lemma} \label{lemm:KO-real-realization}
Let $S=\Spec(\R)$.
\begin{enumerate}
\item $r_\R(\KGL) = 0$
\item The map $r_\R(\KO) \to r_\R(\KO_2^\comp)$ identifies with $\KO^\topsup[1/2] \to (\KO^\topsup)_2^\comp[1/2]$.
\end{enumerate}
\end{lemma}

\subsection{Motivic ring spectra $\KO$ and $\KGL$} \label{subsec:KO-KGL}
\subsubsection{} \label{subsub:BC2-triangle}
Recall that $BC_2$ is the ordinary $1$-category with one object $C_2$, and space of endomorphisms given by the group $C_2$ of order $2$.
The category $BC_2^\triangleleft$ is obtained by adding an initial object $*$; in other words $BC_2^\triangleleft$ is equivalent to the opposite of the ordinary $1$-category of $C_2$-orbits (i.e. the subcategory of $\Fin_{C_2}$ on the two objects $C_2$ and $*$).
For a category $\scr C$, $\Fun(BC_2, \scr C)$ is the category of objects in $\scr C$ with a homotopy coherent $C_2$-action.
Given $X \in \Fun(BC_2, \scr C)$ we write $X^{hC_2} \in \scr C$ for its limit.
The functor $\Fun(BC_2^\triangleleft, \scr C) \to \Fun(BC_2, \scr C)$ given by evaluation at $C_2$ admits a (partially defined) right adjoint (the right Kan extension) sending $X \in \Fun(BC_2, \scr C)$ to the diagram $X^{hC_2} \to X$ \cite[Proposition 4.3.2.17 and Definition 4.3.2.2]{lurie-htt}.
It follows that given any object $(X \to Y) \in \Fun(BC_2^\triangleleft, \scr C)$ there is a functorially induced morphism $X \to Y^{hC_2} \in \scr C$ (provided that $Y^{hC_2}$ exists).

\subsubsection{}
For a symmetric monoidal category $\scr C^\otimes$, $\CAlg(\scr C) = \CAlg(\scr C^\otimes)$ denotes the category of $\scr E_\infty$-algebras, and for any category $\scr D$ (symmetric monoidal or not) with finite products $\CMon(\scr D) := \CAlg(\scr D^\times)$.
The group completion or \emph{(direct sum) $K$-theory} functor \[ \CMon(\Grpd) \to \CMon(\Grpd)^\gp \to \Spc_* \] (with the second functor being the forgetful one, using that $\Grpd \wequi \Spc$) is lax symmetric monoidal (e.g. use \cite[Theorem 5.1]{gepner2016universality} and the fact that right adjoints of symmetric monoidal functors are lax symmetric monoidal) and hence induces \[ K^\oplus: \CAlg(\CMon(\Grpd)) \to \CAlg(\Spc_*). \]
The symmetric monoidal structure on $\CMon(\Spc)$ is given by tensor product, and the one on $\Spc_*$ by smash product.

\subsubsection{}
Let $S$ be a scheme.
Write $\Vect(S)$ for the ordinary $1$-groupoid of vector bundles on $S$.
This carries the following structures.
\begin{itemize}
\item For $V, W \in \Vect(S)$, there is the \emph{direct sum} $V \oplus W \in \Vect(S)$; this way $\Vect(S)$ becomes an object in $\CMon(\Grpd)$.
\item For $V, W \in \Vect(S)$, there is their \emph{tensor product} $V \otimes W \in \Vect(S)$.
  This operation distributes over the sum, promoting $\Vect(S)$ to $\CAlg(\CMon(\Grpd))$.
\item For $V \in \Vect(S)$, there is the \emph{dual} $V^\dual = \iHom(V, \scr O_S)$.
  This operation commutes with sum and tensor product, and there is a functorial isomorphism $(V^\dual)^\dual \wequi V$.
  Hence $\Vect(S)$ is promoted to $\Fun(BC_2, \CAlg(\CMon(\Grpd)))$.\footnote{To be precise, the generator of $C_2$ acts via $\Vect(S) \to \Vect(S)$, $V \mapsto V^\dual$, $\alpha: V \xrightarrow{\wequi} W \mapsto (\alpha^\dual)^{-1}: V^\dual \to W^\dual$. We must pass to inverses in order to obtain a functor $\Vect(S) \to \Vect(S)$ instead of $\Vect(S)^\op \to \Vect(S)$.}
\item For $f: X \to Y \in \Sch$, there is a pullback operation $f^*: \Vect(Y) \to \Vect(X)$.
  This is compatible with composition in $f$, and also with all the previous structures.
\end{itemize}
All in all we obtain \[ \Vect: BC_2 \times \Sch^\op \to \CAlg(\CMon(\Grpd)). \]
Applying $K^\oplus$ we get $K^\oplus \Vect: BC_2 \times \Sch^\op \to \CAlg(\Spc_*)$, or equivalently \[ K^\oplus \Vect \in \Fun(BC_2, \CAlg(\PSh(\Sch)_*)). \]
If $X \in \Sch$ is affine, then $K^\oplus \Vect(X)$ is the \emph{$K$-theory space of $X$} \cite[Theorem 7.6]{thomason-trobaugh}, but in general this is not correct.
We can fix this issue by passing to Zariski sheaves: \[ \K := L_\Zar K^\oplus \Vect \in \Fun(BC_2, \CAlg(\PSh(\Sch)_*)). \]
Then for every $X \in \Sch$ the space $\K(X)$ is the (Thomason--Trobaugh) $K$-theory space of $X$ \cite[Theorem 8.1]{thomason-trobaugh}.

\subsubsection{}
Consider the homotopy fixed points $\Vect(S)^{hC_2}$.
This is an ordinary $1$-category, which can be described as follows: an object consists of a vector bundle $V$ together with an isomorphism $\alpha: V \wequi V^\dual$, such that $\alpha^\vee$ corresponds to $\alpha$ under the double dual identification.
In other words, $\Vect(S)^{hC_2}$ is the groupoid of \emph{non-degenerate symmetric bilinear forms} (see e.g. \cite[Definition 2.4]{schlichting2010hermitian}) on $S$, which we also denote by $\Bil(S)$.

Using \S\ref{subsub:BC2-triangle}, we may thus extend $\Vect$ over $BC_2^\triangleleft$ to obtain \[ (\Bil \to \Vect): BC_2^\triangleleft \times \Sch^\op \to \Grpd. \]
Using that the forgetful functor $\CAlg(\CMon(\Grpd)) \to \Grpd$ preserves limits (being a right adjoint), this further upgrades to \[ (\Bil \to \Vect) \in \Fun(BC_2^\triangleleft, \CAlg(\CMon(\Grpd))). \]
Applying as before direct sum $K$-theory and sheafification, we obtain the presheaf \[ \GW := L_\Zar K^\oplus \Bil \in \PSh(\Sch_{\Z[1/2]})_*. \]
We are restricting the base schemes here to $1/2 \in S$ because then this definition of $\GW$ given us the \emph{correct} Grothendieck-Witt space \cite[Remark 4.13]{schlichting2010hermitian} \cite[Corollary 8.5]{schlichting2010mayer}.\NB{Is this true in general?}
All in all we have thus built \[ (\GW \to \K) \in \Fun(BC_2^\triangleleft, \CAlg(\PSh(\Sch_{\Z[1/2]})_*)). \]

\subsubsection{} \label{subsub:KO-cons}
Restricting $\K$ from $\PSh(\Sch)_*$ to $\PSh(\Sm_S)_*$ and motivically localizing, we obtain $L_\mot \K \in \CAlg(\Spc(S)_*)$ and then $\Sigma^\infty L_\mot \K \in \CAlg(\SH(S))$.
We have the \emph{Bott element} $\beta_{\KGL} \in \widetilde{\K}^0(\P^1) \wequi [S^{2,1}, \K]$ and hence obtain $\beta_{\KGL}' \in \pi_{2,1} \Sigma^\infty L_\mot \K$.
By definition we have \[ \KGL_S := (\Sigma^\infty L_\mot \K)[\beta_{\KGL}'^{-1}] \in \CAlg(\SH(S)). \]
See \cite[Proposition 3.2, Lemma 3.3 and Theorem 3.8]{hoyois2016cdh} for details on periodic $\scr E_\infty$-ring spectra; the upshot is that $\KGL_S$ is represented by the prespectrum $(L_\mot \K, L_\mot \K, \dots)$ with bonding maps given by multiplication by $\beta_\KGL$.
It follows from \cite[Example 3.4]{hoyois2016cdh} and \cite[Proposition 6.8 and Theorem 10.8]{thomason-trobaugh} that if $S$ is Noetherian and regular, then the canonical map $\K|_{\Sm_S} \to \Omega^\infty \KGL_S$ is an equivalence; in other words $\KGL$ represents algebraic $K$-theory.

There is a similar Bott element $\beta \in \widetilde{\GW}^0(\HP^1 \wedge \HP^1) \wequi [S^{8,4}, \GW]$ \cite[Definition 5.3, Theorem 5.1]{panin2010motivic}, and inverting it in the suspension spectrum we obtain \[ \KO_S := (\Sigma^\infty L_\mot \K)[\beta'^{-1}] \in \CAlg(\SH(S)). \]
Arguing as above, using \cite[Theorems 9.6, 9.8 and 9.10]{schlichting2016hermitian} we see that if $S$ is Noetherian, regular, and $1/2 \in S$, then the canonical map $\GW|_{\Sm_S} \to \Omega^\infty \KO_S$ is an equivalence; in other words $\KO$ represents hermitian $K$-theory.
In fact (in this situation) \cite[Theorem 9.10]{schlichting2016hermitian} implies that \begin{equation}\label{eq:KO-higher} \Omega^\infty \Sigma^{2n,n} \KO \wequi \GW^{[n]}, \end{equation} where $\GW^{[n]}$ is the presheaf of $n$-shifted Grothendieck--Witt spaces \cite[Definition 9.1]{schlichting2016hermitian}.

One may show that the image $\bar\beta$ of $\beta$ under the map $\GW \to \K$ is $\beta_{\KGL}^4$ (see e.g. \cite[Proposition 3.3]{rondigs2013slices}).
It follows that we could equivalently define $\KGL$ as $(\Sigma^\infty L_\mot \K)[\bar\beta'^{-1}]$, and in particular we obtain a morphism of $\scr E_\infty$-rings $\KO \to \KGL$.
We now make this $C_2$-equivariant.

\subsubsection{}
Consider the functor \[ F: \CAlg(\SH(S)) \to \Cat, E \mapsto 2^{\pi_{**} E}; \] in other words $F(E)$ is the set of subsets of $\pi_{**} E$.
We view this power set as a category (in fact poset) by partially ordering it by inclusion.
\newcommand{\mrk}{\mathrm{mrk}}
The functor $F$ classifies a fibration (by \cite[Theorem 3.2.0.1]{lurie-htt}) \[ \CAlg(\SH(S))_\mrk \to \CAlg(\SH(S)). \]
Thus the objects of $\CAlg(\SH(S))_\mrk$ are pairs $(E, X)$ with $E \in \CAlg(\SH(S))$ and $X \subset \pi_{**}(E)$, and the morphisms $(E, X) \to (E', X')$ are morphisms $E \xrightarrow{\alpha} E'$ such that $\alpha_*(X) \subset X'$.
In particular given a functor $b: \scr C \to \CAlg(\SH(S))$, a lift of $b$ to $\CAlg(\SH(S))_\mrk$ is a section of $F$, or in other words a choice for every object $c \in \scr C$ of $X_c \subset \pi_{**}(b(c))$ subject to the \emph{condition} that for every map $\alpha: c' \to c \in \scr C$ we have $b(\alpha)_*(X_{c'}) \subset X_c$.

The functor $\CAlg(\SH(S)) \to \CAlg(\SH(S))_\mrk, E \mapsto (E, \pi_{**}(E)^\times)$ has a left adjoint which sends $(E, X)$ to the initial $E$-algebra in which all elements of $X$ become invertible.

\subsubsection{}
We have the functor \[ (\Sigma^\infty L_\mot \GW \to \Sigma^\infty L_\mot \K) : BC_2^\triangleleft \to \CAlg(\SH(S)). \]
We lift this to a functor $BC_2^\triangleleft \to \CAlg(\SH(S))_\mrk$ by choosing the sets $\{\beta'\}$ and $\{\bar\beta'\}$ respectively.
Here $\bar\beta'$ is the image $\beta_\KGL'^4$ of $\beta'$ in $\pi_{**} \Sigma^\infty L_\mot \K$ and so fixed by the $C_2$-action.
Composing with the periodization functor $\CAlg(\SH(S))_\mrk \to \CAlg(\SH(S))$ we obtain \[ (\KO_S \to \KGL_S) \in \Fun(BC_2^\triangleleft, \CAlg(\SH(S)). \]

\subsubsection{}
It is straightforward to make this entire construction functorial in $S$ as well.\NB{details?}

\subsubsection{}
It can be shown that the space $L_\mot \K$ (respectively $L_\mot \GW$) is motivically equivalent to the product of $\Z$ and the infinite Grassmannian (respectively the infinite orthogonal Grassmannian) \cite[Proposition 8.1 and Theorem 8.2]{schlichting2015geometric} \cite[Corollary 2.10]{hoyois2016equivariant}.
Since Grassmannians are stable under base change, so are the motivic spaces $L_\mot \K$ and $L_\mot \GW$, and hence so are the spectra $\KGL$ and $\KO$.

\subsubsection{}
Recall the real realization functor $r_\R$ from \S\ref{subsec:r-R}.
We now determine $r_\R(\KO)$ (and $r_\R(\KGL)$).
\begin{proof}[Proof of Lemma \ref{lemm:KO-real-realization}]
(1) Since $\KGL \wequi \Sigma^\infty(\Z \times \Gr)[\beta_\KGL^{-1}]$ we get $r_\R(\KGL) \wequi \Sigma^\infty(\Z \times \Gr(\R))[r_\R(\beta_\KGL)^{-1}]$.
It suffices to show that $r_\R(\beta_\KGL) \in \pi_1(\Z \times \Gr(\R))$ is nilpotent.
Since $\Gr(\R) \wequi BO$, by Bott periodicity \cite[\S4.1]{karoubi2005bott} we have $\pi_3(\Z \times \Gr(\R)) \wequi \pi_2(O) \wequi 0$, whence the result.

(2) The Wood cofiber sequence \cite[Theorem 3.4]{rondigs2013slices} $\Sigma^{1,1}\KO \xrightarrow{\eta} \KO \to \KGL$ together with $r_\R(\eta) = 2$ (Remark \ref{rmk:eta-realization}) and (1) implies that $r_\R(\KO)$ is $2$-periodic.
We deduce that $r_\R(\KO) \wequi r_\R(\KO[1/2]^-) \wequi r_\R(\KW[1/2])$, and this was shown to be $\KO[1/2]$ in \cite[Theorem 4.4]{rondigs2016eta}.
It also follows that $r_\R(\KO_2^\comp)$ is rational, and to conclude it suffices to show that $\pi_* r_\R(\KO_2^\comp) \wequi \Q_2[\beta^{\pm 1}]$.
We have (see Corollary \ref{cor:rR-eta-2-per} for the first step) \[ \pi_* r_\R(\KO_2^\comp) \wequi \pi_*(\KO_2^\comp[1/2]^-) \wequi \pi_*(\KO_2^\comp[1/\eta,1/2]) \wequi \pi_*(\KO_2^\comp[1/(\beta\eta^4),1/2]). \]
For $n < 0$ we have $\pi_n(\KO) \wequi W(\R) \wequi \Z$ if $n \equiv 0 \pmod{4}$, and $=0$ else; in either case multiplication by $\beta\eta^4$ is an isomorphism \cite[Proposition 6.3]{schlichting2016hermitian} \cite[Theorem 1.5.22]{balmer2005witt}.
This implies that $\pi_n(\KO_2^\comp) \wequi \Z_2^\comp$ or $0$, depending on $n$ as before, the same is true in the colimit along $\beta\eta^4$.
The result follows.
\end{proof}

\subsection{The motivic Snaith and homotopy fixed point theorems} \label{subsec:GS-htpy}
\subsubsection{}
The presheaf $\Gm: X \mapsto \scr O(X)^\times$ defines an abelian group object in $\PSh(\Sm_S)_{\le 0}$.
Sending an element to its inverse lifts this to $\Gm \in \Fun(BC_2, \Ab(\PSh(\Sm_S)_{\le 0}))$.
Let $L \subset \Vect$ denote the subgroupoid spanned sectionwise by the trivial line bundle $\scr O$.
This is closed under tensor products and duals, and hence defines a subfunctor \[ L \hookrightarrow \Vect: BC_2 \to \CMon(\PSh(\Sm_S)). \]
By inspection $\Omega L$ is discrete, and identifies with $\Gm$.
Since $L$ is (sectionwise) connected, we thus obtain $B\Gm \wequi L \hookrightarrow \Vect$.
Composing with $\Vect \to \Omega^\infty \KGL$ and using the adjunction \[ \Sigma^\infty_+ \dashv \Omega^\infty: \CMon(\PSh(\Sm_S)) \adj \CAlg(\SH(S)) \] we obtain a map\footnote{Note that if $\scr C \adj \scr D$ is an adjunction then so is $\Fun(\scr A, \scr C) \adj \Fun(\scr A, \scr D)$, e.g. by \cite[Lemmas D.3 and D.6]{bachmann-norms}.} \[ \Sigma^\infty_+ B\Gm \to \KGL \in \Fun(BC_2, \CAlg(\SH(S))). \]
We also have the map \[ \tilde\beta_{\KGL}: S^{2,1} \wequi \Sigma^\infty \P^1 \to \Sigma^\infty \P^\infty \wequi \Sigma^\infty B\Gm \to \Sigma^\infty_+ B\Gm, \] employing the stable splitting $\scr X_+ \wequi \scr X \vee S^0$.
The image of $\tilde\beta_{\KGL}$ in $\pi_{**} \KGL$ is the Bott element $\beta_{\KGL}$ \cite[Proposition 4.2]{gepner2009motivic}.
We may thus lift the map $\Sigma^\infty_+ B\Gm \to \KGL$ to $\Fun(BC_2, \CAlg(\SH(S))_\mrk)$ by choosing the sets $\{\tilde \beta_{\KGL}, \sigma \tilde \beta_{\KGL}\}$ and $\{\beta_{\KGL}, -\beta_{\KGL}\}$ respectively above $\Sigma^\infty_+ B\Gm$ and $\KGL$ (here $\sigma$ denotes the action by $C_2$, so that in particular $\sigma \beta_{\KGL} = -\beta_{\KGL}$).
Inverting the marked elements and noting that $\beta_{\KGL}, -\beta_{\KGL}$ are units in $\KGL$ we obtain the \emph{motivic Snaith map} \[ \Sigma^\infty_+ B\Gm[\tilde\beta_{\KGL}^{-1}, \sigma \tilde\beta_{\KGL}^{-1}] \to \KGL \in \Fun(BC_2, \CAlg(\SH(S))). \]
\begin{lemma}
The motivic Snaith map is an equivalence.
\end{lemma}
\begin{proof}
We have $\Sigma^\infty_+ B\Gm[\tilde\beta_{\KGL}^{-1}] \wequi \KGL$ by \cite[Theorem 4.17]{gepner2009motivic}.
The image of $\sigma \tilde \beta_{\KGL}$ in this ring is a unit (e.g. because it corresponds to $-\beta_{\KGL}$), so the further periodization does nothing.
\end{proof}

\subsubsection{} \label{subsub:GS}
Since $\Gm$ is a discrete abelian group object, we have the endomorphisms $\adamspsi^n: \Gm \to \Gm, x \mapsto x^n$ for all $n \in \Z$, and they all commute.
In particular we obtain \[ \adamspsi^n: \Gm \to \Gm \in \Fun(BC_2, \Ab(\PSh(\Sm_S)_{\le 0})). \]
This deloops to \[ \adamspsi^n: B\Gm \to B\Gm \in \Fun(BC_2, \CMon(\PSh(\Sm_S))). \] 
\begin{lemma} \label{lemm:action-KGL}
The composite \[ S^{2,1} \xrightarrow{\tilde \beta_{\KGL}} \Sigma^\infty_+ B\Gm \xrightarrow{\Sigma^\infty_+ B \adamspsi^n} \Sigma^\infty_+ B\Gm \to \KGL \] is homotopic to $n\beta_{\KGL}$.
\end{lemma}
\begin{proof}
The stable splitting $\P^1 \to \P^1_+$ induces the map \begin{gather*} \K^0(\P^1_+) \wequi \K^0(*)[\scr O(1)]/(\scr O(1)-1)^2 \to \K^0(\P^1) \wequi \K^0(*) \\ a + b\scr O(1) \mapsto b. \end{gather*}
The element \[ \P^1_+ \to B\Gm_+ \xrightarrow{B\adamspsi} B\Gm_+ \to \KGL \in K^0(\P^1_+) \] corresponds to \[ \scr O(1)^{\otimes n} = ([\scr O(1)-1] + 1)^n = 1 + n[\scr O(1)-1] \] which thus maps to $n\beta_{\KGL}$ as desired.
\end{proof}
We can thus form the following composite in $\Fun(BC_2, \CAlg(\SH(S))_\mrk)$ \[ (\Sigma^\infty_+ B\Gm, \{\tilde \beta_{\KGL}, \sigma \tilde \beta_{\KGL}\}) \xrightarrow{\adamspsi^n} (\Sigma^\infty_+ B\Gm, \{\adamspsi^n(\tilde \beta_{\KGL}), \sigma \adamspsi^n(\tilde \beta_{\KGL})\}) \to (\KGL, \{n\beta_{\KGL}, -n\beta_{\KGL}\}). \]
Inverting the marked elements we obtain \[ \KGL \wequi \Sigma^\infty_+ B\Gm[\tilde\beta_{\KGL}^{-1}, \sigma \tilde\beta_{\KGL}^{-1}] \to \KGL[n\beta_{\KGL}^{-1}] \wequi \KGL[1/n]. \]
Further inverting $n$ in the source, we finally arrive at \[ \adamspsi^n_\Snaith: \KGL[1/n] \to \KGL[1/n] \in \Fun(BC_2, \CAlg(\SH(S))). \]

\begin{proposition}
The underlying map $\adamspsi^n_\Snaith: \KGL[1/n] \to \KGL[1/n]$ coincides (up to homotopy) with the map $\adamspsi^n_\Riou$ constructed by Riou \cite[Definition 5.3.2 and sentences thereafter]{riou2010algebraic}.
\end{proposition}
\begin{proof}
Since both operations are stable under base change, we may assume that $S=Spec(\Z)$.
By \cite[Remark 5.2.9]{riou2010algebraic} the map \[ [\KGL, \KGL] \to [\Sigma^\infty_+ \P^\infty, \KGL] \wequi \Z\fpsr{U} \] is injective.
Here $U = \scr O(1) - 1$ is the first Chern class of the tautological bundle.
The image of $\adamspsi^n_\Riou$ is $(1+U)^n$ \cite[Definition 5.3.2]{riou2010algebraic}.
It suffices to verify that the same holds for $\adamspsi^n_\Snaith$.
The image in this case is given by \[ \Sigma^\infty_+ \P^\infty \wequi \Sigma^\infty_+ B\Gm \xrightarrow{B\psi_n} \Sigma^\infty_+ B\Gm \to \KGL. \]
This corresponds to $\scr O(1)^{\otimes n} = (1 + U)^n$ by construction.
\end{proof}

\begin{remark} \label{rmk:adams-classical}
In light of \cite[\S3]{riou2010algebraic}, this means that the $\adamspsi_\Snaith^n$ act on (higher) algebraic $K$-groups in the same way as any of the other standard constructions.
\end{remark}

\subsubsection{}
The map \[ \KO \to \KGL \in \Fun(BC_2^\triangleleft, \CAlg(\SH(S))) \] induces by definition a map \[ \KO \to \KGL^{hC_2}. \]
\begin{proposition} \label{prop:htpy-fixed}
Let $S$ be Noetherian, regular, $1/2 \in S$ and $vcd_2(s) < \infty$ for all $s \in S$.
Then the map $\KO \to \KGL^{hC_2} \in \SH(S)$ is a $2$-adic equivalence.
\end{proposition}
\begin{proof}
It suffices to show that for every (absolutely) affine, smooth $S$-scheme $X$ the morphism \[ \map(\Sigma^{2n,n} X_+, \KO)/2 \to \map(\Sigma^{2n,n} X_+, \KGL^{hC_2})/2 \] is an equivalence.
Since these spectra are periodic, we may assume that $n \ge 0$, and then since $X$ is arbitrary we may assume that $n=0$.\NB{details?}
Since $X$ is a QL-scheme in the sense of \cite[Definition 2.1]{berrick2015homotopy}, the claim follows from \cite[Theorem 2.4 and Corollary 2.6]{berrick2015homotopy}.
To be precise, their definition of (Hermitian) $K$-theory (and the involution on $K$-theory) uses perfect complexes, but the evident functor $\Vect \to \Perf$ is duality preserving (when using the $0$-shifted duality on $\Perf$, which is why we reduced to $n=0$) and induces an equivalence on (Hermitian) $K$-theory spaces, as we have seen in \S\ref{subsub:KO-cons}.
\end{proof}

\begin{definition}
Let $n$ be odd, whence invertible in $\Z_2^\comp$.
We denote by \[ \adamspsi^n_{\OSnaith}: \KO_2^\comp \to \KO_2^\comp \] the map induced from $\adamspsi^n_\Snaith$ by completing at $2$ and taking homotopy fixed points.
(If the base scheme does not satisfy the assumptions of Proposition \ref{prop:htpy-fixed}, pull back from $\Z[1/2]$.)
\end{definition}

\subsection{Action on the Bott element}
For an algebraically closed field $\bar K$ of characteristic zero, the group $\K_2(\bar K)$ is uniquely divisible, and $\K_1(\bar K) = \bar K^\times \wequi \Q/\Z \oplus D$ where $\Q/\Z$ corresponds to the roots of unity and $D$ is uniquely divisible \cite[Theorem VI.1.6]{weibel-k-book}.
This implies (using e.g. Lemma \ref{lemm:completion-homotopy}(2)) that $\pi_2(\KGL(\bar K)_2^\comp) \wequi \Z_2^\comp$.
\begin{lemma} \label{lemm:adams-action-alg-closed}
The action of $\adamspsi^n$ on $\pi_2(\KGL(\bar K)_2^\comp)$ is given by multiplication by $n$.
\end{lemma}
\begin{proof}
We have $\pi_2(\KGL(\bar K)_2^\comp) \wequi \pi_1 L_2^\comp \K_1(\bar K)$, compatibly with the Adams action.
It thus suffices to show that the action on $\K_1(\bar K)$ is multiplication by $n$, i.e. $\adamspsi^n([a]) = [a^n]$.
This is \cite[Example IV.5.4.1]{weibel-k-book}.
\end{proof}

Viewing $\GL(K)$ as a discrete topological group, apply the topological $+$-construction to obtain a space $B\GL(K)^+ \times \Z$ with $\pi_i(B\GL(K)^+ \times \Z) = \K_i(K)$ for $i \ge 0$.
In this model, the $C_2$-action on $\K(K) \wequi B\GL(K)^+ \times \Z$ is induced from the automorphism of $\GL(K)$ given by $A \mapsto A^{-T}$ (and the identity on $\Z$).
If $K = \R$ or $K = \C$, we can give $\GL(K)$ its usual topology instead; denote the result by $\GL(K^\topsup)$.
Functoriality of the $+$-construction yields \[ \K(K) \to B\GL(K^\topsup)^+ \times \Z \in \Fun(BC_2, \SH_{\ge 0}). \]
These maps are in fact $p$-adic equivalences for all $p$ \cite[Corollary 4.7]{suslin1984k}.

\begin{lemma} \label{lemm:adams-technical}
Let $E \in \Fun(BC_2 \times BC_2', \CAlg(\SH))$ denote $\ku_2^\comp$ with its usual action by complex conjugation and passage to dual bundles.
Suppose given a map $\adamspsi: E \to E$ such that the induced map on $\pi_2$ is given by multiplication by the odd integer $n$.
Then the induced endomorphism of \[ ((E^{hC_2})_{\ge 0})^{hC_2'} \] acts by multiplication by $n^{-2}$ on $\pi_{-4}$.
\end{lemma}
\begin{proof}\discuss{Suggestions?}
By Goerss--Hopkins obstruction theory, the space $\Map_{\CAlg(\SH)}(\KU_2^\comp, \KU_2^\comp)$ is discrete and isomorphic to $\Hom_{\CAlg^\heart}(\pi_* \KU_2^\comp, \pi_* \KU_2^\comp) \wequi (\Z_2^\comp)^\times$ via evaluation at the Bott element\NB{$\KU_p^\comp = E(1)$} \cite[Corollary 7.7]{goerss2004moduli}.
$K(1)$-localization and connective cover provide inverse equivalences $\Map_{\CAlg(\SH)}(\KU_2^\comp, \KU_2^\comp) \wequi \Map_{\CAlg(\SH)}(\ku_2^\comp, \ku_2^\comp)$.
It follows that \[ \Map_{\Fun(BC_2 \times BC_2', \CAlg(\SH))}(E, E) \wequi \Map_{\CAlg(\SH)}(E, E)^{hC_2 \times C_2'} \] is also discrete.\NB{Details for the above steps?}
The upshot of all this is that $\adamspsi$ is just given by the ordinary $\adamspsi^n$, made equivariant with respect to $C_2 \times C_2'$ in the usual way.

We have $(E^{hC_2})_{\ge 0} \wequi \ko_2^\comp$\NB{Ref?} and the induced action by $C_2'$ is trivial (since $\ko \wequi BO^+ \times \Z$ and transpose coincides with inverse on orthogonal matrices).
By the above discussion, the map induced by $\adamspsi$ is the usual Adams operation $\adamspsi^n$ on $\ko_2^\comp$, made compatible with the trivial action in the canonical (trivial) way.
We thus need to show that the Adams action on $\pi_{-4} (\ko_2^\comp)^{hC_2'} \wequi (\ko_2^\comp)^4 \RP^\infty_+$ is by multiplication by $n^{-2}$.

We first consider the $\KO$-cohomology.
Let us write $x \in \ko_4$ and $\beta \in \ko_8$ for the generators, so that $x^2 = 4\beta$.
The groups $(\KO_2^\comp)^{4*}(\RP^\infty_+)$ can be read off from \cite[Theorem 1]{fujii1967ko} as follows\NB{In other words $\KO^*(\RP^\infty) \wequi \KO^*(*) \oplus \widetilde{\KO}^*(\RP^\infty)$, and the latter $\Z_2^\comp$-module is free on a generator $\lambda$ if $*=0$, free on $x\lambda$ if $*=-4$, and determined by periodicity for general $* \equiv 0 \pmod{4}$.} \[ (\KO_2^\comp)^{4*} \RP^\infty_+ \wequi \Z_2^\comp[\lambda, x, \beta, \beta^{-1}]/(\lambda^2+2\lambda, x^2-4\beta); \] the $(\KO_2^\comp)^* \wequi \Z_2^\comp[x, \beta, \beta^{-1}]/(x^2-4\beta)$-algebra structure (coming from pullback along $\RP^\infty \to *$) is the evident one.
Here $|\lambda|=0$ and $\lambda$ corresponds to $\scr O(1)-1$.
The algebra structure is compatible with Adams operations (here we are crucially using that $\adamspsi$ is compatible with the trivial action in the trivial way\NB{I.e. this is why we needed G-H obstruction stuff...}).
Note that $\adamspsi(\lambda) = \lambda$.
For this it suffices to show that $\adamspsi(\scr O(1)) = \scr O(1)$; but $\adamspsi(\scr O(1)) = \scr O(n)$ \cite[Theorem 5.1(iii)]{adams1962vector} and $\scr O(2) \wequi \scr O$ (e.g. since $\scr O(2) = (1+\lambda)^2 = 1$ since $\lambda^2=-2\lambda$), so this holds since $n$ is odd.
It follows that the Adams action on $(\KO_2^\comp)^{4*}(\RP^\infty_+)$ is via multiplication by $n^{-2*}$, since this holds for the generators $1$ ($\adamspsi$ being a ring map), $\lambda$ (as seen above), and $x, \beta^{\pm 1}$ (which can be checked in $\KO^*$, where it e.g. follows from the injection into $\KU^*$)

It remains to observe that $(\ko_2^\comp)^4(\RP^\infty_+) \to (\KO_2^\comp)^4(\RP^\infty_+)$ is injective.
Indeed the obstruction to this is \[ [\Sigma^{-3} \RP^\infty_+, (\KO_2^\comp)_{< 0}]  = [\Sigma^{-3} \RP^\infty_+, (\KO_2^\comp)_{\le -4}] = 0. \]
\end{proof}

\begin{theorem} \label{thm:bott-action}
Let $S$ be a scheme with $1/2 \in S$ and $n$ odd.
The action of $\adamspsi^n_{\OSnaith}$ on the Bott element $\beta$ is given by multiplication by $n^2 \cdot n^2_\epsilon$.
\end{theorem}
\begin{proof}
It suffices to prove the result for $S = \Spec(\Z[1/2])$.
We can write $\adamspsi_\OSnaith^n(\beta) = a \beta$, for some $a \in \GW(\Z[1/2])_2^\comp$; we need to show that $a = n^2 \cdot n^2_\epsilon$.
Since the map $\KO_2^\comp \to \KGL_2^\comp$ is compatible with the Adams action by construction, it follows (e.g. from Lemma \ref{lemm:action-KGL}) that $\rk(a) = n^4$.
Hence (using Lemma \ref{lemm:witt-injection}) it suffices to prove that $a$ has image $n^2$ in $\W(\Z[1/2])$.

We first prove the result for $S = Spec(\R)$ instead.
Consider \[ E = \KGL_2^\comp(\C) := \map_{\SH(\R)}(\Sigma^\infty_+ \Spec(\C), \KGL_2^\comp) \in \Fun(BC_2 \times BC_2', \SH); \] here the first action comes from complex conjugation and the second from the $C_2$-action on $\KGL$ (i.e. taking duals).
We have a map $E \to \K(\C^\topsup)_2^\comp \wequi \ku_2^\comp$ which is an equivalence by Suslin's theorem \cite[Corollary 4.7]{suslin1984k}.
It is $C_2 \times C_2'$-equivariant for the usual action on $\ku_2^\comp$.
By Lemma \ref{lemm:adams-action-alg-closed} the action of $\adamspsi^n$ on $\pi_2 E$ is by multiplication by $n$.
We may hence apply Lemma \ref{lemm:adams-technical} to deduce that the Adams action on $\pi_{-4}((E^{hC_2})_{\ge 0})^{hC_2'}$ is by multiplication by $n^{-2}$.
By the Quillen--Lichtenbaum conjecture for algebraic $K$-theory \cite[Theorem 2]{rosenschon2005homotopy} we have $((\KGL(\C)_2^\comp)^{hC_2})_{\ge 0} \wequi \KGL(\R)_2^\comp$.
We thus learn that $\adamspsi_\OSnaith^n$ acts by multiplication by $n^{-2}$ on $\pi_{-4}(\KO(\R)_2^\comp) \wequi \W(\R)_2^\comp\{\eta^4\beta^{-1}\}$.
The claim (over $\R$) follows since the Adams action on $\eta$ must be trivial (since it comes from the sphere spectrum; see Example \ref{ex:adamspsi-pi0}).

Now we go back to $S=Spec(\Z[1/2])$.
It would be enough to show that $\adamspsi^n_\OSnaith(\beta) = a\beta$, for some $a \in \GW(\Z)_2^\comp \subset \GW(\Z[1/2])_2^\comp$.
Indeed then we could determine $a$ by comparison with $S = \Spec(\R)$.
The containment $a \in \GW(\Z)_2^\comp \subset \GW(\Z[1/2])_2^\comp$ actually does hold for relatively formal reasons; see Lemma \ref{lemm:htpy-fixed-KGL-Z}.
Thus the proof is concluded.

Since the proof of Lemma \ref{lemm:htpy-fixed-KGL-Z} is rather involved (in that it relies on \cite{GW-III}), we provide here an alternative way of proceeding.\todo{this is a mess}
It is enough to show that $\adamspsi^n_\OSnaith$ acts on $\pi_{-4}(\KO_2^\comp) \wequi \W(\Z[1/2])_2^\comp\{\beta^{-1}\eta^4\}$ by multiplication by $n^{-2}$.
By arguing as in for example \cite[proof of Theorem 5.8]{bachmann-euler}, we find that \[ (*)\,\, \W(\Z[1/2]) \wequi \Z[g]/(g^2, 2g), \] where $g := \lra{2}-1$.\NB{details?}

%We deduce that the image of $a$ in $\W(\Z[1/2])$ is given by $a = n^2 + \epsilon g$, where $\epsilon \in \{0,1\}$.
Consider the decomposition \[ \Spec(\F_2) \xrightarrow{i} \Spec(\Z) \xleftarrow{j} \Spec(\Z[1/2]). \]
The sequence of functors \[ \Perf(\F_2) \xrightarrow{i_*} \Perf(\Z) \xrightarrow{j^*} \Perf(\Z[1/2]) \] induces a localization cofiber sequence $\K(\F_2) \to \K(\Z) \to \K(\Z[1/2])$ \cite[Example 6.11]{weibel-k-book} \cite[Theorems 3.21, and 7.4]{thomason-trobaugh}.
The functor $i_*$ is duality preserving if we give $\Perf(\F_2)$ the duality $\iHom(\ph, i^! \scr O)$\todo{ref}; then all the functors become $C_2$-equivariant and we get an induced localization sequence \[ \KGL_2^\comp(\F_2)^{hC_2} \to \KGL_2^\comp(\Z)^{hC_2} \to \KGL_2^\comp(\Z[1/2])^{hC_2}. \]
Since $i^! \scr O \wequi \Sigma \scr O$ we see that we get the \emph{shifted} duality on $\K(\F_2)$.
We have $\K(\F_2)_2^\comp \wequi \HZ_2^\comp$ \cite[Corollary IV.1.13]{weibel-k-book}, and the $C_2$-action is given by multiplication by $-1$, the duality being shifted.
In particular (see e.g. \cite[Lemma 1]{vcadek1999cohomology}\NB{$(\HZ_2^\comp)^{hC_2} \wequi (\HZ^{hC_2})_2^\comp$}) \[ \pi_*(\K(\F_2)_2^\comp)^{hC_2} = H^{-*}(BC_2, \Z_2^\comp) = \begin{cases} 0 & *\text{ even } \\ \Z/2 & *<0 \text{ odd }\end{cases}. \]
We thus get a short exact sequence \begin{gather*} (**)\,\, 0 = \pi_{-4} \KGL_2^\comp(\F_2)^{hC_2} \xrightarrow{i_*} \pi_{-4} \KGL_2^\comp(\Z)^{hC_2} \xrightarrow{j^*} \pi_{-4} \KGL_2^\comp(\Z[1/2])^{hC_2}\\ \xrightarrow{\partial} \pi_{-5} \KGL_2^\comp(\F_2)^{hC_2} = \Z/2. \end{gather*}

We now determine the image of the injection $j^*$.
By the homotopy fixed point theorem and $(*)$ we have \[ \pi_{-4} \KGL_2^\comp(\Z[1/2])^{hC_2} \wequi \W(\Z[1/2])_2^\comp \wequi \Z_2^\comp[g]/(g^2, 2g) \wequi \Z_2^\comp\{1\} \oplus \Z/2\{g\}. \]
The filtration induced by the homotopy fixed point spectral sequence is the $\I$-adic one, so the first two subquotients are  \[ \gr^0(\Z[1/2]) = \W(\Z[1/2])/\I \wequi \F_2\{1\} \quad\text{and}\quad \gr^2(\Z[1/2]) = \I(\Z[1/2])/\I^2 \wequi \F_2\{\lra{-1}, \lra{2}\}, \] which must be a subquotient of the appropriate group $e_2^{(i)}(\Z[1/2]) = H^{4+i}(C_2, K_i(\Z[1/2])_2^\comp)$ on the $E_2$ page.
By \cite[Corollary 16.3]{milnor1971introduction} \[ K_1(\Z[1/2])_2^\comp \wequi (\Z[1/2]^\times)_2^\comp \wequi \Z_2^\comp \oplus \Z/2, \] and also $K_0(\Z[1/2])_2^\comp \wequi \Z_2^\comp$.
We deduce that $e_2^{(i)}(\Z[1/2])$ is an $\F_2$-vector space of dimension at most $1$ if $i=0$, and at most $2$ if $i = 1$\NB{because we don't bother to work out the action}; since the subquotient $\gr^i(\Z[1/2])$ has the same dimension we find that $e_2^{(i)}(\Z[1/2]) = \gr^i(\Z[1/2])$.
The map $e_2^{(0)}(\Z) \to e_2^{(0)}(\Z[1/2])$ is an isomorphism, which implies that $\gr^0(\Z) \to \gr^0(\Z[1/2])$ is injective.
On the other hand $K_1(\Z)_2^\comp \wequi \Z^\times$, which implies that the map $e_2^{(1)}(\Z) \to e_2^{(1)}(\Z[1/2]) \wequi \F_2\{\lra{-1}, \lra{2}\}$ is an injection onto $\F_2\{\lra{-1}\}$.
These facts together imply that $j^*: \pi_{-4} \KGL_2^\comp(\Z)^{hC_2} \to \pi_{-4} \KGL_2^\comp(\Z[1/2])^{hC_2}$ does not hit the element $g$.\NB{$\gr^0(j^*)$ is an injection, whence $g\in F^1$ can only be hit from $F^1$, but not hit on $\gr^1$}
Consequently $\partial(g) = 1$ and $\partial(1 + \epsilon g) = 0$ for some $\epsilon \in \{0,1\}$.
Let us put $\alpha := 1 + \epsilon g$.
It follows from exactness of $(**)$ that \[ \pi_{-4} \KGL_2^\comp(\Z)^{hC_2} \wequi \Z_2^\comp\{\alpha\} \subset \pi_{-4} \KGL_2^\comp(\Z[1/2])^{hC_2}. \]

Note that we have an action of $\adamspsi^n := \adamspsi^n_\Snaith$ also on $\KGL(\Z)$, and hence on $\KGL_2^\comp(\Z)^{hC_2}$.
It follows that $\adamspsi^n(\alpha) = p\alpha$ for some $p \in \Z_2$.
We know (by Example \ref{ex:adamspsi-pi0}) that $\adamspsi$ acts on $\pi_{-4} \KGL_2^\comp(\Z[1/2])^{hC_2}$ by multiplication by some element $a + \epsilon' g \in \W(\Z[1/2])$; here $a \in \Z_2$ and $\epsilon' \in \Z/2$.
Comparison with the case $S=\Spec(\R)$ shows that $a = n^{-2}$.
We thus get \[ p j^* \alpha = j^*(p\alpha) = j^*(\adamspsi(\alpha)) = \adamspsi(j^*(\alpha)) = (n^{-2} + \epsilon' g) j^*\alpha \] and so \begin{align*} p(1 + \epsilon g) &= (n^{-2} + \epsilon' g)(1+\epsilon g) \\ &= n^{-2} + (n^{-2} \epsilon + \epsilon') g. \end{align*}
Comparing coefficients of $1$ yields $p=n^{-2}$, and then comparing coefficients of $g$ yields $\epsilon'=0$.
This was to be shown.
\end{proof}

\subsection{The $2$-adic fracture square} \label{subsec:2-fracture}
Recall the following well-known fact.
\begin{lemma}
Let $\scr C$ be a stable, presentable, compactly generated $\infty$-category and $n \in \Z$.
Then for every $E \in \scr C$ the natural commutative square
\begin{equation*}
\begin{CD}
E @>>> E[1/n] \\
@VVV  @VVV    \\
E_n^\comp @>>> E_n^\comp[1/n]
\end{CD}
\end{equation*}
is cartesian.
\end{lemma}
\begin{proof}
Let $X \in \scr C$ be compact.
Functors of the form $\map(X, \ph)$ preserve limits and colimits, so completion and localization, and form a conservative collection.
This reduces the result to $\scr C = \SH$ where it is well-known; see e.g. \cite[Proposition 2.2]{bauer2011bousfield}.
\end{proof}

We can apply this with $\scr C = \SH(S)$ and $n=2$ to obtain the fracture square
\begin{equation} \label{eq:2-facture}
\begin{CD}
E @>>> E[1/2] \wequi E[1/2]^- \vee E[1/2]^+ \\
@VVV     @VVV \\
E_2^\comp @>>> E_2^\comp[1/2] \wequi E_2^\comp[1/2]^- \vee E_2^\comp[1/2]^+; \\
\end{CD}
\end{equation}
here we have used the decomposition of $2$-periodic spectra into $+$ and $-$ parts (see \S\ref{subsub:plus-minus}).
\begin{definition}\label{def:generalized-plus}
For $E \in \SH(S)$ we define \[ E^\pm = E_2^\comp \times_{E_2^\comp[1/2]^\pm} E[1/2]^\pm. \]
\end{definition}
Note that there are maps \[ E \to E^\pm \quad\text{and}\quad E^\pm \to E_2^\comp \to E_2^\comp[1/2]^\mp. \]

\begin{lemma} \label{lemm:pasting}
Let $\scr C$ be an $\infty$-category and $X, A_1, A_2, B_1, B_2 \in \scr C$ with maps $X \to B_1 \times B_2$ and $A_i \to B_i$.
Then \[ X \times_{B_1 \times B_2} (A_1 \times A_2) \wequi (X \times_{B_1} A_1) \times_{B_2} A_2. \]
\end{lemma}
\begin{proof}
Consider the following commutative diagram
\begin{equation*}
\begin{tikzcd}
(X \times_{B_1} A_1) \times_{B_2} A_2 \ar[r] \ar[d] & A_1 \times A_2 \ar[r] \ar[d]            & A_2 \ar[d] \\
X \times_{B_1} A_1 \ar[r] \ar[d] & A_1 \times B_2 \ar[r, shift right=1.5] \ar[r, shift left=1.5] \ar[d] & {B_2 \atop A_1} \ar[d] \\
X \ar[r]                         & B_1 \times B_2 \ar[r]                                      & B_1.
\end{tikzcd}
\end{equation*}
The two squares on the right hand side are cartesian, and so are the horizontal rectangles.
It follows that the left hand squares are cartesian, and hence so is the left hand vertical rectangle, by the pasting law \cite[Lemma 4.4.2.1]{lurie-htt}.
\end{proof}

Our slightly unconventional Definition \ref{def:generalized-plus} is partially justified by the following result.
\begin{corollary} \label{cor:mot-fracture}
For $E \in \SH(S)$ we have cartesian squares
\begin{equation*}
\begin{split}
\begin{CD}
E @>>> E[1/2]^- \\
@VVV    @VVV    \\
E^+ @>>> E_2^\comp[1/2]^-
\end{CD}
\end{split}
\quad\text{and}\quad
\begin{split}
\begin{CD}
E @>>> E[1/2]^+ \\
@VVV    @VVV    \\
E^- @>>> E_2^\comp[1/2]^+
\end{CD}
\end{split}
\end{equation*}
\end{corollary}
\begin{proof}
Immediate from Lemma \ref{lemm:pasting}, the fracture square \eqref{eq:2-facture} and the definition of $E^\pm$.
\end{proof}

\begin{remark} \label{rmk:e-infinity-fracturing}
Note that if $E \in \CAlg(\SH(S))$ then $E_2^\comp$, $E[1/2]^+ = E[1/2]_\eta^\comp$ and $E[1/2]^- = E[1/2,1/\eta]$ are $\scr E_\infty$-rings, and hence so are $E^{\pm}$ (being pullbacks of $\scr E_\infty$-rings along $\scr E_\infty$-ring maps).
All in all the fracture squares for $E$ from Corollary \ref{cor:mot-fracture} are pullback diagrams in $\CAlg(\SH(S))$.
\end{remark}

\subsection{Proof of the main theorem}
\begin{lemma}[Heard] \label{lemm:heard}
The diagram $(\alpha: \KO \to \KGL) \in \Fun(BC_2^\triangleleft, \SH(S))$ is a universal $\eta$-complete limit diagram: if $F: \SH(S) \to \scr C$ is any functor preserving binary products, then $F(\alpha/\eta)): F(\KO/\eta) \to (F(\KGL/\eta))^{hC_2}$ is an equivalence.
\end{lemma}
\begin{proof}
This is essentially \cite[\S3]{heard2017homotopy}.
To paraphrase, $(\KO/\eta \to \KGL/\eta) \in \Fun(BC_2, \SH(S))$ identifies with $(\KGL \to \KGL^{C_2})$, where by $\KGL^{C_2}$ we mean the product $\KGL \times \KGL$ with its switch action.
This is clearly a universal limit diagram.
In slightly more detail, let $i: * \to BC_2$ and $p: BC_2 \to *$ be the canonical functors.
Then $X^{C_2} \wequi i_*(X)$ and $Y^{hC_2} \wequi p_*(Y)$; thus $(X^{C_2})^{hC_2} \wequi p_*i_* X \wequi (\id)_* X \wequi X$.
\end{proof}

We can use this to identify $\KO^+$, in the sense of Definition \ref{def:generalized-plus}.
\begin{lemma} \label{lemm:KO+}
Let $S$ satisfy the assumptions of Proposition \ref{prop:htpy-fixed} (such as $\Spec(\Z[1/2])$).
Then for $n$ odd we have $\KO[1/n]^+ \wequi \KGL[1/n]^{hC_2} \in \SH(S)$.
\NB{for $n$ odd we have $E^+[1/n] \wequi E[1/n]^+$, and $F^{hC_2}[1/n] \wequi F[1/n]^{hC_2}$: comparing fracture squares, first claim reduces to $E_2^\comp \wequi E_2^\comp[1/n] \wequi E[1/n]_2^\comp$. For latter claim, take $(\ph)^{hC_2}$ in $2$-adic fracture square for $F, F[1/n]$. Suffices to show that inverting $n$ commutes with $(\ph)^{hC_2}$ on $F[1/2]$, which holds since $(\ph[1/2])^{hC_2} \wequi (\ph[1/2])_{hC_2}$.}
The same holds for $n$ even and any $S$.
\end{lemma}
\begin{proof}
For $n$ even we have \[ \KO[1/n]^+ \wequi \KO[1/n]_\eta^\comp \wequi \KGL[1/n]^{hC_2}, \] by Lemma \ref{lemm:heard}.
Thus we now consider $n$ odd.

Taking homotopy fixed points in the $2$-adic fracture square for $\KGL[1/n]$ we obtain a cartesian square
\begin{equation*}
\begin{CD}
\KGL[1/n]^{hC_2} @>>> \KGL[1/2n]^{hC_2} \\
@VVV               @VVV \\
(\KGL_2^\comp)^{hC_2} @>>> (\KGL_2^\comp[1/2])^{hC_2},
\end{CD}
\end{equation*}
noting that $\KGL[1/n]_2^\comp \wequi \KGL_2^\comp$.
By Proposition \ref{prop:htpy-fixed} we have $(\KGL_2^\comp)^{hC_2} \wequi \KO_2^\comp \wequi \KO[1/n]_2^\comp$.
We also have \[ \KGL[1/2n]^{hC_2} \wequi (\KGL[1/2n]^{hC_2})_\eta^\comp \stackrel{(L.\ref{lemm:heard})}{\wequi} \KO[1/2n]_\eta^\comp \wequi \KO[1/n][1/2]^+; \] here we have used that $\KGL[1/2n]^{hC_2}$ is $\eta$-complete since $\KGL[1/2n]$ is ($\eta$ acting by $0$).
The same argument shows that $(\KGL_2^\comp[1/2])^{hC_2} \wequi \KO_2^\comp[1/2]^+ \wequi \KO[1/n]_2^\comp[1/2]^+$.
Hence the above cartesian square identifies with the defining square of $\KO[1/n]^+$.
\end{proof}

\begin{remark} \label{rmk:construction-even}
Using $\KO[1/n]^+ \wequi \KGL[1/n]^{hC_2}$, for any $n$ we can define $(\adamspsi^n)^+: \KO[1/n]^+ \to \KO[1/n]^+$ as $(\adamspsi^n_\Snaith)^{hC_2}$.
If $n$ is even this is the best we can do (see Remark \ref{rmk:obstruction-even}).
For $n$ odd we shall see how to extend this to an endomorphism of all of $\KO[1/n]$.
\end{remark}

\begin{definition} \label{def:psin-OSnaith-ext}
For $n$ odd and $S$ satisfying the assumptions of Proposition \ref{prop:htpy-fixed}, we denote by \[ \adamspsi^n_\OSnaith: \KO[1/n]^+ \to \KO[1/n]^+ \] the map induced from $\adamspsi^n_\Snaith: \KGL[1/n] \to \KGL[1/n]$ (see \S\ref{subsub:GS}) via the equivalence $\KGL[1/n]^{hC_2} \wequi \KO[1/n]^+$.
\end{definition}

For $S = \Spec(\Z[1/2])$ we have $\KO[1/2n]^- \in \SH(\Z[1/2])[1/2]^- \stackrel{r_\R}{\wequi} \SH[1/2]$, and $r_\R(\KO) \wequi \KO^\topsup[1/2]$ by Lemma \ref{lemm:KO-real-realization}.
By $\adamspsi^n_\topsup: \KO^\topsup[1/n] \to \KO^\topsup[1/n]$ we mean the topological Adams operation, for example obtained by taking homotopy fixed points of the complex realization of the operation on $\KGL$.
\begin{lemma} \label{lemm:adamspsi-E-infty}
In $\CAlg(\SH(\Z[1/2]))$ there exists a unique (up to homotopy) map $\adamspsi^n: \KO[1/n] \to \KO[1/n]$ such that the two squares
\begin{equation*}
\begin{split}
\begin{CD}
\KO[1/n] @>>> \KO[1/n]^+ \\
@V{\adamspsi^n}VV  @V{\adamspsi_{\OSnaith}^n}VV \\
\KO[1/n] @>>> \KO[1/n]^+ \\
\end{CD}
\end{split}
\quad\text{and}\quad
\begin{split}
\begin{CD}
\KO[1/n] @>>> \KO[1/2n]^- \wequi \KO^\topsup[1/2n] \\
@V{\adamspsi^n}VV       @V{\adamspsi^n_\topsup}VV \\
\KO[1/n] @>>> \KO[1/2n]^- \wequi \KO^\topsup[1/2n] \\
\end{CD}
\end{split}
\end{equation*}
commute (in $\CAlg(\SH(\Z[1/2]))$).
\end{lemma}
\begin{proof}
Applying Remark \ref{rmk:e-infinity-fracturing} with $E = \KO[1/n]$ (and $S=\Spec(\Z[1/2])$) we obtain a pullback square
\begin{equation*}
\begin{CD}
\Map_{\CAlg(\SH(\Z[1/2]))}(\KO[1/n], \KO[1/n]) @>>> \Map_{\CAlg(\SH(\Z[1/2]))}(\KO[1/n], \KO[1/n]^+) \\
@VVV @VVV \\
\Map_{\CAlg(\SH(\Z[1/2]))}(\KO[1/n], \KO[1/2n]^-) @>>> \Map_{\CAlg(\SH(\Z[1/2]))}(\KO[1/n], \KO_2^\comp[1/2]^-) =: M.
\end{CD}
\end{equation*}
We have two maps $a,b: \KO[1/n] \to \KO_2^\comp[1/2]^- \in \CAlg(\SH(\Z[1/2]))$, induced by $\adamspsi^n_\OSnaith$ and $\adamspsi^n_\topsup$, respectively.
A choice of $\scr E_\infty$-ring map $\adamspsi^n$ with the desired properties is a path in $M$ from $a$ to $b$.
Thus such a map exists if $a, b$ are homotopic, and is unique if $\pi_1$ of the corresponding component of $M$ vanishes.
Since $\KO_2^\comp[1/2]^- \in \SH(\Z[1/2])[1/2]^- \otimes \Q \stackrel{r_\R}{\wequi} D(\Q)$, we find that \[ M \wequi \Map_{\CAlg(D(\Q))}(r_\R(\KO) \otimes \Q, r_\R(\KO_2^\comp[1/2]^-)). \]
We claim that $r_\R(\KO) \otimes \Q$ is the free $\scr E_\infty$-ring over $\Q$ on an invertible generator in degree $4$.
Indeed if we denote that universal object by $F(\Sigma^4 \1)[x^{-1}]$, then the canonical map $F(\Sigma^4 \1)[x^{-1}] \to r_\R(\KO) \otimes \Q$ induces an isomorphism on $\pi_*$ by Lemma \ref{lemm:KO-real-realization} and the well-known computation $\pi_*F(\Sigma^{2n} \1) \wequi \Q[x]$ (with $|x|=2n$; this follows e.g. from \cite[Corollary 8.4]{richter2017algebraic}).
We deduce that \[ M \subset \Omega^{\infty + 4} r_\R(\KO_2^\comp[1/2]^-) \] is given by the union of those path components corresponding to invertible elements of $\pi_4 r_\R(\KO_2^\comp[1/2]^-)$.
It thus follows from Lemma \ref{lemm:KO-real-realization} again that $\pi_0 M \wequi \Q_2^\times\{x\}$ and $\pi_1 M = 0$ (for any choice of base point).
In other words $\adamspsi^n$ exists if and only if $a$ and $b$ act in the same way on $\beta$, and if so it is unique.
Since both maps act by multiplication by $n^2$ (see Theorem \ref{thm:bott-action} for $\adamspsi^n_\OSnaith$, and for $\adamspsi^n_\topsup$ we can use e.g. comparison with the action on $\KU$ and Lemma \ref{lemm:action-KGL}), the result follows.
\end{proof}

\begin{proof}[Proof of Theorem \ref{thm:ko-adams}]
Since everything is stable under base change, we may assume that $S = \Spec(\Z[1/2])$.
We let $\adamspsi^n$ be the map constructed in Lemma \ref{lemm:adamspsi-E-infty}.
Hence $(*)$ the induced endomorphism of $\KO^+ \wequi \KGL^{hC_2}$ is given by $\adamspsi^n_\OSnaith \wequi (\adamspsi^n_\Snaith)^{hC_2}$.
It follows that $(**)$ the induced endomorphism of $\KO_2^\comp \wequi (\KO^+)_2^\comp \wequi (\KGL_2^\comp)^{hC_2}$ is also given by $\adamspsi^n_\OSnaith$.

(1,3) Hold by definition.

(2) By Corollary \ref{corr:completion-inj} and Lemma \ref{lemm:GW-inj}, the map $\pi_{8,4}(\KO) \to \pi_{8,4}(\KO_2^\comp)$ is injective, and hence it suffices to prove the claim for $(\adamspsi^n)_2^\comp$.
By $(**)$, this is $\adamspsi^n_\OSnaith$, so the claim follows from Theorem \ref{thm:bott-action}.

(4) The map $\KO \to \KGL$ factors as $\KO \to \KO^+ \wequi \KGL^{hC_2} \to \KGL$, so the claim follows from $(*)$.
\end{proof}

\begin{remark}
The construction of $\adamspsi^n$ also implies that if $S$ is a regular $QL$-scheme, then under the equivalence $\KO[1/n]^+ \wequi \KGL[1/n]^{hC_2}$  we have $(\adamspsi^n)^+ \wequi \adamspsi^n_\OSnaith$ (see Definition \ref{def:psin-OSnaith-ext}).\NB{really?}
\end{remark}

\subsection{The motivic image of orthogonal $j$ spectrum}
In this section we show that our Adams operation can be used to construct a ``motivic image of orthogonal $j$'' spectrum.
This idea originated in discussions with JD Quigley and Dominic Culver regarding \cite{culver2019complex}.
None of the results in this section are used in the sequel.

We begin by determining the lowest two ``generalized slices'' of the sphere spectrum.
\begin{theorem} \label{thm:very-effective-slices}
Let $k$ be a field of exponential characteristic $e \ne 2$.
\begin{enumerate}
\item The unit map $\1 \xrightarrow{u} \KO \in \SH(k)$ induces an equivalence on $\tilde s_0$.
\item There is a fibration sequence\todo{(Röndigs) figure out the boundary map?} \[ \Sigma^{3,2}\HZ/24 \to \tilde s_1(\1)[1/e] \xrightarrow{\tilde s_1(u)} \tilde s_1(\KO)[1/e] \wequi \Sigma^{2,1} \HZ/2. \]
\end{enumerate}
\end{theorem}
\begin{proof}
(1) For $E \in \SH(k)$ there is a functorial cofiber sequence \cite[Lemma 11(1)]{bachmann-very-effective} \[ s_0(E_{\ge 1}) \to \tilde{s}_0 E \to f_0 \ul{\pi}_0(E)_*. \]
It thus suffices to show that $\1 \to \KO$ induces an equivalence on $s_0(\ph_{\ge 0})$ and $\ul{\pi}_0(\ph)_0$.
Indeed then it also induces an equivalence on $f_0\ul{\pi}_0(\ph)_*$, so on $s_0(\ul{\pi}_0(\ph)_*)$, and on $s_0(\ph_{\ge 1})$ (the latter since $s_0$ is a stable functor), and hence on $\tilde s_0(\ph)$ by the cofiber sequence.
It is well-known that the unit map $\1 \to \KO$ induces an isomorphism on $\ul{\pi}_{0,0}$.\NB{ref?}
Similarly we know that $s_0(\1) \wequi \HZ$ and $s_0(\KO_{\ge 0}) \wequi \HZ$ (e.g. use \cite[Theorem 16]{bachmann-very-effective}).
In the commutative diagram
\begin{equation*}
\begin{CD}
\GW(k) \wequi \pi_{0,0}(\1) @>{\wequi}>> \pi_{0,0}(\KO_{\ge 0}) \wequi \GW(k) \\
@VVV                                        @VVV \\
\Z \wequi \pi_{0,0} s_0(\1) @>{\alpha}>> \pi_{0,0} s_0(\KO_{\ge 0}) \wequi \Z
\end{CD}
\end{equation*}
the vertical maps are both the rank map; it follows that $\alpha$ is an isomorphism.
Since this holds over any field, the map $s_0(\1) \to s_0(\KO_{\ge 0})$ induces an isomorphism on $\ul{\pi}_{0,0}$; since this is the only non-vanishing effective homotopy sheaf of $\HZ$ the map is an equivalence.

(2) We invert the exponential characteristic throughout.
For $E \in \SH(k)$ there is a functorial cofiber sequence \cite[Lemmas 11(2) and 8]{bachmann-very-effective} \[ f_2 \Sigma \ul{\pi}_1(E)_* \to \tilde s_1(E) \xrightarrow{\beta} s_1(E_{\ge 1}). \]
Note that $s_1$ of the first term of the cofiber sequence vanishes, so $(*)$ the map $\beta$ is equivalent to the projection $\tilde s_1 E \to s_1 \tilde s_1 E$.

We shall show that (a) $f_2 \Sigma \ul{\pi}_1(\1)_* \wequi \Sigma^{3,2} \HZ/24$, (b) the unit map induces an isomorphism $s_1(\1_{\ge 1}) \to s_1(\KO_{\ge 1}) \wequi \Sigma^{2,1} \HZ/2$, and (c) $f_2 \Sigma \ul{\pi}_1(\KO)_* = 0$.
This will imply the result.

(a) First note that \[ \ul{\pi}_1(f_0 \KO)_{-2} \wequi \ul{\pi}_1(\KO)_{-2} \wequi \ul{\pi}_0(\KO^{[-1]})_{-1} \stackrel{(**)}{\wequi} \Z_{-1} \wequi 0, \] where $(**)$ is obtained from e.g. \cite[Table 1]{bachmann-very-effective}.
This implies via \cite[(1.2)]{rondigs2016first} that \[ \pi_1(\1)_{-2} \wequi \Z/24. \]
Since this holds compatibly over any field (see \cite[Remark 5.8]{rondigs2016first}), Lemma \ref{lemm:HZ-detect} below implies the claim.

(b) The claim that $s_1(\KO_{\ge 1}) \wequi \Sigma^{2,1} \HZ/2$ is immediate from $(*)$ and \cite[Theorem 16]{bachmann-very-effective}.
Since $\1 \to \KO$ induces an isomorphism $\ul{\pi}_0(\ph)_0$ it also does on $f_0 \ul{\pi}_0(\ph)_*$ and hence on $s_1 \ul{\pi}_0(\ph)_*$.
Considering the cofiber sequence $s_1(\ph_{\ge 1}) \to s_1(\ph_{\ge 0}) \to s_1(\ul{\pi}_0(\ph)_*)$, it thus suffices to show that $s_1(\1) \to s_1(\KO_{\ge 0})$ is an equivalence.
Note that $s_1(\KO_{\ge 0}) \wequi s_1 f_0 \KO_{\ge 0} \wequi s_1 \ko$.
The claim thus follows from \cite[Theorem 3.2]{ananyevskiy2017very} and \cite[Corollary 2.13 and Lemma 2.28]{rondigs2016first} (both spectra are given by $\Sigma^{1,1} \HZ/2$).

(c) Since $\tilde s_1 \KO \wequi \Sigma^{2,1} \HZ/2$ \cite[Theorem 16]{bachmann-very-effective} is a $1$-slice, the claim follows from $(*)$.

This concludes the proof.
\end{proof}

\begin{lemma} \label{lemm:HZ-detect}
Let $k$ be a field, $H \in \SH(k)^{\veff\heart}$.
Suppose that (1) $\pi_{0,0}(H) \wequi \Z/n$ for some $n \in \Z$, and (2) for every finitely generated separable field extension $K/k$, the map $\ul{\pi}_{0,0}(H)(k) \to \ul{\pi}_{0,0}(H)(K)$ is an isomorphism.
Then $H \wequi \HZ/n$.
\end{lemma}
\begin{proof}
Write $\ul{GW} \in \SH(k)^{\veff\heart}$ for the unit.
The isomorphism $\pi_{0,0}(H) \wequi \Z/n$ induces a map $\alpha: \ul{GW} \to H$.
We also have the rank map $\beta: \ul{GW} \to \HZ/n$.
We shall show that $\alpha, \beta$ are surjections with equal kernels; this will prove the result.
Since $\ul{\pi}_{0,0}(H)(k) \to \ul{\pi}_{0,0}(H)(K)$ is an isomorphism and $\ul{GW}(k) \to \ul{\pi}_{0,0}(H)(k)$ is surjective, $\alpha$ is surjective.
We claim that $\alpha(K)$ is the rank map.
Indeed let $\bar{K}/K$ be a separable closure and consider the commutative diagram
\begin{equation*}
\begin{CD}
\ul{GW}(K) @>{\alpha(K)}>> \ul{\pi}_{0,0}(H)(K) \wequi \Z/n \\
@VVV                       @|     \\
\Z \wequi \ul{GW}(\bar K) @>{\alpha(\bar K)}>> \ul{\pi}_{0,0}(H)(\bar K) \wequi \Z/n.
\end{CD}
\end{equation*}
The map $\alpha(\bar K)$ is the unique morphism of abelian groups sending $1$ to $1$, whence $\alpha(K)$ is the rank map as desired.
We deduce that $\ker(\alpha)(K) = \ker(\beta)(K)$, and hence $\ker(\alpha) = \ker(\beta)$ by unramifiedness (and since ``taking the underlying sheaf'' is an exact conservative functor \cite[Proposition 5(3)]{bachmann-very-effective}).
The result follows.
\end{proof}

We put $\ksp := \tilde f_0 \Sigma^{4,2} \KO$, so that \[ \Sigma^{4,2} \ksp \wequi \tilde f_2 \KO. \]
\begin{corollary}\NB{Also for $n$ even with $\ko[1/n\cdot n_\epsilon]$...}
For $n$ odd, the Adams operation $\adamspsi^n-1: \ko[1/n] \to \ko[1/n]$ lifts to $\adamsphi^n: \ko[1/n] \to \Sigma^{4,2} \ksp[1/ne]$.
\end{corollary}
\begin{proof}
We invert $e$ throughout.

Let us write $\tilde f_{\le 1}$ for the cofiber of $\tilde f_2 \to \id$.
Then we have a cofiber sequence $\Sigma^{4,2} \ksp \to \ko \to C := \tilde f_{\le 1} \ko$.
We need to show that the composite \[ \ko[1/n] \xrightarrow{\adamspsi^n-1} \ko[1/n] \to C[1/n] \] is zero.
Considering the cofiber sequence \[ \1 \to \ko \to D, \] it suffices to show that (a) the composite $\1 \to \ko \xrightarrow{\adamspsi^n-1} \ko \to C$ is zero, and (b) any map $D \to C$ is zero.

(a) It is enough to show that $\1 \to \ko \xrightarrow{\adamspsi^n-1} \ko$ is zero.
This is clear since $\adamspsi^n(1) = 1$.

(b) Recall that $\Sigma^{4,2} \SH(k)^\veff$ defines the non-negative part of a $t$-structure on $\SH(k)$ with $\tilde f_2$ and $\tilde f_{\le 1}$ as non-negative and negative truncation, respectively (see e.g. \cite[\S B]{bachmann-norms}).
It follows that $C$ is in the negative part of this $t$-structure, and hence it is enough to show that $D$ is in the non-negative part.
Consider the following commutative diagram, in which all rows and columns are cofiber sequences (defining $E, F$)
\begin{equation*}
\begin{CD}
\tilde f_2 \1 @>>> \1 @>>> \tilde f_{\le 1} \1 \\
@VVV            @VVV          @VVV \\
\tilde f_2 \ko @>>> \ko @>>> \tilde f_{\le 1} \ko \\
@VVV                @VVV     @VVV \\
E              @>>> D   @>>> F.
\end{CD}
\end{equation*}
It suffices to show that $E, F \in \Sigma^{4,2} \SH(k)^\veff$.
This is clear for $E$, since $\Sigma^{4,2} \SH(k)^\veff$ is closed under colimits.
It follows from Theorem \ref{thm:very-effective-slices} that $F \wequi \Sigma^{4,2} \HZ/24 \in \Sigma^{4,2} \SH(k)^\veff$.\NB{details?}

This concludes the proof.
\end{proof}
\NB{Uniqueness of the lift?}

\begin{definition}
A motivic image of orthogonal $j$ spectrum is any spectrum $\jo$ in a fiber sequence \[ \jo \to \ko_{(2)} \xrightarrow{\adamsphi^3} \Sigma^{4,2} \ksp_{(2)}. \]
\end{definition}
Clearly the unit map $\1_{(2)} \to \ko_{(2)}$ lifts to $\jo$.
Theorem \ref{thm:main} states that the lifted unit map $\1_{(2)} \to \jo$ is an $\eta$-periodic equivalence.\NB{What is complex realisation of $\jo$?}

\subsection{Geometric operations} \label{subsec:geometric-ops}
In this section we compare our stable Adams operations to the ones constructed by Fasel--Haution \cite{FH-adams}.
None of the results in this section are used in the sequel.

\subsubsection{}
We begin with the following adaptation of \cite[Theorem 13.1]{panin2010motivic}.
\begin{lemma} \label{lemm:KO-opns}
For $n \ne 0$ the map \[ [\KO[1/n], \KO[1/n]]_{\CAlg(\h\SH(\Z[1/2]))} \to \KSp[1/n]^0(\HP^\infty) \times \GW(\Z[1/2])[1/n], \alpha \mapsto (\alpha(H(1)), \alpha(\beta)/\beta) \] is injective.
Here $H(1) \in \KSp^0(\HP^\infty)$ corresponds to the tautological bundle.
\end{lemma}
\begin{proof}
There is a presentation $\KO \wequi \colim_i \Sigma^{\infty+(4-8i,2-4i)} X_i$, for a sequence of pointed, smooth, affine $\Z[1/2]$-schemes $X_i$ \cite[Theorem 12.3]{panin2010motivic}.
By \cite[proof of Theorems 13.1, 13.2, 13.3]{panin2010motivic}, the induced map $\gamma: [\KO, \KO[1/n]] \to \lim_i \KO[1/n]^{8i-4,4i-2}(X_i)$ is an injection (in fact, equivalence).
Note that $\gamma(\alpha)_i = \alpha(\gamma(\id)_i)$.
The isomorphism $\KO[1/n]^{8i-4,4i-2}(X_i) \wequi \KSp[1/n]^0(X_i)\{\beta^{-i}\}$ hence shows that a ring map $\alpha: \KO[1/n] \to \KO[1/n]$ is determined by its effect on $\KSp[1/n]^0(X)$ for $X$ smooth affine, and on $\beta$.
By \cite[Theorem 8.1]{panin2010motivic}, since $X$ is affine the action of $\alpha$ is determined by the action on the class corresponding to the tautological bundle of $\HGr(r, \infty)$ (for various $r$), and on $\beta$.
The claim now follows by the computation of the cohomology of $\HGr(r, \infty)$ in terms of $\HP^\infty$ \cite[Theorems 11.4 and 8.1]{panin2010quaternionic}.
\end{proof}

Recall that over any base scheme $S$ with $1/2 \in S$ we have \[ \KSp^0(\HP^\infty_S) = \KO^{4,2}(\HP^\infty_S) \wequi \bigoplus_{i \ge 0} \KO^{4-4i,2-2i}(S)\{b_1^{\KO}(\gamma)^i\}. \]
Since $\KO^{8i, 4i}(S) = \GW(S)$ and $\KO^{8i-4,4i-2}(S) = \Z$ make sense also for $S=\Spec(\Z)$ (with $\GW(\Z) = \Z \oplus \Z\{\lra{-1}\}$), we will by slight abuse of notation put \[ \KSp^0(\HP^\infty_\Z) := \bigoplus_{i \ge 0} \KO^{4-4i,2-2i}(\Z)\{b_1^{\KO}(\gamma)^i\}. \]
\begin{corollary} \label{cor:KO-ring-end}
Let $n$ be odd and $R \subset [\KO[1/n], \KO[1/n]]_{\SH(\Z[1/2])}$ denote the set of those homotopy ring maps such that $\alpha(\beta)/\beta \in \GW(\Z)[1/n] \subset \GW(\Z[1/2])[1/n]$ and $\alpha(H(1)) \in \KSp^0(\HP^\infty_\Z)[1/n] \subset \KSp^0(\HP^\infty_{\Z[1/2]})[1/n]$.
Then the map \[ R \to \GW(\Z)[1/n], \alpha \mapsto \alpha(\beta)/\beta \] is an injection.
\end{corollary}
\begin{proof}
Consider the commutative diagram
\begin{equation*}
\begin{CD}
R @>>> \KSp[1/n]^0(\HP^\infty_\Z) \times \GW(\Z)[1/n] \\
@V{\alpha \mapsto \alpha \otimes \Q}VV @VVV \\
[\KO \otimes \Q, \KO \otimes \Q]_{\CAlg(\h\SH(\Z[1/2]))} @>>> (\KSp \otimes \Q)^0(\HP^\infty_{\Z[1/2]})\times (\GW(\Z[1/2]) \otimes \Q).
\end{CD}
\end{equation*}
The top horizontal map is injective by Lemma \ref{lemm:KO-opns}, and the right hand vertical map is injective since $\GW(\Z), \Z$ are torsion-free.
It follows that the left hand vertical map is injective.

Write $H_- \in \KO^{-4,-2}(*)$ for the trivial symplectic bundle.
We shall now show that a homotopy ring map $\alpha: \KO \otimes \Q \to \KO \otimes \Q$ (over $\Z[1/2]$) is determined by its effect on $\beta$ and $H_-$.
Since \[ \SH(S)\otimes \Q \wequi (\SH(S) \otimes \Q)^+ \times (\SH(S) \otimes \Q)^- \] as symmetric monoidal categories, we need only prove the same claim about $(\KO \otimes \Q)^\pm$.
We have equivalences of homotopy ring spectra $(\KO \otimes \Q)^+ \wequi \HZ[t, t^{-1}] \otimes \Q$ and $(\KO \otimes \Q)^- \wequi (\1 \otimes \Q)^-[u, u^{-1}]$ (see e.g. \cite[Theorem D]{deglise2019borel}); here $|t| = (4,2)$ and $|u| = (8,4)$.
Since \[ [\1, \Sigma^{2n,n}(\1 \otimes \Q)^-] \stackrel{r_\R}{\wequi} [\Q, \Q[n]]_{D(\Q)} = 0 \text{ for } n \ne 0 \] and similarly \[ [\HZ, \Sigma^{2n,n} \HZ \otimes \Q]_{\SH(\Z[1/2])} = 0 \text{ for } n \ne 0 \] (see e.g. \cite[Remark 5.3.16]{riou2010algebraic}), it follows that $\alpha$ is determined by its effect on $t, u$.
Since $t$ can be chosen to be the image of $H_-$, and $u$ the image of $\beta$, the claim follows.

We deduce the following: given $\alpha \in R$, write $\alpha(\beta) = a \beta$ and $\alpha(H_-) = b H_-$, for some $a \in \GW(\Z)[1/n]$, $b \in \Z[1/n]$ (uniquely determined).
Then $a$ and $b$ determine $\alpha$.
To conclude the proof, we need to show that $a$ determines $b$.
Since $H_-^2 = 2h\beta$, we find that $b^2 = \rk(a)$, so that $a$ determines $b$ up to a sign.
It will thus suffice to prove that $b \equiv 1 \pmod{4}$.\NB{Since among $b,-b$, only one can be $\equiv 1 \pmod 4$. Here we are definitely using that $n$ is odd, btw.}
Since the complex realization of $\KO$ is $\KO^\topsup$ (see e.g. \cite[Lemma 2.13]{ananyevskiy2017very}), we may as well show: if $\alpha: \KO_2^\comp \to \KO_2^\comp \in \SH$ is a homotopy unital map, then $(\alpha-1)(H_-) \equiv 0 \pmod 4$.
One has $[\KO_2^\comp, \KO_2^\comp]_{\SH} \wequi \Z_2^\comp\fpsr{T}$, where $T=\adamspsi^5-1$ (see e.g. \cite[Proposition 3.7]{hahn2007iwasawa}).
We may thus (formally) write $\alpha = \sum_{i \ge 0} a_i T^i$ for certain $a_i \in \Z_2^\comp$.
Since $\alpha$ is unital we must have $1 = \alpha(1) = a_0$.
Since $\adamspsi^5(H_-) = 25H_-$ (e.g. by comparison with the Adams action on $\KU_2^\comp$) we get $T(H_-) = 24H_- \equiv 0 \pmod 4$.
This implies the claim and concludes the proof.
\end{proof}

\subsubsection{}
For a scheme $X$ over $\Z[1/2]$, put $\GW^\pm(X) = \KO^0(X) \oplus \KSp^0(X)$.
This is a commutative, $\Z/2$-graded $\lambda$-ring with $\lambda$-operations given by exterior powers of vector bundles \cite[Theorem 4.2.4]{FH-adams}.
We denote the associated Adams operations by $\adamspsi^n_\geo$.
Recall the stable operations $\adamspsi^n_\FH: \KO[1/n] \to \KO[1/n]$ from \cite[\S5.2]{FH-adams}.
\begin{lemma} \label{lemm:adamspsi-FH-action}
For $n$ odd, we have \[ \pi_0\Omega^\infty(\adamspsi^n_\FH) = \adamspsi^n_\geo|_{\KO^0(\ph)[1/n]} \] and \[ \pi_0\Omega^{\infty+4,2}(\adamspsi^n_\FH) = \lra{(-1)^{n(n-1)/2}}n n_\epsilon \adamspsi^n_\geo|_{\KSp^0(\ph)[1/n]}. \]
Moreover $\adamspsi^n_\FH(\beta) = n^2 n^2_\epsilon \beta$.
\end{lemma}
\begin{proof}
Put $\alpha_i = \pi_0\Map(\Sigma^{4i,2i} \ph, \adamspsi^n_\FH)$, $\omega = \lra{(-1)^{n(n-1)/2}}n n_\epsilon$.
By construction, for $X \in \Sm_S$ and $E \in \KO^0(X)$ or $E \in \KSp^0(X)$ we have \begin{equation} \label{eq:adams-FH-ind} \alpha_{i}((H(1)-H_-) \boxtimes E) = (H(1)-H_-) \boxtimes \alpha_{i+1}(E) \in \GW^\pm(\HP^1 \wedge X_+). \end{equation}
Furthermore by construction \cite[p. 16]{FH-adams}, the map $\alpha_{-1}$ is given by $\omega^{-1} \adamspsi^n_\geo$.
Since $\adamspsi^n_\geo(H(1)-H_-) = \omega(H(1)-H_-)$ \cite[Lemma 5.1.4]{FH-adams}, and $\adamspsi^n_\geo$ preserves products in all of $\GW^\pm(X)$, \eqref{eq:adams-FH-ind} inductively implies that $\alpha_i = \omega^i \adamspsi^n_\geo$.
All claims follow.
\end{proof}

\subsubsection{}
So far no satisfactory spectrum $\KO \in \SH(\Z)$ has been constructed.
We offer a replacement $\KO' \in \SH(\Z)$ which has enough good properties for our purposes.
To construct it, we form a pullback square in $\CAlg(\SH(\Z))$ as follows
\begin{equation} \label{eq:define-KO'}
\begin{CD}
\KO'_\Z @>>> \KO^\topsup[1/2] \\
@VVV           @VVV           \\
\KGL_\Z^{hC_2} @>>> (\KO^\topsup)_2^\comp[1/2].
\end{CD}
\end{equation}
We view $\KO^\topsup[1/2]$ as an element of $\SH[1/2] \wequi \SH(\Z)[1/2]^-$, and similarly for $(\KO^\topsup)_2^\comp$.
Note that the canonical functor $\SH(\Z) \to \SH(\Z[1/2])$ preserves limits and induces an equivalence on $(\ph) \otimes \Z[1/2]^-$.
In order to define the map $\KGL_\Z^{hC_2} \to (\KO^\topsup)_2^\comp[1/2] \in \SH(\Z)$, it thus suffices to produce the same map over $\Z[1/2]$.
There we have $\KGL^{hC_2} \wequi \KO^+$ and $(\KO^\topsup)_2^\comp[1/2] \wequi \KO_2^\comp[1/2]^- \wequi (\KO^+)_2^\comp[1/2]^-$ (see Lemma \ref{lemm:KO+}, Corollary \ref{cor:mot-fracture} and Lemma \ref{lemm:KO-real-realization}), yielding an evident map.
If $S$ is any scheme, we define $\KO'_S \in \CAlg(\SH(S))$ by pullback from $\Z$.
We stress that we do not know if this is a reasonable definition in general.

\begin{remark} \label{rmk:construct-KO'}
$\KO'_\Z$ is essentially constructed in such a way as to recover the $2$-adic fracture square (in the sense of Corollary \ref{cor:mot-fracture}) for the conjectural object $\KO_\Z$.
Indeed by construction, $(\KO'_\Z)_2^\comp \wequi (\KGL_\Z^{hC_2})_2^\comp$, which one expects to coincide with $(\KO_\Z)_2^\comp$ by the homotopy fixed point theorem; relatedly one expects that $\KGL_\Z^{hC_2} \wequi \KO_\Z^+$ (as in Lemma \ref{lemm:KO+}).
Then we glue in the minus part, which should not differ between $\Z$ and $\Z[1/2]$, to obtain an object with the expected homotopy groups integrally.
\end{remark}

Here are some useful properties of $\KO'$.
For part (2) below we require some difficult results about Hermitian $K$-theory over schemes in which $2$ is not invertible from \cite{GW-III}.
\begin{lemma} \label{lemm:htpy-fixed-KGL-Z}
\begin{enumerate}
\item For any scheme $S$, the ring spectrum $\KO'_S \in \SH(S)$ is $\SL$-oriented.
\item For $S=\Spec(\Z)$ we have \[ \pi_{4*,2*}\KO'_\Z \wequi \GW(\Z)[\beta, \beta^{-1}, H_-]/(\I(\Z)H_-, H_-^2-2h\beta), \] and also \[ \pi_{4*,2*}(\KGL_\Z^{hC_2})_2^\comp \wequi \GW(\Z)_2^\comp[\beta, \beta^{-1}, H_-]/(\I(\Z)H_-, H_-^2-2h\beta), \]
\item If $1/2 \in S$ then $\KO'_S \wequi \KO_S$.
\end{enumerate}
\end{lemma}
\begin{proof}\NB{These are a bit sketchy...}
(3) Since $\KO$ is stable under base change among schemes containing $1/2$, we may assume that $S = \Spec(\Z[1/2])$.
Since $\SH(\Z) \to \SH(\Z[1/2])$ preserves limits, we know that $\KO'_{\Z[1/2]}$ sits in a similar cartesian square as $\KO'_\Z$.
In fact, by construction (see Remark \ref{rmk:construct-KO'}), the pulled back square identifies with the $2$-adic fracture square for $\KO_{\Z[1/2]}$, whence the result.

(1) We may assume that $S=\Spec(\Z)$.
The construction of the ring map $\MSL \to \KO \to \KGL$ in \cite[Corollary B.3]{bachmann-euler} shows that over any base, $\MSL \to \KGL$ refines to a $C_2$-equivariant map (for the trivial action on $\MSL$).
We thus obtain a ring map $\MSL \to \KGL^{hC_2}_{\Z}$.
We wish to patch this together with a ring map $\MSL \to \KO^\topsup[1/2]$, using the defining pullback square for $\KO'_\Z$.
As in the construction of the pullback square, in order to map into the minus part we may as well base change to $\Z[1/2]$; in this case we already have an orientation $\MSL \to \KO \in \SH(\Z[1/2])$ and the result follows.

(2) We first prove the second statement, i.e. we compute $A = \pi_{4*,2*}((\KGL^{hC_2})_2^\comp)$.
We have $\map(X, \KGL) \wequi \K(\Perf_X)$, for any $X \in \Sm_\Z$.
The $C_2$-action on $\KGL$ corresponds to the action on $\Perf_X$ by dualization $E \mapsto \Hom(E, \scr O_X)$.
The equivalence $\Omega_{\P^1}\K \wequi \K$ is implemented by the pushforward $i_*: \K(X) \to \K(X \times \P^1)$.
Since $i_*: \Perf_X \to \Perf_{X \times \P^1}$ is duality preserving if on $\Perf_{X \times \P^1}$ we use the $n$-shifted duality and on $\Perf_X$ the $(n+1)$-shifted duality, we find that for $n \ge 0$ the duality on $\map(\Sigma^{2n,n} X, \KGL) \wequi \K(\Perf_X)$ is the $n$-shifted one.
Iterated tensoring with $\scr O_X[1]$ induces a $C_2$-equivariant automorphism of $\Perf_X$, intertwining the $4n$-shifted duality and the usual one, or the $(4n+2)$-shifted duality and the negative of the usual one \cite[Proposition 7]{schlichting2010mayer}.
Denote by $\beta \in \pi_{8,4}(\KGL^{hC_2})$ the element corresponding to $1 \in \pi_{0,0}(\KGL^{hC_2})$ under the induced equivalence; in other words $\beta$ is represented by the perfect complex $\scr O[2]$ (with a certain canonical $4$-shifted duality naively expressed as $\Hom(\scr O[2], \scr O[4]) \wequi \Hom(\scr O, \scr O[2]) \wequi \scr O[2]$).
Multiplication by $\beta$ is an automorphism of $\KGL^{hC_2}$ (indeed this can be checked before taking fixed points, and $\beta$ corresponds to an automorphism of $\Perf_X$, and hence of $\KGL$).
%Noting that $i_* \Sigma \scr O \in \Perf_{\P^1}$ represents the Bott element $-\beta_\KGL \in \pi_{2,1}(\KGL)$\footnote{Indeed the exact sequence $\scr O(-1) \to \scr O \to i_* \scr O$ shows that $[i_*\Sigma \scr O] = -([\scr O] - [\scr O(-1)]) = [\scr O(-1)] - 1$, which corresponds to $-\beta_\KGL$ by the argument from Lemma \ref{lemm:action-KGL}.}, we deduce that there is an element $\beta \in \pi_{8,4} \KGL^{hC_2}$, represented by $\Sigma^4 \scr O$, which lifts $\beta_\KGL^4$ and such that multiplication by it induces an automorphism of $\KGL^{hC_2}$ (since it induces an automorphism before taking fixed points).
It follows that $A$ is $(8,4)$-periodic.
We shall show that $A_{0,0} = \GW(\Z)_2^\comp$ and $A_{4,2} = \Z_2^\comp$, in such a way that the map \[ \alpha: A \to \pi_{4*,2*}((\KGL_{\Z[1/2]}^{hC_2})_2^\comp) \wequi (\KO^{4*,2*}_{\Z[1/2]})_2^\comp \] is the canonical one in degrees $(0,0)$ and $(4,2)$.
Since $\alpha(\beta) = \beta$ by construction, this implies that the map $\alpha$ is an injection, and we can check all the desired relations over $\Z[1/2]$, where we know them to be true.

Thus it remains to determine $A_{0,0}$ and $A_{4,2}$.
Since we have determined above the dualities on $\map(\1, \KGL) \wequi \K(\Z)$ and $\map(S^{4,2}, \KGL) \wequi \K(\Z)$ to be the usual one and its negative, the desired result is an immediate consequence of \cite[Theorem 2]{GW-III}.

Now we prove the first statement.
Note that \[ \KGL^{hC_2}_\Z[1/2]^- \wequi \KGL^{hC_2}_{\Z[1/2]}[1/2]^- \stackrel{L.\ref{lemm:KO+}}{\wequi} \KO_{\Z[1/2]}^+[1/2]^- \wequi (\KO_{\Z[1/2]})_2^\comp[1/2]^- \stackrel{L.\ref{lemm:KO-real-realization}}{\wequi} (\KO^\topsup)_2^\comp[1/2], \] where the middle equivalence is a direct consequence of Definition \ref{def:generalized-plus}.
Similarly \[ (\KGL^{hC_2}_\Z)_2^\comp[1/2]^- \wequi (\KO^\topsup)_2^\comp[1/2]. \]
On the other hand \[ \KGL^{hC_2}_\Z[1/2]/\eta \wequi (\KGL_\Z/\eta)^{hC_2}[1/2] \wequi \KGL[1/2] \wequi \KGL_\Z[1/2]^{hC_2}/\eta \] (see the proof of Lemma \ref{lemm:heard}), and so $\KGL^{hC_2}_\Z[1/2]^+ \wequi \KGL_\Z[1/2]^{hC_2}$, and similarly $(\KGL^{hC_2}_\Z)_2^\comp[1/2]^+ \wequi (\KGL_\Z)_2^\comp[1/2]^{hC_2}$.
Thus the defining square \eqref{eq:define-KO'} yields \[ \KO'_\Z[1/2]^- \wequi \KO^\topsup[1/2] \quad\text{and}\quad \KO'_\Z[1/2]^+ \wequi \KGL_\Z[1/2]^{hC_2}; \] of course also \[ (\KO'_\Z)_2^\comp \wequi (\KGL^{hC_2})_2^\comp \] and thus \[ (\KO'_\Z)_2^\comp[1/2]^- \wequi (\KO^\topsup)_2^\comp[1/2] \quad\text{and}\quad (\KO'_\Z)_2^\comp[1/2]^+ \wequi (\KGL_\Z)_2^\comp[1/2]^{hC_2}. \]
The ordinary $2$-adic fracture square for $\KO'_\Z$ thus yields an exact sequence \begin{gather*} \pi_{4*+1,2*}(\KGL_\Z)_2^\comp[1/2]^{hC_2} \oplus \pi_{4*+1,2*}(\KO^\topsup)_2^\comp[1/2] \to \\ \pi_{4*,2*} \KO'_\Z \to A_{4*,2*} \oplus \pi_{4*,2*} \KGL_\Z[1/2]^{hC_2} \oplus \pi_{4*,2*}\KO^\topsup[1/2] \to A_{4*,2*}[1/2]. \end{gather*}
Since $K_1(\Z) = \Z/2$ is torsion and $\pi_* (\KO^\topsup)_2^\comp[1/2] \wequi \Q_2[\beta^\pm]$ vanishes in odd degrees, the first group in the exact sequence is zero.
The desired result can now be read off since $\pi_{4*,2*} \KGL_\Z[1/2]^{hC_2} \wequi \Z[1/2, \beta_\KGL^{\pm 2}]$ and $\pi_{4*,2*}\KO^\topsup[1/2] \wequi \Z[1/2, \beta^\pm]$.
\end{proof}

\subsubsection{}
We can now compare the Adams operations.
\begin{proposition} \label{prop:FH-comparison}
Let $S$ be a scheme with $1/2 \in S$, and $n$ an odd integer.
The two maps \[ \adamspsi^n, \adamspsi^n_\FH: \KO[1/n] \to \KO[1/n] \in \SH(S) \] are homotopic.
\end{proposition}
\begin{proof}
By definition, the maps are pulled back from $\Spec(\Z[1/2])$, so we may assume that $S = \Spec(\Z[1/2])$.
By Lemma \ref{lemm:adamspsi-FH-action} and Theorem \ref{thm:ko-adams}, both maps act on $\beta$ by multiplication by $n^2n^2_\epsilon$.
Hence by Corollary \ref{cor:KO-ring-end}, it suffices to check that both maps send $H(1)$ to an element of $\KO^{4,2}(\HP^\infty_\Z)[1/n] \subset \KO^{4,2}(\HP^\infty_{\Z[1/2]})[1/n]$.

First we treat $\adamspsi^n_\FH$.
Lemma \ref{lemm:adamspsi-FH-action} shows that up to a factor in $\GW(\Z)[1/n]^\times$, $\adamspsi^n_\FH(H(1)) = \adamspsi^n_\geo(H(1))$.
By definition, this is a linear combination with integer coefficients of exterior powers of $H(1)$, which are clearly already defined over $\Z$.

For $\adamspsi^n$, note that it suffices to show that the subring $\KO^{4*,2*}(\HP^\infty_\Z)_2^\comp$ is preserved.
Lemma \ref{lemm:htpy-fixed-KGL-Z} (together with the homotopy fixed point theorem over $\Z[1/2]$) implies that this ring coincides with the image of \[ ((\KGL_2^\comp)^{hC_2})^{4*,2*}(\HP^\infty_\Z) \to ((\KGL_2^\comp)^{hC_2})^{4*,2*}(\HP^\infty_{\Z[1/2]}) \wequi (\KO_2^\comp)^{4*,2*}(\HP^\infty_{\Z[1/2]}). \]
The result follows since $(\adamspsi^n)_2^\comp = \adamspsi^n_\OSnaith$ by definition, and the latter is already defined over $\Z$.
\end{proof}

\section{Cobordism spectra} \label{sec:cobordism}
\subsection{Summary}
We will be using the algebraic cobordism spectra $\MSL$ and $\MSp$ \cite[Example 16.22]{bachmann-norms} \cite{panin2010algebraic}.
They can be constructed explicitly out of the Thom spaces of tautological bundles on special linear (respectively quaternionic) Grassmannians; we review this in \S\ref{subsec:real-realization}.
In particular there is a canonical map \begin{equation}\label{eq:HP-MSp} \Sigma^{-4,-2} \Sigma^\infty_+ \HP^\infty \to \Sigma^{-4,-2} \Th(\gamma) \to \MSp, \end{equation} where $\gamma$ denotes the tautological bundle on $\HP^\infty$.\NB{In fact $\Th(\gamma) \wequi \HP^\infty$ (no plus!)}

There are notions of $\SL$-oriented and $\Sp$-oriented multiplicative cohomology theories \cite[Definitions 5.1 and 8.1]{panin2010algebraic}.
We will only deal with cohomology theories represented by homotopy commutative ring spectra $A \in \CAlg(\h\SH(S))$.
In this situation one way of exhibiting an $\SL$-orientation (respectively $\Sp$-orientation) is to exhibit a homotopy ring map $\MSL \to A$ (respectively $\MSp \to A$) (see e.g. \cite[Proposition 4.13]{bachmann-euler} or \cite[Theorems 5.5 and 13.2]{panin2010algebraic}).

If $A$ is $\Sp$-oriented and $V$ is a symplectic vector bundle bundle on $X \in \Sm_S$, then we obtain the \emph{Borel classes} $b_i(V) \in A^{4i,2i}(X)$ \cite[Definition 11.5]{panin2010algebraic}.
One has \cite[Theorem 8.2]{panin2010algebraic} \[ A^{**}(\HP^\infty) \wequi A^{**}\fpsr{b_1(\gamma)}. \]
Similarly if $A$ is $\SL$-oriented and $V$ is an oriented vector bundle, then we obtain the \emph{Pontryagin classes} $p_i(V) \in A^{8i,4i}(X)$ \cite[Definition 19]{ananyevskiy2015special}.
When dealing with $\eta$-periodic cohomology theories, we will implicitly shift everything into weight zero (by multiplying by powers of $\eta$); so then we write \[ b_i(V) \in A^{2i}(X) \quad\text{and}\quad p_i(V) \in A^{4i}(X). \]

The cohomology of $\MSp$ and $\MSL$ have been worked out in \cite[Theorem 13.1]{panin2010algebraic} \cite[Theorem 10]{ananyevskiy2015special}\footnote{In \cite{ananyevskiy2015special}, the author works over perfect fields of characteristic $\ne 2$ only. However, all results hold over general base schemes [personal communication]. Indeed the only place where the assumption on the base is used is in Lemma 12. This holds over general bases, as can be seen as follows. It suffices to show that the endomorphisms $[x:y] \mapsto [y:x]$ and $[x:y] \mapsto [-x:y]$ of $\P^1$ are $\A^1$-homotopic. Since both have determinant $-1$, and $\SL_2(\Z)$ is $\A^1$-connected (being generated by elementary matrices), the result follows.}.
We perform the straightforward dualization:
\begin{theorem} \label{thm:cobordism-MSL-summary}
Let $A$ be a cohomology theory (i.e. $A \in \CAlg(\h\SH(S))$).
\begin{enumerate}
\item Suppose that $A$ is $\Sp$-oriented.
  The Kronecker pairing $A_{**}(\HP^n) \otimes A^{**}(\HP^n) \to A_{**}$ is perfect.
  Denote by $\{\beta_i\}_{i=0}^n \in A_{**}(\HP^n)$ the dual basis to $\{b_1(\gamma)^i\}_{i=0}^n \in A^{**}(\HP^n)$.
  The $\beta_i$ are compatible with $\HP^n \hookrightarrow \HP^{n+1}$; denote by $\beta_i \in A_{**}(\HP^\infty)$ their common images.
  Write $e_i \in A_{4i,2i}(\MSp)$ for the image of $\beta_{i+1}$ under the map induced by \eqref{eq:HP-MSp}.
  Then \[ A_{**}(\MSp) \wequi A_{**}[e_0, e_1, e_2, \dots]/(e_0 - 1). \]
\item Suppose that $A$ is $\eta$-periodic and $\SL$-oriented.
  The canonical map $\MSp \to \MSL$ annihilates $e_i$ for $i$ odd and \[ A_* \MSL \wequi A_*[e_0, e_2, \dots]/(e_0-1). \]
\end{enumerate}
\end{theorem}

\begin{example} \label{ex:SL-orientations}
$\KO$ admits a ring map from $\MSL$ (see e.g. \cite[Corollary B.3]{bachmann-euler}), and hence is $\SL$-oriented (whence also $\Sp$-oriented).
Since $\MSL$ is very effective, $\ko = \tilde f_0 \KO$ is also $\SL$-oriented.
Similarly so are $\kw = \ko[\eta^{-1}]$ and $\KW = \KO[\eta^{-1}]$, since they receive ring maps from $\ko$.
\end{example}
If $A$ denotes any of the above theories, then by Remark \ref{rmk:adams-kw} we obtain Adams operations on $A$-homology.
We are particularly interested in the case $S=\Spec(k)$ and $A = \kw_{(2)}$.
Note that if $\W(k) = \F_2$ then $\kw \wequi \kw_{(2)}$ and so all $\adamspsi^n$ act on $\kw_* E$ for any $E \in \SH(k)$ and $n$ odd.
\begin{proposition} \label{prop:adams-cobordism-summary}
Let $S=\Spec(k)$, where $k$ is a field (of characteristic $\ne 2$).
Suppose that $\W(k) = \F_2$.
Then \[ \adamspsi^3(e_{2i}) \in e_{2i} + \beta e_{2i-2} + \beta^2 \kw_* \MSL. \]
In fact \[ \adamspsi^3(e_i) \in e_i + \beta^2 \kw_* \MSp + \begin{cases} \beta e_{i-2} & i \text{ even} \\ 0 & i \text{ odd} \end{cases}. \]
\end{proposition}

We also record the following well-known facts.
\begin{lemma} \label{lemm:cobordism-realisation}
We have $r_\R(\MSp) \wequi \MU$ and $r_\R(\MSL) \wequi \MSO$.
\end{lemma}

\subsection{Real realization} \label{subsec:real-realization}
For an algebraic group $G$, write $BG$ for the stack of $G$-torsors, viewed as a sheaf of groupoids on $\Sm_S$.\footnote{There might be some concern here which topology we are using, but we shall only apply the following discussion to special groups, where all torsors are Zariski-locally trivial.}
Suppose given furthermore an inclusion $G \to \GL_n$.
Then we obtain a map $BG \to B\GL_n \to \K^\circ$ corresponding to the tautological virtual bundle on $B\GL_n$, and consequently a Thom spectrum $MG \in \SH(S)$ \cite[\S16.2]{bachmann-norms}.
For a $G$-torsor $E$, write $V(E) = E \times_G \A^n$ for the associated vector bundle, where $G$ acts on $\A^n$ via the embedding into $\GL_n$.
\begin{lemma} \label{lemm:thom-space-tautology}
There is a cofibration sequence\footnote{Again, the $G$-orbits should be taken as sheaves in some sufficiently strong topology.} \[ \Sigma^{\infty-2n,n}_+ (\A^n \setminus 0)_{hG} \to \Sigma^{\infty-2n,n}_+ BG \to MG. \]
\end{lemma}
\begin{proof}
By definition, \[ MG \wequi \colim_{E: X \to BG} \Th(V(E)-n), \] where the colimit is over all smooth schemes with a map to $BG$, i.e. smooth schemes and $G$-torsors on them, and $\Th(V(E))$ is the associated Thom spectrum.
In other words there is a cofiber sequence $\Sigma^{\infty-2n,n}_+ V(E) \setminus 0 \to \Sigma^{\infty-2n,n} X_+ \to \Th(V(E))$ and hence \[ \colim_{E: X \to BG} \Sigma^{\infty-2n,n}_+ V(E) \setminus 0 \to \Sigma^{\infty-2n,n}_+ BG \to MG. \]
By universality of colimits in $\infty$-topoi \cite[Theorem 6.1.0.6]{lurie-htt}, passage to $G$-orbits preserves pullbacks.\NB{details?}
It follows that there is a cartesian square
\begin{equation*}
\begin{CD}
(E \times (\A^n \setminus 0))_{hG} \wequi V(E) \setminus 0 @>>> (\A^n \setminus 0)_{hG} \\
@VVV              @VpVV \\
E_{hG} \wequi X     @>>> *_{hG} \wequi BG.
\end{CD}
\end{equation*}
Thus $\colim_{E: X \to BG} V(E) \setminus 0 \wequi (\A^n \setminus 0)_{hG}$, by universality of colimits again (this time in the presheaf $\infty$-topos).
The result follows.
\end{proof}

If $G=\GL_n$, the association $E \mapsto V(E)$ induces an equivalence between $G$-torsors and vector bundles of rank $n$.
Similarly if $G=\SL_n$, we get an equivalence between $G$-torsors and oriented vector bundles (i.e. carrying a trivialization of the determinant), and if $G=\Sp_n$ then we obtain an equivalence with symplectic bundles of rank $2n$ (i.e. carrying a non-degenerate, alternating bilinear form).
See e.g. \cite[\S3]{asok2015affine} for this.
The Grassmannian variety $\Gr(n, k)$ represents the functor of $n$-dimensional vector subbundles of $\scr O^k$; similarly $\SGr(n, k)$ represents the functor of $n$-dimensional vector subbundles of $\scr O^k$ together with an orientation, and $\HGr(n,k)$ represents the functor of $2n$-dimensional vector subbundles $V$ of $\scr O^{2k}$ such that the restriction of the canonical alternating form on $\scr O^{2k}$ to $V$ remains non-degenerate.
There are thus canonical $\GL_n, \SL_n$ and $\Sp_n$ torsors on $\Gr(n,k), \SGr(n,k)$ and $\HGr(n,k)$, respectively, inducing maps \begin{equation} \label{eq:Gr-BGL} \Gr(n, \infty) \to B\GL_n, \quad \SGr(n, \infty) \to B\SL_n \quad\text{and}\quad \HGr(n,\infty) \to B\Sp_n. \end{equation}
Each of these maps is well-known to be a motivic equivalence; see e.g. \cite[Proposition 2.6]{A1-homotopy-theory}, \cite[proof of Theorem 4.1.1]{asok2015affine}, \cite[proof of Theorem 8.2]{panin2010motivic}.

We can use this to connect to more standard definitions of Thom spectra.
\begin{lemma} \label{lemm:thom-space-classical}
We have motivic equivalences \[ M\GL_n \wequi \Sigma^{\infty-2n,n}\Th(\gamma_n^\GL), \quad M\SL_n \wequi \Sigma^{\infty-2n,n}\Th(\gamma_n^\SL) \quad\text{and}\quad M\Sp_n \wequi \Sigma^{\infty-4n,2n}\Th(\gamma_n^\Sp), \] where $\gamma_n^\GL \to \Gr(n,\infty)$ (respectively $\gamma_n^\SL \to \SGr(n,\infty), \gamma_n^\Sp \to \HGr(n,\infty)$) denotes the tautological virtual bundle.
\end{lemma}
\begin{proof}
We need to prove that the motivic Thom spectrum functor inverts the motivic equivalence $\Gr_n \to B\GL_n$, and similarly for $\SL_n, \Sp_n$.
This follows from \cite[Proposition 16.9 and Remark 16.11]{bachmann-norms}.
\end{proof}

By definition \cite[Example 16.22]{bachmann-norms}, the ($\scr E_\infty$-ring maps of) spectra \begin{equation} \label{eq:thom-spectra} \MSp \to \MSL \to \MGL \end{equation} are obtained by applying the motivic Thom spectrum formalism to the maps $\K^{\Sp\circ} \to \K^{\SL\circ} \to \K^{\circ}$.
Since $\K^{\Sp\circ} \stackrel{\A^1}{\wequi} \colim_n B\Sp_n$ (argue as in the discussion just before \cite[Theorem 16.13]{bachmann-norms}), we find (using Lemma \ref{lemm:thom-space-classical} and \cite[Proposition 16.9 and Remark 16.11]{bachmann-norms}) that \[ \MSp \wequi \colim_n M\Sp_n \wequi \colim_n \Sigma^{\infty - 4n,2n}\Th(\gamma_n^\Sp), \] as expected.
Similarly \[ \MSL \wequi \colim_n \Sigma^{\infty - 2n,n}\Th(\gamma_n^\SL) \quad\text{and}\quad \MGL \wequi \colim_n \Sigma^{\infty - 2n,n}\Th(\gamma_n^\GL). \]
\begin{corollary}
We have $r_\R(\MSp) \wequi \MU$ and $r_\R(\MSL) \wequi \MSO$.
In fact $r_\R(\MSp_n) \wequi \MU_n$ and $r_\R(\MSL_n) \wequi \MSO_n$.
\end{corollary}
\begin{proof}
It suffices to show the ``in fact'' part.
By Lemma \ref{lemm:thom-space-tautology}, we have \[ \MSL_n \wequi \Sigma^{\infty-2n,n}\cof(T_n \to G_n), \] where $T_n = (\A^n \setminus 0)_{h\SL_n}$ and $G_n = *_{h\SL_n}$.
Since $\SL_n$ is a special group \cite[\S4.4.b]{serre1958espaces}, we can take the homotopy orbits in the Zariski topology, and hence (since Zariski equivalences are motivic equivalences), we can just take the homotopy orbits in the category of motivic spaces.
In other words we obtain \[ T_n \wequi \colim_{m \in \Delta^\op} \SL_n^{\times m} \times (\A^n \setminus 0) \in \Spc(S)_*. \]
Since $r_\R$ preserves colimits and finite products, we find that \[ r_\R(T_n) \wequi \colim_{m \in \Delta^\op} \SL_n(\R)^{\times m} \times (\R^n \setminus 0) \wequi (\R^n \setminus 0)_{h\SL_n(\R)}. \]
The inclusion $SO_n \to \SL_n(\R)$ is a homotopy equivalence \cite[\S E.5]{hall2015lie}, so that this is the same as $(\R^n \setminus 0)_{hSO_n}$.
Similarly we find that $r_\R(G_n) \wequi BSO_n$, whence \[ r_\R(\MSL_n) \wequi \Sigma^{-n}\cof((\R^n \setminus 0)_{hSO_n} \to BSO_n) \wequi \MSO_n, \] as desired.
The argument for $\MSp$ is similar, using that $\Sp_n$ is special \cite[\S4.4.c]{serre1958espaces} and $\Sp_n(\R) \wequi U_n$ \cite[\S4.4]{Arnold2001}.
\end{proof}

\subsection{Homology}
\subsubsection{}
We recall some well-known facts about duality in homology and cohomology theories represented by (motivic) spectra.

Let $A \in \CAlg(\h\SH(S))$ and $X \in \SH(S)$.
Then we have the \emph{Kronecker pairing} \begin{gather*} A_{**}(X) \otimes_{A_{**}} A^{**}(X) \to A_{**} \\ (f: \Sigma^{**} \1 \to A \wedge X) \otimes (g: X \to \Sigma^{**} A) \mapsto (\Sigma^{**} \1 \xrightarrow{f} A \wedge X \xrightarrow{\id \wedge g} A \wedge A \to A). \end{gather*}
This is easily seen to be $A_{**}$-bilinear.
\begin{lemma} \label{lemm:perfect-pairing}
If $X$ is cellular and strongly dualizable, and either $A_{**}X$ or $A^{**}X$ is flat over $A$, then the Kronecker pairing is perfect.\NB{And so in particular both $A_{**}X$ and $A^{**}X$ are flat.}
\end{lemma}
\begin{proof}\discuss{Suggestions?}
Write $DX$ for the dual and put $\otimes := \otimes_{A_{**}}$.
Replacing $X$ by $DX$ if necessary, we may assume that $A^{**}X$ is flat.
Then for any cellular object $Y$ we get $A_{**}(DX \wedge Y) \wequi A_{**}(DX) \otimes A_{**}(Y)$; indeed this holds for spheres by construction and is a natural transformation of homological functors preserving filtered colimits (here we use that $A^{**}X$ is flat).

Let $u: \1 \to DX \wedge X$ and $c: X \wedge DX \to \1$ be the unit and co-unit of the strong duality between $X$ and $DX$.
Since $X$ is cellular, $A_{**}(DX \wedge X) \wequi A_{**}(DX) \otimes A_{**}(X)$, and hence $A_{**}(u), A_{**}(c)$ define maps of the correct shape to exhibit a strong duality between $A_{**}(X)$ and $A_{**}(DX)$.
Consider the following diagram
\begin{equation*}
\begin{tikzcd}[column sep=tiny]
A_{**}(X) \otimes A_{**} \ar[dd, "\wequi" swap] \ar[dr, "\id \otimes A_{**}(u)"] & A_{**}(X) \otimes A_{**}(DX) \otimes A_{**}(X) \ar[d, shift right=10, "\wequi" swap] \ar[d, shift left=10, "\wequi"] & A_{**} \otimes A_{**}(X) \ar[dd, "\wequi"] \\
     & A_{**}(X) \otimes A_{**}(DX \wedge X) \quad\quad A_{**}(X \wedge DX) \otimes A_{**}(X) \ar[ur, "A_{**}(c) \otimes \id"] \ar[d, shift right=10] \ar[d, shift left=10] \\
A_{**}(X \wedge \1) \ar[r, "A_{**}(\id \wedge u)"] & A_{**}(X \wedge DX \wedge X) \ar[r, "A_{**}(c \wedge \id)"] & A_{**}(\1 \wedge X).
\end{tikzcd}
\end{equation*}
The vertical maps are lax monoidal structure maps; the equivalence $A_{**}(DX \wedge X) \wequi A_{**}(DX) \otimes A_{**}(X)$ has already been established.
The diagram commutes because $A_{**}$ is lax symmetric monoidal.
Up to suppressing tensoring with the unit, the bottom horizontal composite is the identity (by definition of $u, c$ exhibiting a strong duality).
It follows that the top horizontal composite (inverting the middle vertical maps) $(\id \otimes A_{**}(c)) \circ (A_{**}(u) \otimes \id)$ is the identity.
A similar diagram shows that $(A_{**}(c) \otimes \id) \circ (\id \otimes A_{**}(u))$ is the identity.
Thus $A_{**}(u), A_{**}(c)$ exhibit a strong duality between $A_{**}(X)$ and $A_{**}(DX)$.
This was to be shown.
\end{proof}

In preparation for later, we also observe the following.
\begin{lemma} \label{lemm:kronecker-adams}
Let $\adamspsi: A \to A$ be a ring automorphism.
Then the Kronecker pairing satisfies \[ \lra{\adamspsi x, y} = \adamspsi \lra{x, \adamspsi^{-1} y}. \]
\end{lemma}
\begin{proof}
The commutative\NB{right hand square commutes b/c $\adamspsi^{-1}$ is a ring map} diagram (in which suspensions have been suppressed)
\begin{equation*}
\begin{CD}
\1 @>x>> X \wedge A @>{y \wedge \id}>> A \wedge A \\
@.          @V{\id \wedge \adamspsi}VV     @V{\id \wedge \adamspsi}VV \\
        @.   X \wedge A @>{y \wedge \id}>> A \wedge A @>m>> A \\
@.             @.                   @V{\adamspsi^{-1} \wedge \adamspsi^{-1}}VV  @V{\adamspsi^{-1}}VV \\
        @.               @.                A \wedge A @>m>> A
\end{CD}
\end{equation*}
shows\NB{paths rdrrd vs rrddr} that $\adamspsi^{-1}\lra{\adamspsi x, y} = \lra{x, \adamspsi^{-1}y}$.
The result follows.
\end{proof}

Let \[ X_0 \to X_1 \to \dots \in \SH(S) \] be a directed system and put $X = \colim_i X_i$.
Suppose that $\limone_i A^{**}(X_i) = 0$.
Then \[ A^{**}(X) \wequi \lim_i A^{**}(X_i), \] and we can give this the \emph{inverse limit topology} (i.e. give each $A^{**}(X_i)$ the discrete topology and take the limit in bigraded topological abelian groups).
\begin{corollary} \label{cor:duality}
Assume in addition that each $X_i$ is strongly dualizable and cellular, and $A^{**}(X_i)$ or $A_{**}(X_i)$ is flat.
Then \begin{equation}\label{eq:cont-coh} A^{**}(X) \wequi \Hom_{A_{**}}(A_{**}(X), A_{**}) \quad\text{and}\quad A_{**}(X) \wequi \Hom_{A_{**},c}(A^{**}(X), A_{**}). \end{equation}
Here $\Hom_{A_{**},c}$ means continuous homomorphisms (for the discrete topology on the target and the inverse limit topology on the source).
\end{corollary}
\begin{proof}
For the first claim we compute \begin{gather*} A^{**}(X) \wequi \lim_i A^{**}(X_i) \wequi \lim_i \Hom_{A_{**}}(A_{**}(X_i), A_{**}) \\\wequi \Hom_{A_{**}}(\colim_i A_{**}(X_i), A_{**}) \wequi \Hom_{A_{**}}(A_{**}(X), A_{**}), \end{gather*} using the Milnor exact sequence \cite[Proposition VI.2.15]{goerss2009simplicial}, Lemma \ref{lemm:perfect-pairing}, and compactness of the spheres.
For the second claim, note that by definition a basis of open neighborhoods of $0$ in $A^{**}(X)$ is given by $\ker(A^{**}(X) \to A^{**}(X_i))$.
Any continuous homomorphism thus factors through $A^{**}(X_i)$ for some $i$, yielding the formula \[ \Hom_{A_{**},c}(A^{**}(X), A_{**}) \wequi \colim_i \Hom_{A_{**}}(A^{**}(X_i), A_{**}). \]
But also \[ \colim_i \Hom_{A_{**}}(A^{**}(X_i), A_{**}) \wequi \colim_i A_{**}(X_i) \wequi A_{**}(X), \] using Lemma \ref{lemm:perfect-pairing} and compactness of the spheres again.
\end{proof}

\subsubsection{}
We recall some standard facts about oriented ring spectra and Thom isomorphisms.
Given a homotopy ring spectrum $E$ over $S$ and an algebraic group $G \to \GL_n$, a $G$-orientation of $E$ consists of a choice of \emph{Thom class} $t(V) \in E^{2n,n}(\Th(V))$ for every $G$-bundle $V$ on a smooth $S$-scheme, satisfying certain naturality and normalization axioms (see e.g. \cite[Definition 3.3]{ananyevskiy2019sl}).
\begin{lemma} \label{lemm:homology-thom}
Let $E$ be $G$-oriented and $V$ a Nisnevich locally trivial $G$-bundle on $X \in \Sm_S$ (e.g. $G$ special and $V$ arbitrary).
Then the map \[ E \wedge \Th(V) \xrightarrow{\id_E \wedge \Delta} E \wedge \Th(V) \wedge X_+ \xrightarrow{\id_E \wedge t(V) \wedge \id_X} E \wedge \Sigma^{2n,n} E \wedge X_+ \xrightarrow{m} E \wedge \Sigma^{2n,n} X_+ \] is an equivalence.
\end{lemma}
We call this equivalence the \emph{(homological) Thom isomorphism}; it induces in particular $t: E_{**}\Th(V) \wequi E_{*-2n,*-n}(X)$.
\begin{proof}
By the smooth projection formula, we may assume that $X=S$.
By Nisnevich separation \cite[Proposition 6.23]{hoyois2016equivariant} and naturality of the Thom class, we may assume that $V$ is trivial.
In this case the map is homotopic to the identity, by definition.
\end{proof}

\subsubsection{}
We now compute the homology of $\MSp$, following the standard topological proof.

Write $\HGr(r,n)$ denote the quaternionic Grassmannian of $2r$-dimensional symplectic subspaces in $2n$-dimensional symplectic space (see e.g. \cite{panin2010quaternionic}).
For example $\HGr(1,n) = \HP^n$.

\begin{lemma}
Suppose that $A$ is $\Sp$-oriented.
We have $A_{**}(\HP^\infty) \wequi A_{**}\{\beta_0, \beta_1, \dots\}$, and the canonical map $\alpha: (\HP^\infty)^n \to \HGr(n, \infty)$ induces \[ A_{**}(\HGr(n, \infty)) \wequi \Sym^n(A_{**}(\HP^\infty)^{\otimes n}). \]
\end{lemma}
\begin{proof}
$\Sigma^\infty_+ \HGr(r,n) \in \SH(S)$ is cellular and strongly dualizable \cite[Proposition 3.1]{rondigs2016cellularity}.
By \cite[Theorem 11.4]{panin2010quaternionic} the maps \[ A^{**}(\HGr(r,n)) \leftarrow A^{**}(\HGr(r,n+1)) \] are surjective, and \begin{equation}\label{eq:coh-HGr} A^{**}(\HGr(r, \infty)) \wequi \lim_n A^{**}(\HGr(r,n)) \wequi A^{**}\fpsr{b_1, \dots, b_r}. \end{equation}
We deduce (using Corollary \ref{cor:duality}) that $A_{**}(\HP^\infty)$ is the topological dual of $A^{**}(\HP^\infty) \wequi A^{**}\fpsr{b_1}$, so that compatible classes $\beta_i$ exist as claimed.
In particular \begin{equation} \label{eq:hoh-HP} A_{**}(\HP^\infty) \wequi A_{**}\{\beta_0, \beta_1, \dots\}. \end{equation}

Since \eqref{eq:coh-HGr} holds over any base, we find that \[ A^{**}((\HP^\infty)^n) \wequi A^{**}\fpsr{a_1, \dots, a_n}, \] where $a_i = b_1(\gamma_i)$, $\gamma_i$ being the tautological bundle on the $i$-th factor $\HP^\infty$.
There is a map\footnote{Recall that $\HGr(n, k)$ represents the functor of $2n$-dimensional symplectic subbundles of the trivial symplectic bundle $\scr O^{2k}$. On $\HP^k$ we thus have $\gamma \hookrightarrow \scr O^{2k}$ whence $\gamma^{\boxplus n} \hookrightarrow \scr O^{2nk}$; this defines a map $(\HP^k)^n \to \HGr(n, nk)$. Now take colimits.} $\alpha: (\HP^\infty)^n \to \HGr(n, \infty)$ such that \[ \alpha^*(\gamma) \wequi \gamma_1 \boxplus \gamma_2 \boxplus \dots \boxplus \gamma_n =: E. \]
By the Cartan formula \cite[Theorem 10.5]{panin2010quaternionic}, the map \[ \alpha^*: A^{**}(\HGr(n,\infty)) \to A^{**}((\HP^\infty)^n) \] is given by \[ b_i \mapsto b_i(E) = \sigma_i(a_1, \dots, a_n), \] where $\sigma_i$ is the $i$-th elementary symmetric polynomial.
The map $\alpha^*$ is thus a split injection onto $(A^{**}((\HP^\infty)^n))^{\Sigma_n}$.
Passage to continuous duals yields $\alpha_*$, which is hence as claimed.
\end{proof}
\begin{remark}
The above result can also be deduced from \cite[Theorem 5.10]{ananyevskiy2017stable}.
\end{remark}

\begin{lemma}\label{lemm:taut-contractible} \NB{ref?}
Let $\gamma$ be the tautological bundle on $\HP^\infty$.
Then $\gamma \setminus 0$ is (motivically) contractible.
\end{lemma}
\begin{proof}
Let $\gamma_n$ be the tautological bundle on $\HP^n$.
Then $\gamma_n \setminus 0 = \Sp_{n}/(G \times \Sp_{n-1})$, where $G \subset \Sp_1$ is the subgroup fixing the first vector.\NB{$\Sp_n$ acts transitively on the set of pairs of a 2d symplectic subspace of $k^n$ and a vector in it: can always extend this data to a symplectic basis. The stabilizer of the standard subspace with first basis vector is as displayed.}
One checks (e.g. using $\Sp_1 = \SL_2$) that $G \wequi \mathbb{G}_a$.
Since all of these groups are special (and the actions are free), the étale quotients are homotopy quotients, and we deduce that $\gamma_n \setminus 0 \wequi \Sp_{n}/\Sp_{n-1}$\NB{$\wequi \A^{2n} \setminus 0$}.
Taking colimits we get $\gamma \setminus 0 \wequi \Sp/\Sp \wequi *$.
\end{proof}

\begin{proof}[Proof of Theorem \ref{thm:cobordism-MSL-summary}(1)]
Consider the motivic space $\BSp \stackrel{\A^1}{\wequi} \colim_n \HGr(n, \infty)$ (see \eqref{eq:Gr-BGL}).
This has an $H$-space structure coming from addition of vector bundles.
The diagram
\begin{equation*}
\begin{CD}
(\HP^\infty)^n @>>> \BSp^n \\
@V{\alpha}VV     @V{m}VV \\
\HGr(n, \infty) @>>> \BSp
\end{CD}
\end{equation*}
commutes (essentially by construction); here the horizontal maps classify the tautological bundles.
The inclusion at the base point $(\HP^\infty)^n \to (\HP^\infty)^{n+1}$ covers the canonical map $\HGr(n, \infty) \to \HGr(n+1,\infty)$ and induces in cohomology the quotient map \[ A^{**}((\HP^\infty)^{n+1}) \wequi A^{**}\fpsr{b}^{\hat\otimes {n+1}} \to A^{**}\fpsr{b}^{\hat\otimes {n}} \wequi A^{**}((\HP^\infty)^{n}). \]
Passing to continuous duals, we deduce that under the identification $A_{**}(\HGr(n,\infty)) \wequi \Sym^n(A_{**}\HP^\infty)$, the map $\HGr(n,\infty) \to \HGr(n+1,\infty)$ corresponds to multiplication by $\beta_0$.
This implies that \[ A_{**}(\BSp) \wequi \colim_n A_{**}(\HGr(n,\infty)) \wequi \colim_n \Sym^n(A_{**}\HP^\infty) \wequi A_{**}[\beta_0, \beta_1, \dots]/(\beta_0-1). \]

The Thom isomorphism (Lemma \ref{lemm:homology-thom}) induces an isomorphism of rings $A_{**} \BSp \xrightarrow{t} A_{**} \MSp$.
It remains to identify the classes $t(\beta_i)$.
Let $\gamma$ be the tautological bundle on $\HP^\infty$.
The defining cofiber sequence of pointed spaces \[ \gamma \setminus 0 \to \gamma \wequi \HP^\infty \xrightarrow{s} \Th(\gamma) \] has the property that $\gamma \setminus 0$ is contractible (Lemma \ref{lemm:taut-contractible}).
This supplies us with an isomorphism \[ s: A_{**}\{\beta_1, \beta_2, \dots\} \wequi \widetilde{A}_{**}(\HP^\infty) \wequi A_{**}(\Th(\gamma)). \]
To conclude the proof, we shall show that $s(\beta_i) = t(\beta_{i-1})$.

To see this, it suffices to show that the composite \[ \tilde A_{**}(\HP^\infty) \stackrel{s}{\wequi} A_{**}(\Th(\gamma)) \stackrel{t}{\wequi} A_{*-4,*-2}(\HP^\infty) \] maps $\beta_i$ to $\beta_{i-1}$.
By construction, its dual is \[ A^{*+4,*+2}(\HP^\infty) \xrightarrow{t(\gamma) \cdot} A^{**}(\Th(\gamma)) \xrightarrow{s} \tilde A^{**}(\HP^\infty), \] i.e. multiplication by $s^*(t(\gamma))$.
By definition, this is $b = b_1(\gamma)$ \cite[Proposition 7.2]{panin2010quaternionic}, and so the claim follows by dualization.
\end{proof}

\subsubsection{}
Write $\SGr(n, k)$ for the special linear Grassmannian varieties \cite[Definition 22]{ananyevskiy2015special}.
\begin{lemma} \label{lemm:SGr-cellular}
$\Sigma^\infty_+ \SGr(n,k) \in \SH(S)$ is cellular and strongly dualizable.
\end{lemma}
\begin{proof}
By definition we have $\SGr(n,k) = \det\gamma_{n,k} \setminus 0$, where $\gamma_{n,k}$ is the tautological bundle on $\Gr(n,k)$.
We thus have a cofiber sequence \[ \SGr(n,k) \to \Gr(n,k) \to \Th(\det \gamma_{n,k}). \]
Since $\Gr(n,k)$ and $\Th(\det \gamma_{n,k}) \wequi \P(\det \gamma_{n,k} \oplus \scr O)/\P(\det \gamma)$ \cite[Proposition 2.17]{A1-homotopy-theory} are strongly dualizable \cite[Proposition 2.4.31]{triangulated-mixed-motives} so is $\SGr(n,k)$.\NB{Or use that for $p$ smooth proper and $E$ dualizable (e.g. invertible, e.g. $\gamma \setminus 0$), $p_\#$ is dualizable with dual $p_*(DE)$.}
For cellularity we use the description $\SGr(n,k) = \SL_n/P_k'$.
The group $P_k'$ is an extension of special groups ($\SL_k, \SL_{n-k}$ and $\mathbb{G}_a$) and thus special.\NB{ref? Pretty clear}
Thus the étale quotient defining $\SGr(n,k)$ is a Zariski quotient, and hence (the action being free) a homotopy quotient: \[ \SGr(n,k) \wequi (\SL_n)_{hP_k'} \wequi \colim_{i \in \Delta^\op} P_k^{\times i} \times \SL_n. \]
It thus suffices to show that $\SL_n$ and $P_k'$ are (stably) cellular; this is proved in \cite[Proposition 4.1]{wendt2010more}.
\end{proof}

Essentially by construction, the space $\SGr(n,k)$ represents the functor of $n$-dimensional subbundles of $\scr O^k$, together with a choice of trivialization of the determinant.
If $V$ is a symplectic bundle, then $\det(V)$ is trivialized (by the Pfaffian; see e.g. the discussion just before \cite[Definition 4.5]{ananyevskiy2012relation}), and so $V$ is also a special linear bundle.
This induces maps $\HGr(n, k) \to \SGr(2n, 2k)$ and $\HGr(n, \infty) \to \SGr(2n, \infty)$.
\begin{lemma}
Suppose that $A$ is $\SL$-oriented and $\eta$-periodic.
The composite \[ (\HP^\infty)^{n} \to \HGr(n, \infty) \to \SGr(2n, \infty) \to \SGr(2n+1, \infty) \] induces \[ A_{**}(\SGr(2n+1, \infty)) \wequi \Sym^n(A_{**}(\HP^\infty)/\{\beta_1, \beta_3, \dots\}). \]
\end{lemma}
\begin{proof}
Write $\alpha$ for the composite.
We first determine the map \[ \alpha^*: A^{**}(\SGr(2n+1, \infty)) \to A^{**}((\HP^\infty)^{n}) \wequi A^{**}\fpsr{a_1, \dots, a_n}. \]
%If $V$ is any vector bundle on $X \in \Sm_S$, we let $H(V)$ denote the symplectic bundle $V \oplus V^\dual$, with its standard form.
By \cite[Theorem 10]{ananyevskiy2015special} we have \[ A^{**}(\SGr(2n+1, \infty)) \wequi A^{**}\fpsr{p_1, \dots, p_n}; \] here $p_i = p_i(\gamma)$ are the \emph{Pontryagin classes} of the tautological bundle $\gamma$ on $\SGr(2n+1, \infty)$.
Thus $\alpha^*(p_i) = p_i(E)$, where $E := \gamma_1 \boxplus \gamma_2 \boxplus \dots \boxplus \gamma_n$ is the tautological bundle on $(\HP^\infty)^{n}$.
We compute using \cite[Lemma 12]{ananyevskiy2015special} \[ p_t(E) = \prod_i p_t(\gamma_i) = \prod_i (1+a_i^2t^2). \]
In other words \[ \alpha^*(p_i) = \sigma_i(a_1^2, \dots, a_n^2). \]
It follows that $\alpha^*$ is a split injection onto $A^{**}\fpsr{a_1^2, a_2^2, \dots, a_n^2}^{\Sigma_n}$.
By Lemma \ref{lemm:SGr-cellular}, Corollary \ref{cor:duality} and \cite[Remark 13]{ananyevskiy2015special} the map $\alpha_*$ is the topological dual of $\alpha^*$, which is easily checked to be as claimed.
\end{proof}

\begin{proof}[Proof of Theorem \ref{thm:cobordism-MSL-summary}(2)]
By the Thom isomorphism (Lemma \ref{lemm:homology-thom}) we have \[ A_{**} \MSL \wequi \colim_n A_{**} \SGr(2n+1, \infty), \] which by the above identifies with \[ A_{**}(\MSp)/(e_1, e_3, \dots). \]
This was to be shown.
\end{proof}

\subsection{Adams action}
In this section we work over a field $k$ of characteristic $\ne 2$.
Recall that for any $\Sp$-oriented cohomology theory $A$ we have $A^{**}(\HP^\infty) \wequi A^{**}\fpsr{b}$, where $b=b_1(\gamma)$ \cite[Theorem 8.2]{panin2010algebraic}.
Recall also that we put $\KW = \KO[\eta^{-1}]$ and $\kw = \KW_{\ge 0}$.
\begin{lemma}
Suppose that $\W(k) = \F_2$.
Then \[ \adamspsi^3(b) = b(1 + \beta b^2) \in \kw^2(\HP^\infty). \]
\end{lemma}
\begin{proof}
We implicitly invert $3$ throughout this proof.

Since $\kw^{**}(\HP^\infty) \hookrightarrow \KW^{**}(\HP^\infty)$, it suffices to prove the claim for $\KW$.
Note that \[ \KO^{4,2}(\HP^\infty) \wequi \bigoplus_{i \ge 0} \KO^{4-4i,2-2i}\{b_1^{\KO}(\gamma)^i\} \wequi \bigoplus_{i \ge 0} \Z\{b_1^{\KO}(\gamma)^i\} \hookrightarrow \KGL^{4,2}(\HP^\infty); \] here we use that $\KO^{8n,4n} \wequi \GW(k) = \Z$ by assumption, and $\KO^{8n+4,4n+2} \wequi \Z$ as always (see e.g. \cite[Table 1]{bachmann-very-effective}).
Write $\alpha: \KO \to \KGL$ for the canonical map and $H(1) \in \KO^{4,2}(\HP^\infty) \wequi \KSp^0(\HP^\infty)$ for the class of the tautological bundle.
We shall first determine $\adamspsi^3(H(1))$, and to do this we determine $\alpha(\adamspsi^3(H(1))) = \adamspsi^3(\alpha(H(1)))$ (see Theorem \ref{thm:ko-adams}(4)).
We have $\alpha(H(1)) = \beta_\KGL^{-2} \gamma$, where $\gamma \in \KGL^0(\HP^\infty)$ is the tautological bundle.
By Remark \ref{rmk:adams-classical}, ``our'' Adams operation on $\KGL$ is just the classical one, so can be computed in terms of exterior powers of vector bundles (see e.g. \cite[\S II.4]{weibel-k-book}).
It follows that for any rank $2$ vector bundle $V$ we have \[ \adamspsi^3(V) = V^{\otimes 3} - 3V \otimes \det{V} \in \KGL^0(X). \]
If the bundle is symplectic, then $\det V = 1$.
Hence \begin{gather*} \adamspsi^3(\alpha(H(1))) = \adamspsi^3(\beta_\KGL^{-2}\gamma) = 3^{-2}\beta_\KGL^{-2}(\gamma^{\otimes 3} - 3\gamma) \\ = 3^{-2} (\beta_\KGL^{-2}\gamma) (\beta_\KGL^4 (\beta_\KGL^{-2} \gamma)^2 - 3) = \alpha(3^{-2} H(1) (\beta H(1)^2 - 3)). \end{gather*}
Here we have used that $\psi^3$ is a ring map and $\psi^3(\beta_\KGL) = 3\beta_\KGL$ (by construction).
Hence by injectivity of $\alpha$ we get \[ \adamspsi^3(H(1)) = 3^{-2} H(1) (\beta H(1)^2 - 3). \]

By \cite[Theorem 6.10]{ananyevskiy2017stable} we have $b_1^\KO(\gamma) = H(1) - H_-$, where $H_- \in \KO^{4,2}(*)$.
Hence also $\adamspsi^3(H_-) \in \KO^{4,2}(*).$
It follows that $\adamspsi^3(b) \in \KW^{2}(\HP^\infty)$ is the image of $\adamspsi^3(H(1)) - \adamspsi^3(H_-) \in \KO^{4,2}(\HP^\infty)$.
One has $\KW^* \wequi \W(k)[\beta^{\pm 1}]$, with $|\beta|=4$ (see e.g. \S\ref{subsec:KW}).
Thus $\KW^2 = 0$, so $\psi^3(b)$ is just the image of $\adamspsi^3(H(1))$.
Since $2=0 \in \KW^*$, the result follows.
\end{proof}

\begin{remark}
Our proof is complicated by the fact that we did not want to use the geometric description of the Adams operations on $\KO^{**}(\HP^\infty)$; in particular this made it difficult to determine $\adamspsi^3(b)$ without the assumption that $\W(k) = \F_2$.
If we allow ourselves this description (i.e. Proposition \ref{prop:FH-comparison} and Lemma \ref{lemm:adamspsi-FH-action}) we get \[ \adamspsi^3(b_1(\gamma)) = \frac{\adamspsi^3_\geo(H(1) - H_-)}{\lra{-1}\cdot 3 \cdot 3_\epsilon} \in \KO^{4,2}(\HP^\infty). \]
Using that for a rank $2$ symplectic bundle $E$ one has $\adamspsi^3_\geo(E) = E^3 - 3E$\NB{since $(x+y)^3 = x^3+y^3 +3(xy)(x+y)$}, one easily deduces that \[ \adamspsi^3(b) = b(1-\beta b^2/3) \in \kw^2(\HP^\infty), \] over any field.
\end{remark}

\begin{lemma}
For a scheme $S$ and $n$ odd, the maps $\adamspsi^n: \KW[1/n] \to \KW[1/n] \in \SH(S)$ and (provided $S$ is a field) $\adamspsi^n: \kw[1/n] \to \kw[1/n]$ are equivalences.
\end{lemma}
\begin{proof}
The second map is obtained from the first by applying the connective cover functor, so the second claim follows from the first.
The first claim is compatible with base change, so we may check it over $\Spec(\Z[1/2])$, and thus we may reduce to fields \cite[Proposition B.3]{bachmann-norms}.
It thus suffices to show that $\adamspsi^n: \ul{\pi}_* \KW[1/n] \to \ul{\pi}_* \KW[1/n]$ is an isomorphism; recall that $\ul{\pi}_* \KW[1/n] = \ul{W}[1/n, \beta, \beta^{-1}]$.
Hence by Example \ref{ex:adamspsi-pi0} and Theorem \ref{thm:ko-adams}(2,3) the map on $\ul{\pi}_{4i}$ is given by multiplication by $n^{2i}$, which is an isomorphism as needed (and all other homotopy sheaves vanish).
\end{proof}

\begin{proof}[Proof of Proposition \ref{prop:adams-cobordism-summary}]
It suffices to prove the ``in fact'' statement.

Denote the inverse of $\adamspsi^3$ by $\adamspsi^{1/3}$.
Note that $\adamspsi^{3}(\beta) = 9\beta = \beta$, so that $\adamspsi^3$ acts by the identity on $\kw_*$, and hence so does $\adamspsi^{1/3}$.
Applying Lemma \ref{lemm:kronecker-adams} to $\adamspsi^{1/3}$ we deduce that $\lra{\adamspsi^{1/3} x, y} = \lra{x, \adamspsi^3 y}$, i.e. that the action of $\adamspsi^{1/3}$ on $\kw_* \HP^\infty$ is dual to the action of $\adamspsi^3$ on $\kw^* \HP^\infty$.
On $\kw^{*}\HP^\infty$ we have \[ \adamspsi^3(b^n) = \adamspsi^3(b)^n = b^n(1 + \beta b^2)^n = b^n + n\beta b^{n+2} + O(\beta^2). \]
Hence \begin{equation} \label{eq:adams-1/3} \adamspsi^{1/3}(\beta_i) = \beta_i + \beta(i-2)\beta_{i-2} + O(\beta^2). \end{equation}
It follows\NB{for any ring map $\adamspsi: A \to A$ and $a \in A_{**}$, $m \in A_{**}E$ we have $\adamspsi(am) = \adamspsi(a)\adamspsi(m)$} that for any $x \in \kw_* \HP^\infty$ we have $\adamspsi^{1/3}(x) = x + O(\beta)$.
Since $\adamspsi^3$ is $\beta$-linear and inverse to $\adamspsi^{1/3}$, we deduce that $x = \adamspsi^3\adamspsi^{1/3}(x) = \adamspsi^3(x) + O(\beta)$, i.e. $\adamspsi^3(x) = x + O(\beta)$.
Using this when applying $\adamspsi^3$ to \eqref{eq:adams-1/3} we find that\NB{other perspective: $\adamspsi^{1/3} = \id + \beta f + O(\beta^2)$, whence $(\adamspsi^{1/3})^{\circ 2} = \id + O(\beta^2)$, i.e. modulo $\beta^2$ $\adamspsi^{1/3}$ is its own inverse, i.e. $\adamspsi^3 = \adamspsi^{1/3} + O(\beta^2)$} \[ \beta_i = \adamspsi^3(\beta_i) + \beta(i-2)(\beta_{i-2} + O(\beta)) + O(\beta^2), \quad\text{whence}\quad \adamspsi^{3}(\beta_i) = \beta_i + \beta(i-2)\beta_{i-2} + O(\beta^2). \]

Using that the $\beta_{i} \in \kw_* \HP^\infty$ maps to $e_{i-1} \in \kw_* \MSp$ and $2=0 \in \kw_*$ the result follows.
\end{proof}

\section{Some completeness results} \label{sec:completeness}
\subsection{Summary}
\begin{theorem} \label{thm:eta-complete}
Let $k$ be a field of $\chara(k) \ne 2$, and suppose that $\vcd_2(k) < \infty$.
If $E \in \SH(k)^\veff$, then $E_2^\comp$ is $\eta$-complete.
\end{theorem}

In particular, the map \[ \pi_{**} E_2^\comp \to \pi_{**} E_{2,\eta}^\comp \] is an isomorphism.
This is the only form of the result we shall use in the sequel.
In \S\ref{sec:completeness:main} we give a complete argument for this $\pi_{**}$-isomorphism.
The extension to an equivalence of spectra uses a technical result established in the remaining subsections; this is not relevant for the sequel and so can be skipped.

\subsection{Main result} \label{sec:completeness:main}
We first recall and extend the slice completeness result of \cite[\S5]{bachmann-bott}.
Recall that the \emph{slice tower} \cite{voevodsky-slice-filtration} is a functorial tower $f_\bullet E \to E \in \SH(k)$; the \emph{slice completion} of $E$ is \[ \scomp E = \lim_n \cof(f_n E \to E). \]
\begin{proposition} \label{prop:slice-complete}
Let $k$ be a field of exponential characteristic $e$, $t$ coprime to $e$ with $\vcd_t(k) < \infty$.\footnote{We are using here the slightly extended definition of virtual cohomological dimension introduced in \cite[\S1.1.2]{bachmann-bott}, i.e. $\vcd_t(k) = \max \{\cd_p(k[\sqrt{-1}] \mid p | t\}$.}
If $E \in \SH(k)_{\ge 0}$, then $E_{t,\rho}^\comp \to \scomp(E)_{t,\rho}^\comp$ is an equivalence.
\end{proposition}
Note that the fact that $E_{t,\rho}^\comp \to \scomp(E)_{t,\rho}^\comp$ is a $\pi_{**}$-isomorphism was established in \cite[Corollary 5.13]{bachmann-bott}.
We will give the full proof of Proposition \ref{prop:slice-complete} in \S\ref{subsub:sc-proof}.

\begin{example}
Let $E \in \SH(k)[\eta^{-1}]_{\ge 0}$.
Since $\rho=2$ on $\SH(k)[\eta^{-1}]$, we deduce that $E/2^n$ is slice complete.
It follows that slice spectral sequence techniques as in \cite{ormsby2019homotopy} could be used (over fields of characteristic $\ne 2$ and $\vcd_2 < \infty$) to study the relationship between $\1[\eta^{-1}]_2^\comp$ and $\kw_2^\comp$.
We pursue a different strategy in the sequel.
\end{example}

The following argument is closely related to part of \cite[Lemma 20]{hu2011convergence}.
\begin{lemma} \label{lemm:trade-rho}
Let $E \in \SH(k)$.
Then $E_2^\comp$ is $\eta$-complete if and only if $E/(2,\rho)$ is $\eta$-complete.
\end{lemma}
\begin{proof}
Necessity is clear since $\eta$-complete spectra are stable under finite colimits.
We thus show sufficiency.
Consider the fiber sequence \[ F \to E/2 \to (E/2)_\eta^\comp. \]
Then $F$ is $\eta$-periodic, $2$-complete, and $F/\rho \wequi 0$ by assumption.
By \eqref{eq:rho-eta-periodic} we have $F/\rho \wequi F/2$.
We deduce that $F \wequi 0$ ($F$ being $2$-complete); in other words $E/2$ is $\eta$-complete.
The result follows since $\eta$-complete spectra are stable under limits.
\end{proof}

\begin{proof}[Proof of Theorem \ref{thm:eta-complete}]
By Proposition \ref{prop:slice-complete}, $E/(2, \rho) \to \scomp(E)/(2,\rho)$ is an equivalence.
Since slices are $\eta$-complete (being modules over $\HZ_{(2)} \wequi s_0(\1_{(2)})$; see e.g. \cite[Theorem B.4]{bachmann-norms}), we deduce that $E/(2, \rho)$ is $\eta$-complete (being a limit of finite extensions of slices, by the effectivity assumption).
We conclude by Lemma \ref{lemm:trade-rho}.
\end{proof}

\subsection{Remaining proofs}
\subsubsection{}
\begin{definition}
\begin{enumerate}
\item By a \emph{tower} in a category $\scr C$ we mean an object of $\Fun(\Z^\op, \scr C)$, i.e. a diagram \[ E_\bullet = \dots E_2 \to E_1 \to E_0 \to E_{-1} \to \dots \in \scr C. \]
\item Suppose that $\scr C$ is pointed.
  We call a tower $E_\bullet$ \emph{nilpotent} if for every $n \in \Z$ there exists $N > n$ such that the composite map $E_N \to E_n$ is zero.
\end{enumerate}
\end{definition}

\begin{lemma} \label{lemm:boardman}
Let $E_\bullet \in \SH$ be a tower of spectra such that each tower $\pi_i(E_\bullet)$ of abelian groups is nilpotent.
Then \[ \lim_{i \to \infty} E_i \wequi 0. \] 
\end{lemma}
\begin{proof}
For fixed $i$, the tower $\pi_i(E_\bullet)$ is Mittag--Leffler, and hence $\limone_s \pi_i(E_s) \wequi 0$ \cite[Proposition 3.5.7]{weibel-hom-alg}.
Clearly also $\lim_s \pi_i E_s \wequi 0$.
The claim thus follows from the Milnor exact sequence \cite[Proposition VI.2.15]{goerss2009simplicial}.
\end{proof}

If $A_\bullet \to A$ is a tower in $\scr C_{/A}$ with $\scr C$ an abelian category, there is a canonical descending filtration on $A$ denoted by \[ F^kA = \im(A^k \to A). \]

\subsubsection{}%Postnikov towers}
Let $\scr C$ be a triangulated category with a $t$-structure.
Fix $X \in \scr C$.
For every $E \in \scr C$, we have a natural tower
\[ \dots \to E_{\ge 2} \to E_{\ge 1} \to E_{\ge 0} \to \dots \to E. \]
Applying $[X, \ph]$ we obtain a tower in abelian groups, and hence a natural descending filtration on $[X, E]$ which we denote by $G^\bullet[X, E]$.
Given $F \in \scr C^\heart$, we put $H^i(X, F) = [X, \Sigma^i F]$.

\begin{lemma} \label{lemm:postnikov}
Let $\alpha: E \to F \in \scr C$ induce the zero map $H^i(X, \pi_i E) \to H^i(X, \pi_i F)$.
Then \[ \alpha(G^i[X, E]) \subset G^{i+1}[X, F]. \]
\end{lemma}
\begin{proof}
An element of $G^i[X, E]$ is represented by a map $X \to E_{\ge i}$.
Its image in $[X, F]$ is lands in $G^{i+1}[X, F]$ if and only if the composite $X \to E_{\ge i} \to E \to F$ factors through $F_{\ge i+1}  \to F$.
For this it is enough that the composite $X \to E_{\ge i} \to F_{\ge i} \to \Sigma^i\pi_i(F)$ is zero.
Since this factors as $X \to \Sigma^i\pi_i(E) \to \Sigma^i\pi_i(F)$, the result follows.
\end{proof}

We now specialise this to $\scr C = \SH(k)$, $X$ (the suspension spectrum of) a smooth variety.

\begin{remark} \label{rmk:postnikov-end}
Let $X \in \Sm_k$.
Then $G^k[X, E] = 0$ for $k > \dim{X}$; in fact $[X, E_{\ge k}] = 0$.\NB{And so the descent spectral sequence converges strongly.}
This follows from the fact that the Nisnevich topos of $X$ has homotopy dimension at most $\dim X$ (see e.g. \cite[Proposition A.3(3)]{bachmann-norms}).
\end{remark}

\begin{corollary} \label{cor:postnikov-SHk-inc}
Let $X \in \Sm_k$ be connected of dimension $\le d$.
Let $\alpha: E \to F \in \SH(k)$ induce the zero map on the generic stalk $\ul{\pi}_i(\ph)_0(k(X))$, for $i = 0, 1, \dots, d$.
Then the map \[ \alpha_*: [X, E] \to [X, F] \] increases filtration by (at least) $1$.
\end{corollary}
\begin{proof}
We first claim that $\alpha: H^k(X, \ul{\pi}_i(E)_0) \to H^k(X, \ul{\pi}_i(F)_0)$ is the zero map, for $k \ge 0$ and $0 \le i \le d$.
Since Zariski and Nisnevich cohomology of the homotopy sheaves of $E$ agree \cite[Lemma 6.4.7]{morel2005stable}, for this it suffices to show that $\alpha: \ul{\pi}_i(E)_0|_{X_\Zar} \to \ul{\pi}_i(F)_0|_{X_\Zar}$ is the zero map.
This follows from unramifiedness of homotopy sheaves \cite[Lemma 6.4.4]{morel2005stable} and our assumption on $\alpha(k(X))$.
We have thus proved the claim.

Applying Lemma \ref{lemm:postnikov}, we deduce that ($*$) $\alpha(G^i[X, E]) \subset G^{i+1}[X, F]$ for $0 \le i \le d$.
Since $\Sigma^\infty_+ X \in \SH(k)_{\ge 0}$, for $i < 0$ we have $[X, F] \wequi [X, F_{\ge 0}] \wequi [X, F_{\ge i+1}]$, and so ($*$) still holds.
For $i>d$ we have $G^k[X, E] = 0$ by Remark \ref{rmk:postnikov-end}, and so again ($*$) holds.

This concludes the proof.
\end{proof}

\subsubsection{}%Locally nilpotent towers}
\begin{definition} \label{def:locally-nilpotent}
Let $E_\bullet \in \SH(k)$ be a tower.
\begin{enumerate}
\item We call $E_\bullet$ \emph{locally nilpotent} if for every connected $X \in \Sm_k$ and $i, j \in \Z$, the tower \[ \ul{\pi}_{i,j}(E_\bullet)(k(X)) \] is nilpotent.
\item We call $E_\bullet$ \emph{sectionwise nilpotent} if for every $X \in \Sm_k$ and $i, j \in \Z$, the tower \[ [\Sigma^{i,j} X, E_\bullet] \] is nilpotent.
\end{enumerate}
\end{definition}

\begin{proposition} \label{prop:local-sectionwise-nilpotent}
Let $E_\bullet \in \SH(k)$ be locally nilpotent.
Then it is sectionwise nilpotent.
\end{proposition}
\begin{proof}
Let $X \in \Sm_k$.
It suffices to show that there exists $N>0$ such that $[X, E_N] \to [X, E_0]$ is the zero map.
Writing $X$ as the disjoint union of its (finitely many) connected components, and using that $[X_1 \amalg X_2, E] \wequi [X_1, E] \oplus [X_2, E]$, we may assume that $X$ is connected, say of dimension $d$.
We claim that for any $i$ there exists $N(i)>i$ such that $[X, E_{N(i)}] \to [X, E_i]$ increases postnikov filtration by (at least) $1$.
Assuming this for now, if $N = N^{\circ(d+1)}(0)$, then the composite \[ [X, E_N] \to [X, E_{N^{\circ d}(0)}] \to [X, E_{N^{\circ (d-1)}(0)}] \to \dots \to [X, E_0] \] increases postnikov filtration by $d+1$, and hence is the zero map by Remark \ref{rmk:postnikov-end}.

It hence suffices to prove the claim; clearly we may assume that $i=0$.
The local nilpotence assumption implies that there exists $N = N(0) > 0$ such that the maps $\ul{\pi}_i(E_N)_0(k(X)) \to \ul{\pi}_i(E_0)_0(k(X))$ are all zero, for $0 \le i \le d$.
The claim now follows from Corollary \ref{cor:postnikov-SHk-inc}.

This concludes the proof.
\end{proof}

\begin{corollary} \label{cor:loc-nilpotent-convergent}
Let $E_\bullet \in \SH(k)$ be locally nilpotent.  
Then $\lim_i E_i \wequi 0$.
\end{corollary}
\begin{proof}
Immediate from Proposition \ref{prop:local-sectionwise-nilpotent} (which shows that $\pi_i \map(\Sigma^{j,k} X, E_\bullet)$ is nilpotent) and Lemma \ref{lemm:boardman} (which implies that $\map(\Sigma^{j,k} X, \lim_i E_i) \wequi \lim_i \map(\Sigma^{j,k} X, E_i) \wequi 0$).
\end{proof}

\subsubsection{}\label{subsub:sc-proof}
We can now prove our extended slice completeness result.
\begin{proof}[Proof of Proposition \ref{prop:slice-complete}]
By \cite[Theorem 5.3(1)]{bachmann-bott} (using \cite[Remark 5.12]{bachmann-bott}) we know that $f_\bullet(E)/(t^n,\rho^m)$ is locally nilpotent.
The claim thus follows from Corollary \ref{cor:loc-nilpotent-convergent}.
\end{proof}

\section{The $\HW$-homology of $\kw$} \label{sec:HW-kw}
\subsection{Summary}
Throughout $k$ is a field with $\vcd_2(k) < \infty$.
Recall (see Lemmas \ref{lemm:witt-completion}, \ref{lemm:witt-torsion} and \ref{lemm:derived-p-comp}) that then \[ \W(k)_\I^\comp \wequi \W(k)_2^\comp \wequi L_2^\comp \W(k). \]
See \S\ref{subsec:spectra} for a definition of the motivic spectra $\HW, \kw$.
The following result will be improved upon in Proposition \ref{prop:kw-HW-better}.
\begin{theorem} \label{thm:HW-kw-summary}
Let $\vcd_2(k) < \infty$.
We have \[ \pi_* ((\kw \wedge \HW)_2^\wedge) \wequi \begin{cases} \W(k)_\I^\comp & *=4i \ge 0 \\ 0 & \text{else} \end{cases}. \]

There exist generators $t_i \in \pi_{4 \cdot 2^i}((\kw \wedge \HW)_2^\wedge)$ and  $x_i \in \pi_{4i}((\kw \wedge \HW)_2^\wedge)$ such that the following hold.
\begin{enumerate}
\item Let $k'/k$ be an extension with $\vcd_2(k') < \infty$.
  Then $x_i|_{k'}$ generates $\pi_{4i}((\kw_{k'} \wedge \HW_{k'})_2^\wedge) \wequi \W(k')_2^\comp$.
\item For $i \ge 0$ we have $t_i^2 \in (2 + \I(k)^2)t_{i+1}$.
\item Writing $i = \sum_n \epsilon_n 2^n$ for the binary expansion of $i \ge 0$, up to a unit (of $\W(k)_\I^\comp$) we have \[ x_i = \prod_n t_n^{\epsilon_n}. \]
\item $x_0$ can be chosen to be $1$.
\end{enumerate}
\end{theorem}

\begin{remark}
It follows that $\pi_*((\kw \wedge \HW)_2^\wedge)$ is a flat $\W(k)_2^\wedge$-algebra generated by $t_0, t_1, \dots$, subject to the relations $t_i^2 = (2+r_i) t_{i+1}$ for certain $r_i \in \I(k)^2$.
\end{remark}

\subsection{The motivic dual Steenrod algebra} \label{subsec:steenrod}
\subsubsection{} \label{subsub:HZ2}
We will use the \emph{mod $2$ motivic cohomology spectrum} $\HZ/2$.
It is a consequence of the resolution of the Beilinson--Lichtenbaum and Milnor conjectures (see e.g. \cite[(7.1)]{kylling2018hermitian}) that \[ \ul{\pi}_*(\HZ/2)_* \wequi \ul{k}^M_*[\tau]. \]
Here $\ul{k}^M_*$ denotes the homotopy module of \emph{mod $2$ Milnor $K$-theory} \cite[Example 3.33]{A1-alg-top}, and $\tau \in \pi_1(\HZ/2)_1$.
In particular we have a cofiber sequence \begin{equation} \label{eq:cof-tau-HZ2} \Sigma^{0,-1} \HZ/2 \xrightarrow{\tau} \HZ/2 \to \ul{k}^M. \end{equation}

\subsubsection{}
We recall the structure of the motivic dual Steenrod algebra \cite[Theorem 5.6]{hoyois2013motivic}.
There exist elements \[ \tau_i, \xi_j \in \HZ/2_{**}\HZ/2, i = 0, 1, 2, \dots, j = 1, 2, \dots \] with \[ |\tau_i| = (2^{i+1} - 1, 2^i - 1) \quad\text{and}\quad |\xi_i| = (2^{i+1}-2, 2^i-1). \]
Then, when viewed as a left $\HZ/2$-algebra, there is an equivalence\footnote{in the naive sense that the right hand side has an $\HZ/2$-module structure and also a homotopy $\HZ/2$-algebra structure, and there is an equivalence respecting both structures} \begin{equation} \label{eq:motivic-dual-steenrod} \HZ/2 \wedge \HZ/2 \wequi \HZ/2[\tau_0, \tau_1, \dots, \xi_1, \xi_2, \dots]/(\tau_i^2 - (\tau + \rho\tau_0) \xi_{i+1} - \rho \tau_{i+1}). \end{equation}
We may occasionally write $\xi_0 := 1$.
The switch map on $\HZ/2 \wedge \HZ/2$ induces an automorphism of $\HZ/2_{**}\HZ/2$ called \emph{antipode} which we denote by $x \mapsto \overline{x}$.

\begin{lemma} \label{lemm:right-unit}
\begin{enumerate}
\item The action of $\HZ/2^{**}\HZ/2$ on $\HZ/2_{**}$ is determined by $\Sq^1(\tau) = \rho$, $\Sq^i(x) = 0$ for $x = \tau$ and $i>1$, or $i \ge 0$ and $x = 1$ or $x \in \ul{k}_1^M(k)$.
\item The right unit $\eta_R: \HZ/2_{**} \to \HZ/2_{**}\HZ/2$ is given by $\eta_R(\tau) = \tau + \rho \tau_0$ and $\eta_R(a) = a$, for $a \in \ul{k}_*^M(k)$.
\end{enumerate}
\end{lemma}
\begin{proof}
The right unit is a special case of a homology coaction, and so dual to the action of $\HZ/2^{**}\HZ/2$ on $\HZ/2^{**}$\NB{ref?}: $\eta_R(x) = \sum_I \Sq^I(x) \cdot \widehat{\Sq^I}$.
Thus (1) and (2) are (essentially) equivalent.
The right unit is $\ul{k}_*^M(k)$-linear (see e.g. Example \ref{ex:adamspsi-pi0} and use that $\ul{\pi}_0(\1)_* \wequi \ul{K}_*^{MW} \to \ul{\pi}_0(\HZ/2) \wequi \ul{k}_*^M$ is surjective) and satisfies $\eta_R(1) = 1$.
Thus for $x \in \ul{k}^M_*(k)$ we get $\eta_R(x) = x$.
For degree reasons\NB{$\tau$ is in Milnor--Witt stem $1$, $\tau_i$ in $2^i$ and $\xi_i$ in $2^i-1$. The only summands which can contribute are thus $1$, $\tau_0$ and $\xi_1$.}, we must have $\eta_R(\tau) = a\tau + b\tau_0 + c \xi_1$, for some $a \in \F_2$, $b \in \ul{k}_1^M(k)$, $c \in \ul{k}_2^M(k)$.
This translates into $\Sq^0(\tau) = a\tau$, $\Sq^1(\tau) = b$ and $\Sq^2(\tau) = c$ (using \cite[Lemmas 13.1 and 13.5]{voevodsky2003reduced}).
Hence $a=1$.
We have $c = 0$ by \cite[Lemma 9.9]{voevodsky2003reduced}; $b = \rho$ is holds essentially by definition (using that $\Z(1)[1] = \Gm$ and considering the long exact sequence computing $H^*(k,\Z/2(1))$).
\end{proof}

\begin{corollary} \label{cor:right-gens}
The monomials \[ \prod_{i,j} \overline{\tau_i}^{\epsilon_i} \overline{\xi_j}^{n_j} \] form a \emph{left} $\HZ/2_{**}$-module basis of $\HZ/2_{**}\HZ/2$.
\end{corollary}
\begin{proof}\discuss{Suggestions?}
Since their conjugates form a left basis, it is clear that these monomials form a right basis.
It suffices to prove that elements of the form $\tau^p \overline{m}$ (where $\overline m$ is one of the monomials) form a $k_*^M$-basis.
By Lemma \ref{lemm:right-unit}, elements of the form $(\tau+\rho \tau_0)^p \overline{m}$ form a $k_*^M$-basis, where $\overline{m}$ is one of the monomials in the claim.
Order the monomials in the $\overline{\tau}_i, \overline{\xi}_i$ and $\tau$ lexicographically, with $\tau < \overline{\tau}_i < \overline{\xi}_i < \overline{\tau}_{i+1}$.
The relation from \eqref{eq:motivic-dual-steenrod} ensures that $\tau_i^2$ is a sum of monomials $>\tau_i$; this implies that if $\tau^p \overline{m}$ is any such monomial, $\tau_0 \tau^p \overline{m}$ is a sum of larger monomials.
Since $\overline{\tau_0} = \tau_0$ (see again \cite[Theorem 5.6]{hoyois2013motivic}), it follows that the matrix expressing the elements $\{(\tau+\rho \tau_0)^p \overline{m}\}_{p,\overline{m}}$ in terms of the basis $\{\tau^p \overline{m}\}_{p,\overline{m}}$ is triangular (with unit diagonal).
The result follows.\NB{better argument? E.g. can use exotic grading which enables graded Nakayama...}
\end{proof}

\subsubsection{}
The Kronecker pairing induces an isomorphism $\HZ/2^{**}\HZ/2 \wequi (\HZ/2_{**}\HZ/2)^\dual$ (see e.g. \cite[Proposition 5.5]{hoyois-algebraic-cobordism}).
By passing to the dual of the monomial basis of $\HZ/2_{**}\HZ/2$, any monomial $m \in \HZ/2_{p,q}\HZ/2$ defines a \emph{dual} $\widehat m \in \HZ/2^{p,q} \HZ/2$.
\begin{warning}
We can extend this map $\HZ/2_{**}$-linearly, but the extension is not compatible with gradings: if $a \in \HZ/2_{r,s}$ then $\widehat{am} \in \HZ/2^{p-r,q-s} \HZ/2$ whereas $am \in \HZ/2_{p+r,q+s}\HZ/2$.
\end{warning}
\begin{example} \label{ex:steenrod-explicit}
Since $\HZ/2^{p,q}$ is concentrated in $q>0$, except for $\HZ/2^{0,0} = \Z/2$, we find that $\HZ/2^{*,0}\HZ/2$ is generated as an $\F_2$-vector space by $1$ and $\widehat\tau_0$.
In other words \[ \HZ/2^{0,0}\HZ/2 = \F_2, \quad \HZ/2^{1,0} = \F_2\{\widehat{\tau_0}\} \quad\text{and}\quad \HZ/2^{p,0}\HZ/2 = 0 \text{ else}. \]
Similarly $\HZ/2^{*,1}\HZ/2$ is generated by $\tau, k_1^M, \widehat{\xi_1}, \widehat{\tau_1}, \tau\widehat{\tau_0}, k_1^M\widehat{\tau_0}, \widehat{\xi_1\tau_0}, \widehat{\tau_1\tau_0}$, so that \begin{gather*} \HZ/2^{0,1}\HZ/2 = \F_2\{\tau\}, \HZ/2^{1,1}\HZ/2 = k_1^M \oplus \F_2\{\tau\widehat{\tau_0}\}, \HZ/2^{2,1}\HZ/2 = \F_2\{\widehat{\xi_1}\} \oplus k_1^M\{\widehat{\tau_0}\}, \\ \HZ/2^{3,1}\HZ/2 = \F_2\{\widehat{\tau_1}\} \oplus \F_2\{\widehat{\tau_0\xi_1}\}, \HZ/2^{4,1}\HZ/2 = \F_2\{\widehat{\tau_0\tau_1}\}, \text{ and } \HZ/2^{p,1}\HZ/2 = 0 \text{ else}. \end{gather*}
\end{example}

If $\alpha \in \HZ/2^{**}\HZ/2$ is any endomorphism, then we denote by \[ \alpha^L = \alpha \wedge \id, \alpha^R = \id \wedge \alpha : \HZ/2 \wedge \HZ/2 \to \HZ/2 \wedge \HZ/2 \] the induced maps.
We obtain \[ \alpha^L_*, \alpha^R_*: \HZ/2_{**}\HZ/2 \to \HZ/2_{**}\HZ/2. \]
\begin{lemma} \label{lemm:action-formulas}
\begin{enumerate}
\item For any $\alpha \in \HZ/2^{**}\HZ/2$, the map $\alpha^L_*$ is right $\HZ/2_{**}$-linear, and $\alpha^R_*$ is left $\HZ/2_{**}$-linear
\item We have the formulas
\begin{align*}
\begin{split}
\widehat{\tau_0}^L_*(\tau_0) &= 1 \\
\widehat{\tau_0}^L_*(\tau_i) &= 0, i > 0 \\
\widehat{\tau_0}^L_*(\xi_i) &= 0, i \ge 0 \\
\end{split}
\quad
\begin{split}
\widehat{\tau_1}^L_*(\tau_1) &= 1 \\
\widehat{\tau_1}^L_*(\tau_i) &= 0, i \ne 1 \\
\widehat{\tau_1}^L_*(\xi_i) &= 0, i \ge 0 \\
\end{split}
\quad
\begin{split}
\widehat{\xi_1}^L_*(\tau_1) &= \tau_0 \\
\widehat{\xi_1}^L_*(\tau_i) &= 0, i \ne 1 \\
\widehat{\xi_1}^L_*(\xi_1) &= 1 \\
\widehat{\xi_1}^L_*(\xi_i) &= 0, i \ne 1 \\
\end{split}
\quad
\begin{split}
\widehat{\tau_0}^R_*(\tau_i) &= \xi_i, i \ge 0 \\
\widehat{\tau_0}^R_*(\xi_i) &= 0, i \ge 1, \text{ and} \\
\widehat{\xi_1}^R_*(\tau_i) &= 0, i \ge 0 \\
\widehat{\xi_1}^R_*(\xi_i) &= \xi_{i-1}^2, i \ge 0.
\end{split}
\end{align*}
  (Here in the last formula $\xi_{-1} := 0$.)

\item $\widehat{\tau_0}^L_*$ is a derivation (i.e. satisfies $\widehat{\tau_0}^L_*(ab) = \widehat{\tau_0}^L_*(a)b + a\widehat{\tau_0}^L_*(b)$), $\widehat{\tau_1}^L_*$ is a derivation on the right $\HZ/2_{**}$-subalgebra on all the generators except $\tau_0$, and $\widehat{\xi_1}^L_*$ is a derivation on the right $\HZ/2_{**}$-subalgebra on all the generators except $\tau_0, \tau_1$.%???, and $\widehat{\tau_1}^R_*$ is a derivation on the left $\HZ/2_{**}$-subalgebra on all the generators except $\tau_0$.
\end{enumerate}
\end{lemma}
\begin{proof}
(1) Immediate from the definitions.

(2) For $E \in \SH(k)$ we have an action \[ \HZ/2^{**}\HZ/2 \otimes_{\HZ/2_{**}} \HZ/2_{**}E \to \HZ/2_{**} E, \alpha \otimes e \mapsto \alpha_*(e). \]
By the eightfold way \cite[p. 190]{boardman1982eightfold}, this is obtained from the coaction \[ \Delta: \HZ/2_{**} E \to \HZ/2_{**}\HZ/2 \otimes_{\HZ/2_{**}} \HZ/2_{**} E \] by partial dualization: if $\Delta(e) = \sum_i a_i \otimes e_i$ then $\alpha_*(e) = \sum_i \lra{a_i, \alpha} e_i$.\NB{We are ignoring some conjugations here, but should be ok b/c char. 2...}
Applying this with $E = \HZ/2$, it follows that for $\alpha \in \HZ/2^{**}\HZ/2$ and $x \in \HZ/2_{**}\HZ/2$ with $\Delta(x) = \sum x_i \otimes y_i$, we have \[ \alpha^L_*(x) = \sum_i \lra{x_i, \alpha} y_i \text{, and similarly } \alpha^R_*(x) = \sum_i \lra{y_i, \alpha} x_i. \]
The formulas now follow from the formulas for the comultiplication in $\HZ/2_{**}\HZ/2$ \cite[Theorem 5.6]{hoyois2013motivic}.

(3) \NB{Slicker proof/reference? Details?}\discuss{Suggestions?}
By (1), it suffices to show the claims about the $\Z/2$-subalgebras.
Let $g$ be one of the monomial generators of $\HZ/2_{**}\HZ/2$ and $a, b \in \HZ/2_{**}\HZ/2$.
Suppose we can write \[ \Delta(a) = 1 \otimes a + g \otimes a' + \sum a_i \otimes a_i', \quad \Delta(b) = 1 \otimes b + g \otimes b' + \sum b_i \otimes b_i', \] where the $a_i, b_i$ are monomials.
Suppose further that when expanding $a_i b_j$ into monomials, $g$ does not appear.
Then $\widehat{g}^L_*(a) = a', \widehat{g}^L_*(b) = b'$ and \[ \Delta(ab) = g \otimes (ab' + a'b) + \dots, \] where $g$ does not appear on the left in the omitted terms.
It follows that $\widehat{g}^L_*(ab) = ab' + a'b$, which is what we wanted.
$\Delta(a), \Delta(b)$ can always be written in the desired form, the only problem may occur when expanding the product.
As long as $g$ itself is not a product, the only way the assumption can fail is from an ``unexpected'' contribution, i.e. coming from $\tau_i^2 = (\tau + \rho\tau_0) \xi_{i+1} + \rho \tau_{i+1}$.
It thus suffices to ensure that $\Delta(a), \Delta(b)$ do not contain $\tau_i$ on the left hand side, for $0 \le i < N$ for certain $N$ depending on $g$.
This will happen if $\tau_i$ for $i<N$ are excluded as generators (considering again the explicit formulas for the comultiplication).
One checks that for $g = \tau_0, \tau_1, \xi_1$, respectively $N=0,1,2$ work.\NB{Doesn't $N=1$ work for $g=\xi_1$ as well?}
The result follows.
\end{proof}

\subsection{Spectra employed in the proof} \label{subsec:spectra}
\subsubsection{}
We write $\ko = \tilde f_0 \KO$ for the spectrum of \emph{very effective hermitian $K$-theory}.
We also put $\kgl = f_0 \KGL \wequi \tilde f_0 \KGL$.
Note that we have canonical ring maps $\ko \to \kgl \to s_0(\KGL) \wequi \HZ$.
\begin{proposition} \label{prop:HZ2-ko}
The canonical map $\HZ/2_{**} \ko \to \HZ/2_{**}\HZ/2$ is injective and hits the elements $\overline{\xi}_1^2, \overline{\xi}_2, \dots, \overline{\tau}_2, \overline{\tau}_3, \dots$.
The resulting map \[ \HZ/2 \wedge \ko \leftarrow \HZ/2[\overline{\xi}_1^2, \overline{\xi}_2, \dots, \overline{\tau}_2, \overline{\tau}_3, \dots]/(\dots) \hookrightarrow \HZ/2 \wedge \HZ/2 \] is an algebra isomorphism.

Similarly $\HZ/2_{**} \kgl \to \HZ/2_{**}\HZ/2$ is injective and \[ \HZ/2 \wedge \kgl \wequi \HZ/2[\overline{\xi}_1, \overline{\xi}_2, \dots, \overline{\tau}_2, \overline{\tau}_3, \dots]/(\dots) \hookrightarrow \HZ/2 \wedge \HZ/2. \]
Also \[ \HZ/2 \wedge \HZ \wequi \HZ/2[\overline{\xi}_1, \overline{\xi}_2, \dots, \overline{\tau}_1, \overline{\tau}_2, \overline{\tau}_3, \dots]/(\dots) \hookrightarrow \HZ/2 \wedge \HZ/2. \]
\end{proposition}
\begin{proof}
Once we have produced the maps in question, to show they are equivalences it suffices to show isomorphisms on $\pi_{**}$ over any field extension.
Since all the maps and spectra are stable under base change, we thus need only prove the assertions on the level of $\pi_{**}$.

This is essentially contained in \cite{ananyevskiy2017very}.
We first explain the case of $\HZ$.
We have the cofiber sequence \[ \HZ \xrightarrow{2} \HZ \xrightarrow{p} \HZ/2 \xrightarrow{\partial} \Sigma \HZ. \]
The map $\id_{\HZ/2} \wedge 2_{\HZ}$ is given by $2=0$ (since $2 \in \pi_{**}(\1)$), hence after smashing with $\HZ/2$ the cofiber sequence splits and we get \[ \HZ/2_{**}\HZ/2 \wequi \HZ/2_{**}\HZ \oplus C. \]
Putting $\delta = p \circ \partial$, we find that $\HZ/2_{**}\HZ = \ker(\delta_*^R)$.
One knows that $\delta = \widehat{\tau_0}$ corresponds to the Bockstein $\Sq^1$.
Using Lemma \ref{lemm:action-formulas} we find that $\delta_*^L$ vanishes on the right $\HZ/2_{**}$-algebra $M$ generated by $\tau_1, \tau_2, \dots, \xi_1, \xi_2, \dots$ and does not vanish in $\tau_0 M$.
Dualizing, we find that $\delta_*^R$ vanishes on $\overline{M}$ and does not vanish on $\overline{\tau}_0 \overline{M}$ (these are now left $\HZ/2_{**}$-modules).
By Lemma \ref{cor:right-gens} we have $\HZ/2_{**}\HZ/2 = \overline{M} \oplus \overline{\tau}_0 \overline{M}$.
It follows that $\ker(\delta_*^R) = \overline{M}$, as desired.

The argument $\kgl$ is essentially the same, using $\Sigma^{2,1}\kgl \xrightarrow{\beta_\KGL} \kgl \to \HZ$.
It is not immediately apparent that $\beta_\KGL: \HZ/2 \wedge \kgl \to \HZ/2 \wedge \kgl$ is the zero map, since $\beta_{\KGL} \not\in \pi_{**}(\1)$.
We can argue as follows.
Over the ring  $\pi_{**}(\HZ/2 \wedge \kgl)$ the formal group laws $x + y + \beta_\KGL xy$ and $x + y$ are isomorphic; in particular the former must have infinite height \cite[Lemma A2.2.9]{ravenel1986complex}.
Thus $0 = [2]_{\kgl}(x) = \beta_\KGL x^2$, whence $\beta_\KGL$ is zero in $\pi_{**}(\HZ/2 \wedge \kgl)$.
Continuing with the above argument, this time it turns out that $\delta = \widehat{\tau_1}$ \cite[Lemma 2.9]{ananyevskiy2017very}; the rest of the argument goes through as before.

For $\ko$ the same argument works, using $\Sigma^{1,1} \ko \xrightarrow{\eta} \ko \to \kgl$.
Since $\eta \in \pi_{**}(\1)$ it is immediate that it acts by $0$ on $\HZ/2 \wedge \ko$.
The boundary map $\delta$ turns out to be $\widehat{\xi_1}$ \cite[Lemma 2.12]{ananyevskiy2017very} \cite[Lemma 13.1]{voevodsky2003reduced}.
The rest goes through as before.
\NB{details?}
\end{proof}

\subsubsection{} \label{subsec:KW}
We denote by $\KW = \KO[\eta^{-1}]$ the \emph{Witt theory spectrum}.
As the name suggest, it represents Balmer--Witt theory \cite{hornbostel2005a1}.
We put $\kw = \KW_{\ge 0}$.
The image of the Bott element $\beta \in \pi_{8,4} \KO$ yields an element $\beta \in \pi_4 \KW = \pi_4 \kw$ and we have \cite[Theorem 1.5.22]{balmer2005witt} \[ \ul{\pi}_*(\KW) \wequi \ul{W}[\beta^{\pm}] \quad\text{and}\quad \ul{\pi}_*(\kw) \wequi \ul{W}[\beta]. \]
\begin{lemma} \label{lemm:ko-kw}
The canonical map $\ko \to \KO \to \KW$ induces an equivalence $\ko[\eta^{-1}] \wequi \kw$.
\end{lemma}
\begin{proof}
Since $\ko \in \SH(k)_{\ge 0}$, the map $\ko \to \KW$ indeed factors through $\kw$.
We have \[ \ul{\pi}_*(\ko[\eta^{-1}]) \wequi \colim \left[ \ul{\pi}_*(\ko)_0 \xrightarrow{\eta} \ul{\pi}_*(\ko)_{-1} \xrightarrow{\eta} \dots \right]. \]
Since $\ul{\pi}_*(\ko)_{-n} \wequi \ul{\pi}_*(\KO)_{-n}$ for $*,n \ge 0$ and $\ul{\pi}_*(\ko)_{-n} = 0$ for $* < 0$, the result follows.
\end{proof}
We put $\HW = \1[\eta^{-1}]_{\le 0}$; in other words this is just the homotopy module $\ul{W}[\eta^\pm]$.
The unit map induces \begin{equation}\label{eq:HW-kw-trunc} \HW \wequi \kw_{\le 0} \wequi \kw/\beta. \end{equation}
\begin{warning}
In other works $\HW$ would perhaps have been denoted $\ul{K}^W[\eta^{-1}]$ or $\ul{W}[\eta^\pm]$ (and $\HW$ would have denoted something else).
Since this object is so central for our work, we reserve the prominent notation.
\end{warning}

\subsubsection{} \label{subsec:f0HW}
Consider the spectrum $f_0 \HW$.
One may show (e.g. see \cite[Theorem 17]{bachmann-very-effective}) that $\ul{\pi}_0(f_0 \HW)_* \wequi \ul{K}^W_*$ is the homotopy module of \emph{(unramified) Witt $K$-theory}; in other words $\ul{K}^W_* = \ul{I}^*$ (and multiplication by $\eta$ induces the inclusion $\ul{I}^{*+1} \hookrightarrow \ul{I}^*$; for $*<0$ we put $\ul{I}^* = \ul{W}$).
There is a canonical map $\ul{K}_*^W \to \ul{k}^M$ \cite{morel2004puissances}; by the resolution of the Milnor conjecture  \cite{voevodsky2003motivic,orlov2007exact,morel2005milnor} this induces a cofiber sequence \begin{equation} \label{eq:KW-mod-eta} \Sigma^{1,1} \ul{K}^W \xrightarrow{\eta} \ul{K}^W \to \ul{k}^M. \end{equation}
Taking effective covers, we obtain a map $f_0 \HW \to f_0 \ul{k}^M \wequi \HZ/2$ \cite[Lemma 12]{bachmann-very-effective}.
The induced square
\begin{equation} \label{eq:structure-conj}
\begin{CD}
f_0 \HW @>>> \ul{K}^W \\
@VVV         @VVV     \\
\HZ/2   @>>> \ul{k}^M
\end{CD}
\end{equation}
is cartesian \cite[Theorem 17]{bachmann-very-effective}.
\begin{lemma} \label{lemm:tau-tilde}
The map $\tau: \Sigma^{0,-1} \HZ/2 \to \HZ/2$ lifts uniquely (up to homotopy) to a map $\tilde{\tau}: \Sigma^{0,-1} \HZ/2 \to f_0 \HW$, and $\cof(\tilde\tau) \wequi \ul{K}^W.$
\end{lemma}
\begin{proof}
Since $\Sigma^{0,-1} \HZ/2 \wequi \Sigma \Gmp{-1} \wedge \HZ/2 \in \SH(k)_{\ge 1}$ and $\ul{K}^W, \ul{k}^M \in \SH(k)_{\le 0}$, the long exact sequence for $[\Sigma^{0,-1} \HZ/2, \ph]$ shows that there is a unique lift as claimed.
We can view this as a morphism from the bicartesian square
\begin{equation*}
\begin{CD}
\Sigma^{0,-1} \HZ/2 @>>> 0 \\
@V{\id}VV                  @VVV \\
\Sigma^{0,-1} \HZ/2 @>>> 0
\end{CD}
\end{equation*}
into \eqref{eq:structure-conj}; taking cofibers we get a bicartesian square
\begin{equation*}
\begin{CD}
\cof(\tilde \tau) @>>> \ul{K}^W \\
@VVV                  @VVV \\
\cof(\tau) @>>> \ul{k}^M.
\end{CD}
\end{equation*}
The bottom horizontal map is an equivalence by \eqref{eq:cof-tau-HZ2}, hence so is the top one.
This was to be shown.
\end{proof}
\begin{corollary} \label{cor:K-W-eff}
We have $\ul{K}^W \in \Sigma^{0,-1} \SH(k)^\veff$.
\end{corollary}
\begin{proof}
Since $\SH(k)^\veff$ is closed under colimits (and $\HZ/2, f_0\HW \in \SH(k)^\veff$\NB{justification?}), this is immediate from Lemma \ref{lemm:tau-tilde}.
\end{proof}

\subsection{Determination of \texorpdfstring{$\pi_*(\kw \wedge \HW)_2^\comp$}{pi*(kw wedge HW)}}\discuss{Mike: I think this might interest you.}
\subsubsection{}
For $E \in \SH(k)$, consider the $\eta$-multiplication tower \begin{equation} \label{eq:eta-tower} \dots \to \Sigma^{2,2} E \xrightarrow{\eta} \Sigma^{1,1} E \xrightarrow{\eta} E. \end{equation}
Functorially associated with this is a spectral sequence (the $\eta$-Bockstein spectral sequence) with \cite[beginning of \S1.2.2 and Construction 1.2.2.6]{lurie-ha} \begin{equation} \label{eq:eta-ss} \begin{gathered} E_1^{p,q,w} = \pi_{p+q}(\cof(\eta: \Sigma^{-p+1,-p+1} E \to \Sigma^{-p,-p} E))_w, \quad p \le 0 \\ d_r: E_r^{p,q,w} \to E_{r+1}^{p-r,q+r-1,w} \\ \rightsquigarrow \pi_{p+q}(E)_w. \end{gathered} \end{equation}
Here by the last line we mean that the $E_\infty$-page in position $(p,q,w)$ is is related to $\pi_{p+q}(E)_w$ (but we are not claiming any kind of convergence).

\begin{remark} \label{rmk:basechange-ss}
Since the tower \eqref{eq:eta-tower} is compatible with base change, so is the associated spectral sequence \eqref{eq:eta-ss}.
\end{remark}

The boundary map in the cofiber sequence \[ \Sigma^{1,1} E \xrightarrow{\eta} E \xrightarrow{p} E/\eta \xrightarrow{\partial} \Sigma^{2,1} E \] induces the Bockstein \[ \delta = p\partial: E/\eta \to \Sigma^{2,1} E/\eta. \]
By construction, $\delta^2 = 0$ and so $\delta_*$ gives $\pi_{*}(E/\eta)_*$ the structure of a chain complex.
We write its homology (respectively cycles, respectively the entire complex) in spot corresponding to $\pi_a(E/\eta)_b$ as $H_a(\pi_{*}(E/\eta)_*, \delta_*)_b$ (respectively $Z_a(\pi_{*}(E/\eta)_*, \delta_*)_b$, $C_a(\pi_{*}(E/\eta)_*, \delta_*)_b$).
Recall the notion of conditional convergence from \cite[Definition 5.10]{boardman1999conditionally}.
\begin{lemma} \label{lemm:reindexing-convergence}
By suitably re-indexing the spectral sequence \eqref{eq:eta-ss} we obtain a conditionally convergent spectral sequence
\begin{gather*}
  E_1^{s,f,w} = \pi_s(E/\eta)_{w+f} \Rightarrow \pi_s(E_\eta^\comp)_w \\
  d_r: E_r^{s,f,w} \to E^{s-1,f+r,w}_r.
\end{gather*}
Here $E_1^{s,f,w} = 0$ for $f < 0$.
We have \begin{gather*} E_2^{s,0,w} = Z_s(\pi_{*}(E/\eta)_*, \delta_*)_w \\ E_2^{s,f,w} = H_s(\pi_{*}(E/\eta)_{*}, \delta_*)_{f+w}, f > 0. \end{gather*}
\end{lemma}
\begin{proof}
We have $\cof(\eta: \Sigma^{-p+1,-p+1} E \to \Sigma^{-p,-p} E) \wequi \Sigma^{-p,-p} E/\eta$ and hence $E_1^{p,q,w} \wequi \pi_{p+q}(E/\eta)_{w-p}$.
The re-indexing is obtained by putting $s=p+q$ and $f=-p$.

By Lemma \ref{lemm:boardman-cond-conv} below, the spectral sequence converges conditionally to \[ \cof(\lim_p \Sigma^{p,p} E \to E) \wequi \lim_p \cof(\Sigma^{p,p} E \xrightarrow{\eta^p} E) \wequi \lim_p E/\eta^p \wequi E_\eta^\comp. \]

By construction, the $d_1$-differentials are induced by $\delta_*$, whence the identification of the $E_2$-page.
\end{proof}

\begin{remark} \label{rmk:multiplicative-ss}
If $E \to E/\eta$ is an $\scr E_\infty$-ring map, then the spectral sequence above can be identified with the descent spectral sequence for this map.
In particular, it is multiplicative.
Furthermore if $G$ is an $E$-module, then the spectral sequence for $G$ is a module over the one for $E$.
\end{remark}

In the proof of Lemma \ref{lemm:reindexing-convergence}, we have made use of the following well-known fact.
Let $E_\bullet: \Z^\op \to \SH$ be a tower of spectra.
Then as above there is an associated spectral sequence $E^{p,q}_n(E_\bullet)$.
\begin{lemma} \label{lemm:boardman-cond-conv} \discuss{right?}
The spectral sequence $E^{p,q}_n(E_\bullet)$ converges conditionally to $\cof(\lim E_\bullet \to \colim E_\bullet)$.
\end{lemma}
\begin{proof}
Let $E'_p = \cof(\lim E_\bullet \to E_p)$.
Then there is a morphism of towers $E_\bullet \to E'_\bullet$ inducing a morphism of spectral sequences $E^{p,q}_n(E_\bullet) \to E^{p,q}_n(E'_\bullet)$.
By construction, this induces an isomorphism on the $E_1$-page, and hence on all following pages.
Noting that $\lim E'_\bullet \wequi 0$ and $\colim E'_\bullet \wequi \cof(\lim E_\bullet \to \colim E_\bullet)$, we may replace $E_\bullet$ by $E'_\bullet$, and so assume that $\lim E_\bullet \wequi 0$.
Conditional convergence to the colimit means by definition \cite[Definition 5.10]{boardman1999conditionally} that \[ \lim \pi_i(E_\bullet) \wequi 0 \wequi \limone \pi_i(E_\bullet), \text{ for } i \in \Z. \]
By the Milnor exact sequence \cite[Proposition VI.2.15]{goerss2009simplicial}, this follows from (and is in fact equivalent to) $\lim E_\bullet \wequi 0$.
\end{proof}

\subsubsection{}
Applying Remark \ref{rmk:multiplicative-ss} to the map (see \eqref{eq:KW-mod-eta}) \[ \ul{K}^W \to \ul{K}^W/\eta \wequi \ul{k}^M \] and $G = E \wedge \ul{K}^W$, we obtain the spectral sequence \begin{equation} \label{eq:eta-ss2} \begin{gathered} E_1(G)^{*,*,*} = C(\pi_*(E \wedge \ul{k}^M)_*, \delta_*)[h] \Rightarrow \pi_*((E \wedge \ul{K}^W)_\eta^\comp)_* \\ E_2(G)^{*,*,*} = Z(\pi_*(E \wedge \ul{k}^M)_*, \delta_*)[h]/h \cdot \im(\delta_*) \end{gathered}. \end{equation}
Here \[ h = 1 \in E^1(\ul{K}^W)^{0,1,-1} \wequi \ul{k}^M(k)_0, \] and we have used the module structure to act with $h \in E_1(\ul{K}^W)$ on $E_1(G)$.
\begin{example} \label{ex:ss-sphere}
Taking $E=\1$, we get $E_1(G)^{s,f,w} = 0$ unless $s=0$, so the spectral sequence collapses at $E_1$.
The spectral sequence converges to $\pi_0((\ul{K}^W)_\eta^\comp)_* = (\ul{K}^W(k)_*)_\I^\comp$; on $E_1 = E_\infty$ we see the $\I$-adic filtration (as we must) with subquotients given by $\ul{k}^M(k)_*$.
The element $h$ detects $\eta \in \ul{K}^W(k)_{-1}$.
In particular $h$ is a permanent cycle.
\end{example}

Via the spectral sequence, we obtain a filtration $F^\bullet \pi_*(G_\eta^\comp)_*$ on the bigraded group $\pi_*(G_\eta^\comp)_*$.
Multiplication by $\eta$ induces a map $\pi_*(G_\eta^\comp)_* \to \pi_*(G_\eta^\comp)_{*-1}$ which maps $F^\bullet$ to $F^{\bullet + 1}$ and on associated graded corresponds to multiplication by $h$.
Taking the colimit we obtain a filtration on \[ \pi_*(G_\eta^\comp)_*[\eta^{-1}] = \colim_i \pi_*(G_\eta^\comp)_{*-i} \wequi \pi_*(G_\eta^\comp[\eta^{-1}]) \] with \begin{equation}\label{eq:eta-inverted-filt} F^\bullet \pi_*(G_\eta^\comp[\eta^{-1}])_* = \colim_i F^{\bullet + i} \pi_*(G_\eta^\comp)_{*-i} \end{equation} and associated graded \[ \gr^\bullet \pi_*(G_\eta^\comp[\eta^{-1}])_* \wequi E_\infty(G)^{*,\bullet,*}[h^{-1}]. \]
Note that even though the filtration $F^\bullet \pi_*(G_\eta^\comp)_*$ terminates at $F^0$ (i.e. $F^0 \pi_*(G_\eta^\comp)_* = \pi_*(G_\eta^\comp)_*$) this need not be the case for the induced filtration on $\pi_*(G_\eta^\comp[\eta^{-1}])_*$: we could well have \[ F^0 \subsetneq F^{-1} \subsetneq F^{-2} \subsetneq \dots. \]
Note also that we are not making any claim about the completeness etc. of the filtrations.

\subsubsection{}
We may wish to apply spectral sequence \eqref{eq:eta-ss2} with $E = \HZ/2$.
To begin with, using that $(\HZ/2)/\tau \wequi \ul{k}^M$ (see \eqref{eq:cof-tau-HZ2}), the form of $\HZ/2_{**}\HZ/2$ (see \eqref{eq:motivic-dual-steenrod}) and the fact that $\eta_R(\tau) = \tau + \rho\tau_0$ (Lemma \ref{lemm:right-unit}) we find that \begin{equation} \label{eq:HW-km} \pi_{**}(\HZ/2 \wedge \ul{k}^M) \wequi \ul{k}^M_*(k)[\tau_0, \tau_1, \dots, \xi_1, \xi_2, \dots]/(\tau_i^2 - \rho \tau_{i+1}). \end{equation}
Now we determine the differential.
\begin{lemma} \label{lemm:determine-delta}
The action of $\delta_*$ on $\pi_{**}(\HZ/2 \wedge \ul{k}^M)$ satisfies \begin{equation*} \delta_*(\tau_i) = 0 \quad\text{and}\quad \delta_*(\xi_i) = \xi_{i-1}^2 \end{equation*}
\end{lemma}
\begin{proof}
We claim that there exists a commutative square
\begin{equation*}
\begin{CD}
\ul{k}^M @>{\delta}>> \Sigma^{2,1} \ul{k}^M \\
@AAA                              @AAA \\
\HZ/2 @>{\tilde\delta}>> \Sigma^{2,1} \HZ/2,
\end{CD}
\end{equation*}
where the vertical maps are the canonical projections.
Indeed using the cofiber sequence \[ \Sigma^2 \HZ/2 \xrightarrow{\tau} \Sigma^{2,1} \HZ/2 \to \Sigma^{2,1} \ul{k}^M \to \Sigma^3 \HZ/2 \] it suffices to show that $[\HZ/2, \Sigma^3 \HZ/2] = 0$, which follows from the form of the motivic Steenrod algebra (see Example \ref{ex:steenrod-explicit}).
We have (again by Example \ref{ex:steenrod-explicit}) \[ [\HZ/2, \Sigma^{2,1}\HZ/2] \wequi \F_2\{\widehat{\xi_1}\} \oplus \ul{k}_1^M(k) \{\widehat{\tau_0}\}, \] so that $\tilde\delta = a \widehat{\xi_1} + b\widehat{\tau_0}$, for some $a \in \F_2$ and $b \in \ul{k}_1^M(k)$.
Comparison with Lemma \ref{lemm:action-formulas} (and noting that we are looking at $\tilde\delta_*^R$) yields \[ \delta_*(\tau_i) = b\xi_i \text{ and } \delta_*(\xi_i) = a\xi_{i-1}^2. \]
Since $\eta=0$ on $\HZ/2$, smashing the cofiber sequence for $\ul{K}^W/\eta \wequi \ul{k}^M$ with $\HZ/2$ we obtain a splitting \[ \pi_{**}(\HZ/2 \wedge \ul{k}^M) \wequi \pi_{**}(\HZ/2 \wedge \ul{K}^W) \oplus \pi_{**}(\Sigma^{2,1}\HZ/2 \wedge \ul{K}^W), \] with the first summand given by $\ker(\delta_*)$.
Since $\HZ/2 \wedge \ul{k}^M \ne 0$ (e.g. $\ul{\pi}_0(\ph)_* = \ul{k}^M_*$) we find that neither of the (isomorphic) summands can be trivial, and so $\delta_* \ne 0$.
After base change to an algebraically closed field we get $\ul{k}_1^M(\bar k) = 0$, so that the only way to get $\delta_* \ne 0$ is $a=1$.
Now we compute \[ \delta_*(\delta_*(\tau_1)) = \delta_*(b \xi_1) = b. \]
But $\delta^2=0$, so that $b=0$.
The result follows.
\end{proof}

\subsubsection{}
We now apply the spectral sequence with $E = \ko$.
Recall that we have determined $\pi_{**}(\HZ/2 \wedge \ul{k}^M)$ in \eqref{eq:HW-km}.
\begin{lemma}
The map $\ko \to \HZ/2$ induces an isomorphism \[ \pi_{**}(\ko \wedge \ul{k}^M) \wequi \ul{k}^M_*(k)[\xi_1^2, \xi_2, \dots, \tau_2, \tau_3, \dots]/(\tau_i^2 - \rho \tau_{i+1}) \hookrightarrow \pi_{**}(\HZ/2 \wedge \ul{k}^M). \]
The homology of $\delta_*$ acting on this is given by $\ul{k}^M_*(k)[\tau_2, \tau_3, \dots]/(\tau_i^2 - \rho \tau_{i+1})$.
\end{lemma}
\begin{proof}
Applying antipodes in Lemma \ref{prop:HZ2-ko}, we find that $\pi_{**}(\ko \wedge \HZ/2) \hookrightarrow \pi_{**}(\HZ/2 \wedge \HZ/2)$ is the \emph{right} $\HZ/2_{**}$-algebra generated by $\xi_1^2, \xi_2, \dots, \tau_2, \tau_3, \dots$.
Since these generators form part of a right $\HZ/2_{**}$-basis (see Corollary \ref{cor:right-gens}), multiplication by $\eta_R(\tau)$ is injective and we obtain $\pi_{**}(\ko \wedge \ul{k}^M)$ as the quotient.
This proves the first claim.

It follows from Remark \ref{rmk:multiplicative-ss} that $\delta_*$ is a derivation, which by Lemma \ref{lemm:determine-delta} satisfies $\delta_*(\tau_i) = 0$, $\delta_*(\xi_i) = \xi_{i-1}^2$ (and also $\delta_*(\ul{k}_*^M(k)) = 0$, since these elements come from the sphere).
This implies that the homology of $\delta_*$ is given by \[ \ul{k}_*^M(k)[\tau_2, \tau_3, \dots]/(\tau_i^2 - \rho \tau_{i+1}) \otimes H', \] where $H'$ is the homology of $\delta_*$ restricted to the subring \[ \F_2[\xi_1^2, \xi_2, \dots]. \]
We can determine this as follows\discuss{better argument?}, adapting \cite[Proof of \S3 Lemma 16.9]{adams1995stable}.
Let $F_i \subset \F_2[\xi_1^2, \xi_2, \dots]$ be the $\F_2$-vector space with basis $\{\xi_i^{2n}, \xi_i^{2n}\xi_{i+1}\}_{n \ge 0}$.
Then $\delta_*(F_i) \subset F_i$ and $H_*(F_i, \delta) = \F_2\{1\}$.
Since we are working over a field, \[ H_*\left(\bigotimes_i F_i\right) \wequi \bigotimes_i H_*(F_i) \wequi \F_2\{1\}. \]
It thus remains to observe that $\bigotimes_i F_i = \F_2[\xi_1^2, \xi_2, \dots]$; equivalently every monomial (in which $\xi_1$ occurs to even power) can be written uniquely as a product $\prod_i m_i$, with $m_i$ one of the basis elements of $F_i$.
This is easily checked.\NB{Take $m_i = \xi_i^{2n}\xi_{i+1}^\epsilon$, where $2n + \epsilon'$ is the exponent of $\xi_i$, $\epsilon, \epsilon' \in \{0,1\}$ and $\epsilon$ is the parity of the exponent of $\xi_{i+1}$.}
\end{proof}

\begin{lemma}
The spectral sequence $E_*(\ko \wedge \ul{K}^W)^{*,*,*}$ collapses at $E_2$.
\end{lemma}
In other words, there are no further differentials, and $E_2 = E_\infty$.
In particular the spectral sequence converges strongly \cite[Theorem 7.1]{boardman1999conditionally}.
\begin{proof}
We have $E_2(\ko \wedge \ul{k}^W) = Z(\delta_*)[h]/h \cdot \im(\delta_*)$ and so in particular $E_2(\ko \wedge \ul{K}^W)^{*,f>0,*} = h\cdot H(\delta_*)[h]$.
Recall that the generator $\tau_i$ has bidegree $(2^{i+1}-1,2^i-1)$, so that $\tau_i \in \pi_{2^i}(\HZ/2 \wedge \HZ/2)_{1-2^i}$.
Thus $\tau_i$ defines an element with $s=2^i$ and $w=1-2^i$ (and $f=0$) in our spectral sequence.
Similarly $\ul{k}_*^M(k)$ yields elements with $s=0,w=*,f=0$.
Since $i \ge 2$, it follows (using that $|h|=(0,1,-1)$) that $E_2(\ko \wedge \ul{K}^W)^{*,f>0,*}$ is concentrated in stems $s \equiv 0 \pmod{4}$.
Since all differentials lower $s$ by $1$, we find that any differential emanating from positive filtration vanishes.

Now we prove by induction on $r$ that all differentials vanish on $E_r$, starting with $r=2$.
By the induction hypothesis we have $E_r = E_2$.
Let $d_r(x) = y$ be any differential.
Then $d_r(hx) = hd_r(x) = hy$ is another differential, $h$ being a permanent cycle (see Example \ref{ex:ss-sphere}).
Since $hx$ has positive filtration $f$ (since $f(x) \ge 0$ and $f(h) = 1$), by the above we have $hy = 0$.
Since multiplication by $h$ is injective on $E_2^{*,f>0,*} = E_r^{*,f>0,*}$ and $f(y) > 0$, we deduce that $y=0$.
This was to be shown.
\end{proof}

Now we invert $\eta$ to obtain a new filtration, as in \eqref{eq:eta-inverted-filt}.
\begin{lemma}
The filtration $F^\bullet \pi_*((\ko \wedge \ul{K}^W)_\eta^\comp[\eta^{-1}])_*$ is complete, Hausdorff and exhaustive.
\end{lemma}
\begin{proof}
By strong convergence of the spectral sequence, the filtration $F^\bullet \pi_*(\ko \wedge \ul{K}^W)_\eta^\comp)_*$ is complete, Hausdorff and exhaustive.
It is clear that exhaustive filtrations are stable under colimits.
We thus need to show completeness and Hausdorffness, or in other words that \[ R\lim_i F^i \pi_*((\ko \wedge \ul{K}^W)_\eta^\comp[\eta^{-1}])_* \wequi 0. \]
We claim that for $i > 0$ the map \[ \eta: F^i \pi_*((\ko \wedge \ul{K}^W)_\eta)_* \to F^{i+1} \pi_*((\ko \wedge \ul{K}^W)_\eta)_{*-1} \] is an isomorphism.
Indeed giving $F^i \pi_*((\ko \wedge \ul{K}^W)_\eta)_*$ the obvious induced filtration, this becomes a morphism of complete, Hausdorff, exhaustively filtered groups inducing an isomorphism on associated gradeds (given by $h: h^i H(\delta_*)[h] \to h^{i+1}H(\delta_*)[h]$), so this follows from Lemma \ref{lemm:filtration-comparison}.
Hence for $i>0$ the canonical map \[ F^i \pi_*((\ko \wedge \ul{K}^W)_\eta)_* \to F^i \pi_*((\ko \wedge \ul{K}^W)_\eta[\eta^{-1}])_* \] is an isomorphism
Thus \[ 0 \wequi R\lim_i F^i \pi_*((\ko \wedge \ul{K}^W)_\eta^\comp)_*  \xrightarrow{\wequi} R\lim_i F^i \pi_*((\ko \wedge \ul{K}^W)_\eta^\comp[\eta^{-1}])_*, \] and the result follows (here the first equivalence holds since the filtration on $\pi_*((\ko \wedge \ul{K}^W)_\eta^\comp)_*$ is complete and Hausdorff, as noted in the beginning).
\end{proof}

\subsubsection{Proof of Theorem \ref{thm:HW-kw-summary}}
We first compute $\pi_*((\ko \wedge \ul{K}^W)_\eta^\comp[\eta^{-1}])_*$.
This is a filtered graded module over \[ \pi_*((\ul{K}^W)_\eta^\comp[\eta^{-1}])_* \wequi \W(k)_\I^\comp[\eta^\pm]. \]
Here a priori $\W(k)_\I^\comp$ must mean the derived $\I$-adic completion, but since $\vcd_2(k) < \infty$ the $\I$-adic and $2$-adic completion agree (by Lemma \ref{lemm:witt-completion}), and the ordinary completion agrees with the derived completion (by Lemmas \ref{lemm:witt-torsion} and \ref{lemm:derived-p-comp}).
Note that $\gr^\bullet \pi_*((\ul{K}^W)_\eta^\comp[\eta^{-1}])_* = \ul{k}_*^M(k)[h^\pm]$.
Put $u_{i-2} = h^{1-2^i} \tau_i$; we have $(s,f,w)(u_i) = 4(2^i, 1-2^i, 0)$.
For $n = \sum_i \epsilon_i 2^i$ put \[ y_n = \prod_i u_i^{\epsilon_i}. \]
Observe that \[ \gr^\bullet \pi_*((\ko \wedge \ul{K}^W)_\eta^\comp[\eta^{-1}])_* \wequi \ul{k}_*^M(k)[h^\pm, \tau_2, \tau_3, \dots]/(\tau_i^2 - \rho\tau_{i+1}) \wequi \ul{k}_*^M(k)[h^\pm]\{y_0, y_1, \dots \}. \]
This is a free module on $\ul{k}_*^M[h^\pm]$ with at most one generator in every degree, and all our filtrations are complete, Hausdorff and exhaustive, so by Corollary \ref{corr:filtered-free-lifting} we get \[ \pi_*((\ko \wedge \ul{K}^W)_\eta^\comp[\eta^{-1}])_* \wequi \W(k)_\I^\comp[\eta^\pm]\{x_0, x_1, \dots \}, \] where $x_i$ is a lift of $y_i$.

It follows from Corollary \ref{cor:K-W-eff} and Theorem \ref{thm:eta-complete} that $(\ko \wedge \ul{K}^W)_2^\comp \wequi (\ko \wedge \ul{K}^W)_{2,\eta}^\comp$ and hence \[ ((\ko \wedge \ul{K}^W)_\eta^\comp[\eta^{-1}])_2^\comp \wequi ((\ko \wedge \ul{K}^W)_{2,\eta}^\comp[\eta^{-1}])_2^\comp \wequi ((\ko \wedge \ul{K}^W)_2^\comp[\eta^{-1}])_2^\comp \wequi (\ko \wedge \ul{K}^W[\eta^{-1}])_2^\comp. \]
Since $\ko[\eta^{-1}] \wequi \kw$ (Lemma \ref{lemm:ko-kw}) and $\ul{K}^W[\eta^{-1}] \wequi \HW$, we deduce that \[ (\kw \wedge \HW)_2^\comp \wequi ((\ko \wedge \ul{K}^W)_\eta^\comp[\eta^{-1}])_2^\comp. \]
Note that the $\W(k)_\I^\comp$ is derived $2$-complete, and hence all of $\pi_{*}((\ko \wedge \ul{K}^W)_\eta^\comp[\eta^{-1}])$ is.
In other words \[ \pi_{*}((\ko \wedge \ul{K}^W)_\eta^\comp[\eta^{-1}]) \wequi \pi_{*}(((\ko \wedge \ul{K}^W)_\eta^\comp[\eta^{-1}])_2^\comp) \] (by Lemma \ref{lemm:completion-homotopy}).
We have thus computed \[ \pi_{*}((\HW \wedge \kw)_2^\comp) \wequi \W(k)_\I^\comp\{x_0, x_1, \dots\}. \]

(1) Since our spectral sequence as well as the motivic dual Steenrod algebra are stable under base change (Remark \ref{rmk:basechange-ss}), so are the $y_i$.
Corollary \ref{corr:filtered-free-lifting} shows that any set of lifts will generate.

(2) Let $t_i$ be a lift of $u_i$.
Then $t_i^2$ is detected by $u_i^2 = u_{i+1}h\rho$.
Since \[ h\rho = -\eta[-1] = -(\lra{-1}-1) = 2\in \W(k) \] we have \[ t_i^2 = 2t_{i+1} + (\text{higher filtration}). \]

(3) True by construction.

(5) Clearly $1$ lifts $y_0 = 1$, whence the claim.

This concludes the proof.

\section{Main theorem} \label{sec:main}
\begin{lemma} \label{lemm:unit-connectivity}
The unit map $u: \1[\eta^{-1}] \to \kw$ is $1$-connected: $\cof(u) \in \SH(k)_{\ge 2}$.
\end{lemma}
\begin{proof}
Immediate from examination of the homotopy sheaves (see \S\ref{subsec:KW} and \S\ref{subsub:pi0-sphere}).
\end{proof}

\begin{corollary} \label{cor:construct-phi}
\begin{enumerate}
\item Let $\adamspsi: \kw_{(2)} \to \kw_{(2)}$ be any map such that $\adamspsi(1) = 1$.
  Then in the following commutative diagram, the dotted arrow can be filled uniquely up to homotopy.
\begin{equation*}
\begin{tikzcd}
 & \Sigma^4 \kw_{(2)} \ar[d, "\beta"] \\
\kw_{(2)} \ar[r, "\adamspsi-1"] \ar[ru, dotted, "\adamsphi"]  & \kw_{(2)}
\end{tikzcd}
\end{equation*}

\item Let $\adamsphi: \kw_{(2)} \to \Sigma^4 \kw_{(2)}$ be any map.
  Then in the following commutative diagram, the dotted arrow can be filled uniquely up to homotopy.
\begin{equation*}
\begin{tikzcd}
  & \1[\eta^{-1}]_{(2)} \ar[d, "u"] \ar[dl, dotted] \\
\fib(\adamsphi) \ar[r] & \kw_{(2)} \ar[r, "\adamsphi"] & \Sigma^4 \kw_{(2)}
\end{tikzcd}
\end{equation*}
\end{enumerate}
\end{corollary}
\begin{proof}
Note that ($*$) by Lemma \ref{lemm:unit-connectivity}, if $E \in \SH(k)_{\le 0}$ then $[\kw, E] \wequi [\1[\eta^{-1}], E]$.
We have $\cof(\beta: \Sigma^4 \kw_{(2)} \to \kw_{(2)}) \wequi \HW_{(2)}$ (see \eqref{eq:HW-kw-trunc}); hence $\adamsphi$ exists if (a) $\kw \xrightarrow{\adamspsi-1} \kw_{(2)} \to \HW_{(2)}$ is zero, and is unique if (b) $[\kw, \Sigma^{-1}\HW_{(2)}] = 0$.
Supposing that $\adamsphi$ exists, denote its fiber by $F$.
The factorization $\1 \to F$ exists if (c) $[\1, \Sigma^4\kw_{(2)}] = 0$ and is unique if (d) $[\1, \Sigma^3\kw_{(2)}]=0$.

We have $\HW_{(2)} \in \SH(k)_{\le 0}$, whence by ($*$) for (a) it suffices to show that \[ \1 \xrightarrow{1} \kw \xrightarrow{\adamspsi-1} \kw_{(2)} \to \HW_{(2)} \] is zero, which holds since $\adamspsi(1) = 1$ by assumption.
For (b), again by ($*$) it is enough to show that $[\1, \Sigma^{-1}\HW_{(2)}] = 0$.
This is clear since $\Sigma^{-1} \HW_{(2)} \in \SH(k)_{<0}$.
(c) and (d) are immediate.
\end{proof}

We apply Corollary \ref{cor:construct-phi} to the map $\adamspsi = \adamspsi^3$ constructed in Remark \ref{rmk:adams-kw} to obtain $\adamsphi: \kw_{(2)} \to \Sigma^4\kw_{(2)}$.
From now on, this is the only map we will denote by $\adamsphi$.
\begin{remark} \label{rmk:phi-defn}
The defining property of $\adamsphi$ implies that for $E \in \SH(k)_{(2)}$ and $a \in \kw_*(E)$ we get $\beta \adamsphi(a) = \adamspsi^3(a) - a$.
In particular, if $\kw_*E$ is $\beta$-torsion free, then \begin{equation}\label{eq:adamsphi-compute} \adamsphi(a) = (\adamspsi^3(a)-a)/\beta. \end{equation}
Specializing even further, assume that $2=0 \in \W(k)$.
Then $\adamspsi^3(\beta)=9\beta=\beta$ (by Theorem \ref{thm:ko-adams}(2)) and hence \begin{equation}\label{eq:beta-phi-commute} \adamsphi(\beta x) = \beta\adamsphi(x). \end{equation}
\end{remark}
\begin{remark}
Arguing as in Example \ref{ex:adamspsi-pi0}, we see that $\adamsphi: \kw_* E \to \kw_{*-4}E$ is $\W(k)_{(2)}$-linear.
\end{remark}
\begin{example} \label{ex:phi-beta}
Since $\kw_* \wequi \W(k)[\beta]$ is $\beta$-torsion free, we deduce from Theorem \ref{thm:ko-adams}(2) that \[ \adamsphi(\beta^n) = (\adamspsi^3(\beta^n) - \beta^n)/\beta = (\adamspsi^3(\beta)^n-\beta^n)/\beta = (9^n-1)\beta^{n-1}. \]
\end{example}

Recall from Theorem \ref{thm:cobordism-MSL-summary}(2) (and Example \ref{ex:SL-orientations}) that \[ \kw_* \MSL \wequi \kw_*[e_0, e_2, e_4, \dots]/(e_0-1). \]
\begin{lemma} \label{lemm:phi-MSL}
Suppose that $\W(k) \wequi \F_2$.
Consider the action of $\adamsphi$ on $\kw_* \MSL$.
\begin{enumerate}
\item We have $\adamsphi(e_{2i}) \in e_{2i-2} + \beta \kw_* \MSL$.
\item We have $\adamsphi^{\circ i}(e_{2i}) = 1$.
\end{enumerate}
\end{lemma}
\begin{proof}
Since $\kw_*\MSL \wequi \kw_*[e_2, \dots]$ is $\beta$-torsion free, by \eqref{eq:adamsphi-compute} and Proposition \ref{prop:adams-cobordism-summary} we get \[ \adamsphi(e_{2i}) \in \beta^{-1}(e_{2i} + \beta e_{2i-2} + \beta^2 \kw_*\MSL - e_{2i}) = e_{2i-2} + \beta \kw_* \MSL, \] whence (1).
By \eqref{eq:beta-phi-commute}, $\adamsphi$ commutes with $\beta$, and hence by iteration we find that $\adamsphi^{\circ i}(e_{2i}) = 1 + a_i \beta$, for some $a_i \in \kw_{-4} \MSL$.
(2) follows since $\kw_{-4} \MSL = 0$.
\end{proof}

\begin{proposition} \label{prop:kw-HW-better}
Consider the maps \[ \pi_*(\kw \wedge \MSL) \xrightarrow{\alpha} \pi_*(\kw \wedge \kw_{(2)}) \xrightarrow{r} \pi_*(\kw \wedge \HW_{(2)}). \]
Put $\tilde x_i = \alpha(e_{2i})$ and $x_i = r \tilde x_i$.
\begin{enumerate}
\item The canonical maps induce equivalences (of right modules) \[ \kw \wedge \kw_{(2)} \wequi \bigvee_i \kw_{(2)}\{\tilde x_i\} \text{ and } \kw \wedge \HW_{(2)} \wequi \bigvee_i \HW_{(2)}\{x_i\}. \]
\item $x \in \pi_{4i}(\kw \wedge \HW_{(2)}) \wequi \W(k)_{(2)}$ is a generator if and only if $\adamsphi^{\circ i}(x)$ generates $\pi_0(\kw \wedge \HW_{(2)})$.
\end{enumerate}
\end{proposition}
\begin{proof}
(1) We first prove the claim about $\kw \wedge \HW$.
By stability under base change, we may assume that $\vcd_2(k) < \infty$.
Since the field is arbitrary, it suffices to show that the map induces an isomorphism on $\pi_*$.
This we can check separately after inverting $2$ and after completing at $2$ (see Lemma \ref{lemm:inv-compl}).
Note that \eqref{eq:serre-finiteness} implies that $\1[\eta^{-1}] \otimes \Q \wequi \HW \otimes \Q$ and hence $(\kw \wedge \HW) \otimes \Q \wequi \kw \otimes \Q$, and so \[ \pi_* (\kw \wedge \HW) \otimes \Q \wequi \begin{cases} \W(k) \otimes \Q & * = 4n \ge 0 \\ 0 & \text{else} \end{cases}. \]
Similarly by Theorem \ref{thm:HW-kw-summary} we have \[ \pi_* ((\kw \wedge \HW)_2^\wedge) \wequi \begin{cases} \W(k)_\I^\comp & *=4i \ge 0 \\ 0 & \text{else} \end{cases}. \]
We thus need to show that the image of $e_{2i}$ in $\pi_{4i} (\kw \wedge \HW) \otimes \Q \wequi \W(k) \otimes \Q$ and $\pi_{4i}((\kw \wedge \HW)_2^\comp) \wequi \W(k)_\I^\comp$ is a unit.
We first deal with the $2$-complete case, in which it suffices to show that the image of $e_{2i}$ is a unit modulo $\I$ \cite[Tag 05GI]{stacks-project}.
Since $\W(k)/\I = \F_2$ is independent of $k$, we may reduce to the case where $k$ is quadratically closed and hence $\W(k) = \F_2$; it hence suffices to show that the image $y_i$ of $e_{2i}$ is non-zero.
But $\adamsphi^{\circ i}(y_i)$ is the image of $\adamsphi^{\circ i}(e_{2i}) = 1$ (by Lemma \ref{lemm:phi-MSL}).
Since $1 \ne 0 \in \pi_0(\kw \wedge \HW)_2^\comp$ we have $\adamsphi^{\circ i}(y_i) \ne 0$ and so $y_i \ne 0$.

In the rational case we are dealing with $\rho$-periodic spectra, so we may reduce to the case where $k$ is real closed and hence $\W(k) = \Z$ (see \S\ref{subsec:rho}); it is thus enough to show that the image of $e_{2i}$ is non-zero in $\W(k) \otimes \Q \wequi \Q$.
Consider the commutative diagram (``fracture square'')
\begin{equation*}
\begin{CD}
(\kw \wedge \HW)_{(2)} @>>>(\kw \wedge \HW) \otimes \Q \\
@VVV                        @VVV \\
(\kw \wedge \HW)_2^\comp @>>> ((\kw \wedge \HW)_2^\comp) \otimes \Q. \\
\end{CD}
\end{equation*}
On applying $\pi_{4i}$, we obtain a commutative diagram
\begin{equation*}
\begin{CD}
\pi_{4i}(\kw \wedge \HW)_{(2)} @>>> \pi_{4i}(\kw \wedge \HW) \otimes \Q \\
@VVV                        @VVV \\
\Z_2^\comp             @>{\iota}>> \Z_2^\comp[1/2].
\end{CD}
\end{equation*}
Since $\Z_2^\comp$ is torsion-free the map $\iota$ is injective.
$e_{2i}$ maps to a generator in the lower left hand corner, hence has non-zero image in the lower right hand corner.
It follows that it must also have non-zero image in the upper right hand corner.
This completes the proof for $\kw \wedge \HW$.

To treat $\kw \wedge \kw$, we note that we have built a map $\gamma: \bigvee_{n \ge 0} \Sigma^{4n} \kw_{(2)} \to \kw \wedge \kw_{(2)}$ of connective objects in $\kw_{(2)}\Mod$.
The extension of scalars functor $\kw_{(2)}\Mod \to \HW_{(2)}\Mod$ is conservative on bounded below objects (e.g. by \cite[Lemma 29]{bachmann-tambara} and \cite[Corollary 4]{bachmann-hurewicz}).
But \[ \gamma \otimes_{\kw_{(2)}} \HW_{(2)}: \bigvee_{n \ge 0} \Sigma^{4n} \HW_{(2)} \to \kw \wedge \HW_{(2)} \] is just the equivalence constructed above.
The result follows.

(2) We have $x = a x_i$ for some $a \in \W(k)_{(2)}$.
We need to show that $a$ is a unit if and only if $\adamsphi^{\circ i}(x)$ is a generator.
Since $\W(k)_{(2)}$ is a local ring (Lemma \ref{lemm:witt-units}), we may check this modulo $\I$, and hence base change to a quadratic closure.
We may thus assume that $\W(k) = \F_2$.
Now by construction (i.e. Lemma \ref{lemm:phi-MSL}(2)) we have $\adamsphi^{\circ i}(x_i) = 1$, and hence $\adamsphi^{\circ i}(x) = a$ generates $\pi_0$ if and only if $a$ is a unit, if and only if $x$ generates $\pi_{4i}$.
\end{proof}

\begin{theorem} \label{thm:main}
Let $k$ be any field of characteristic $\ne 2$.
Denote by $\adamsphi$ the map obtained via Corollary \ref{cor:construct-phi}(1) from the map $\adamspsi = \adamspsi^3$ constructed in Remark \ref{rmk:adams-kw}.
Then the canonical map $\1[\eta^{-1}]_{(2)} \to \fib(\adamsphi)$ obtained via Corollary \ref{cor:construct-phi}(2) is an equivalence.
In other words, there is a fiber sequence \[ \1[\eta^{-1}]_{(2)} \to \kw_{(2)} \xrightarrow{\adamsphi} \Sigma^4 \kw_{(2)}. \]
\end{theorem}
\begin{proof}
Write $F = \fib(\adamsphi)$.
The ``extension of scalars'' functor $\SH(k)[\eta^{-1}]_{(2)} \to \HW_{(2)}\Mod$ is conservative on bounded below objects (e.g. by \cite[Lemma 29]{bachmann-tambara} and \cite[Corollary 4]{bachmann-hurewicz}).
It consequently suffices to show that $\HW_{(2)} \to F \wedge \HW$ is an equivalence.
Since the field is arbitrary, it is enough to show that we have an isomorphism on $\pi_*$.
Considering the long exact sequence for $\pi_* F$ and Proposition \ref{prop:kw-HW-better}, it suffices to show that $\adamsphi: \pi_{4i} (\kw \wedge \HW_{(2)}) \to \pi_{4i-4} (\kw \wedge \HW_{(2)})$ is an isomorphism for $i>0$.
In other words we need to show that $\adamsphi(x_i)$ generates $\pi_{4i-4}$ as a $\W(k)_{(2)}$-module.
By Proposition \ref{prop:kw-HW-better}(2) this holds if and only if $\adamsphi^{\circ (i-1)}(\adamsphi(x_i))$ generates $\pi_0$.
Since $x_i$ generates $\pi_{4i}$, this is indeed the case, again by Proposition \ref{prop:kw-HW-better}(2).
\end{proof}

\todo{remark about slices approach?}

\section{Applications} \label{sec:applications}
\subsection{Homotopy groups of the \texorpdfstring{$\eta$}{eta}-periodic sphere} \label{subsec:stable-stems}
Determination of the groups $\ul{\pi}_*(\1[\eta^{-1}])(k)$ (or some completed or localized variants), for various fields $k$, has been pursued by various authors over the years.
The case $k=\C$ was first to be approached, by Guillou--Isaksen \cite{guillou2015eta}.
They resolved it up to a conjecture about the classical Adams--Novikov spectral sequence, which was subsequently proved by Andrews--Miller \cite{andrews2017inverting}.
The cases $k=\R$ and $k=\Q$ (both up to $2$-adic completion) were done by Guillou--Isaksen \cite{guillou2016eta} and Wilson \cite{wilson2018eta}, respectively.
All of these authors use the motivic Adams spectral sequence.
In contrast, Ormsby--Röndigs \cite{ormsby2019homotopy} use the slice spectral sequence; this allows them to treat all fields of (characteristic $\ne 2$ and) finite cohomological dimension in which $-1$ is a sum of four squares (e.g. all finitely generated fields of odd characteristic).

Our results allow us to determine these groups for all fields of characteristic $\ne 2$.
\begin{theorem} \label{thm:stable-stems}
Let $k$ be a field, $\chara(k) \ne 2$.
We have \[ \ul{\pi}_*(\1_k[\eta^{-1}]) \wequi \begin{cases} \ul{W} & *=0 \\ \ul{W}[1/2] \otimes \pi_*^s \oplus \mathrm{coker}(8n: \ul{W}_{(2)} \to \ul{W}_{(2)}) & *=4n-1>0 \\ \ul{W}[1/2] \otimes \pi_*^s \oplus \ker(8n: \ul{W}_{(2)} \to \ul{W}_{(2)}) & *=4n>0 \\ \ul{W}[1/2] \otimes \pi_*^s & \text{else} \end{cases}. \]
Here $\pi_*^s$ denotes the classical stable stems.
\end{theorem}
\begin{proof}
The cases $* \le 0$ are clear.

We have $\pi_* \1[\eta^{-1}] \otimes \Q \wequi \W(k) \otimes \Q$ (see \eqref{eq:serre-finiteness}), and hence $\pi_* \1[\eta^{-1}]$ is torsion for $*>0$.
Thus\NB{Using that for any torsion abelian group $A$ we have $A \wequi A[1/p] \oplus A_{(p)}$.} (for $*>0$) \[ \ul{\pi}_*\1[\eta^{-1}] \wequi \ul{\pi}_* \1[1/\eta, 1/2] \oplus \ul{\pi}_* \1[\eta^{-1}]_{(2)}. \]
We first show that $\ul{\pi}_* \1[1/\eta,1/2] \wequi \ul{W}[1/2] \otimes \pi_*^s$.
By Corollary \ref{cor:rho-periodic-stems} and \S\ref{subsub:plus-minus} we have \[ \ul{\pi}_*(\1[1/\eta,1/2])(K) \wequi C(\Sper(K), \Z) \otimes \pi_*^s \otimes \Z[1/2]. \]
The case $*=0$ yields \[ C(\Sper(K), \Z) \otimes \Z[1/2] \wequi \ul{\pi}_0(\1[1/\eta,1/2])(K) \wequi \W(K)[1/2]. \]
The claim follows.

It remains to determine $\ul{\pi}_*(\1[\eta^{-1}])_{(2)}$.
This is an immediate application of Theorem \ref{thm:main}, by taking the associated long exact sequence of homotopy sheaves of the fibration sequence $\1[\eta^{-1}]_{(2)} \to \kw_{(2)} \xrightarrow{\adamsphi} \Sigma^4 \kw_{(2)}$.
Since $\ul{\pi}_* \kw \wequi \ul{W}[\beta]$ (see \S\ref{subsec:KW}), the claim follows, as long as $* \not\in \{4n, 4n-1 \mid n > 0\}$, whereas in these cases we get the kernel and cokernel of \[ \adamsphi: \ul{W}_{(2)} \wequi \ul{\pi}_{4*} \kw_{(2)} \to \ul{\pi}_{4*-4} \kw_{(2)} \wequi \ul{W}_{(2)}. \]
By Example \ref{ex:phi-beta}, this is multiplication by $9^n-1$.
Since we are working $2$-locally, by Lemma \ref{lemm:im-j-congruence} below, up to a unit this is the same as $8n$.
\end{proof}
\begin{lemma} \label{lemm:im-j-congruence} \NB{reference?}
For $n \ge 0$ we have 
\[ \nu_2(9^n - 1) = \nu_2(8n). \]
\end{lemma}
\begin{proof}
Writing \[ 9^n - 1 = (1+8)^n -1 =  8n + \sum_{p \ge 2} {n \choose p} 8^p, \] it suffices to show that $\nu_2({n \choose p} 8^p) > \nu_2(8n)$.
Since ${n \choose p} = \frac{n(n-1) \dots (n-p+1)}{p!}$, it is enough to show that $(*)$ $\nu_2(p!) < \nu_2(8^{p-1}) = 3(p-1)$.
Legendre's formula (see e.g. \cite[Theorem 2.6.4]{moll2012numbers}) implies that $\nu_2(p!) \le p$, whence $(*)$ holds for $p \ge 2$ as needed.
\end{proof}

\subsection{Homotopy groups of \texorpdfstring{$\MSp[\eta^{-1}]$}{MSp}} \label{subsec:htpy-MSp}\discuss{Mike: I think this might interest you}
Next we compute the homotopy sheaves of $\MSp[\eta^{-1}]$.
This will require some slightly involved algebraic manipulations.

\begin{lemma} \label{lemm:almost-derivation}
Suppose that $E \in \SH(k)_{(2)}$ such that $\kw_* E$ is $\beta$-torsion free.
Then for $a, b \in \kw_*E$ we have \[ \adamspsi^3(a)\adamsphi(b) + \adamsphi(a)b = \adamsphi(ab). \]
\end{lemma}
\begin{proof}
Via \eqref{eq:adamsphi-compute} this is a direct computation: \begin{gather*} \adamspsi^3(a)\adamsphi(b) + \adamsphi(a)b = \beta^{-1}[\adamspsi^3(a)(\adamspsi^3(b)-b) + (\adamspsi^3(a)-a)b] \\ = \beta^{-1}[\adamspsi^3(a)\adamspsi^3(b) - ab] = \beta^{-1}[\adamspsi^3(ab) - ab] = \adamsphi(ab). \end{gather*}
\end{proof}
\begin{lemma} \label{lemm:MSp-htpy-quad-closed}
Suppose that $\W(k) = \F_2$.
Then the operation $\adamsphi: \kw_*\MSp \to \kw_{*-4} \MSp$ is surjective, with kernel \[ \F_2[y_1, y_2, \dots] \subset \kw_*\MSp \wequi \F_2[\beta, e_1, e_2, \dots] \] where $\deg y_i = 2i$ and $y_2 = \beta$.
\end{lemma}
\begin{proof}
Let $R=\F_2[\beta]$, which we view as a filtered ring via the filtration by powers of $\beta$.
We also view $\kw_*\MSp$ as filtered by powers of $\beta$.
Note that \[ \gr^\bullet R \wequi \F_2[\beta'] \quad\text{and}\quad \gr^\bullet \kw_*\MSp \wequi \F_2[\beta', e_1', e_2', \dots]; \] here $\beta' \in \gr^1$ and $e_i' \in \gr^0$ are the images of $\beta \in F^1$ and $e_i \in F^0$.
By \eqref{eq:beta-phi-commute} we have $\adamsphi(\beta x) = \beta \adamsphi(x)$; in other words $\adamsphi$ is a filtered morphism.
Note also that all our filtrations are degreewise finite, and hence complete, Hausdorff and exhaustive.\NB{details?}
By Lemma \ref{lemm:filtered-surjection} and Corollary \ref{corr:filtered-free-lifting}(2), it thus suffices to show that $\gr^\bullet(\adamsphi)$ is surjective with kernel $\F_2[y_1', y_2', \dots]$, such that $\beta$ lifts $y_2'$.
Since $\adamsphi(\beta) = 0$ this makes sense, and since $\beta$ lifts $\beta'$ it suffices to show the claim with $y_2' = \beta'$.
By Proposition \ref{prop:adams-cobordism-summary} and \eqref{eq:adamsphi-compute} we have \[ \adamsphi(e_i) \in \beta\kw_* \MSp + \begin{cases} e_{i-2} & i \text{ even} \\ 0 & i \text{ odd} \end{cases}. \]
This implies that \begin{equation}\label{eq:bibi-2} \adamsphi(e_i') = \begin{cases} e_{i-2}' & i \text{ even} \\ 0 & i \text{ odd} \end{cases}. \end{equation}
Also by Proposition \ref{prop:adams-cobordism-summary} we have \[ \adamspsi^3(e_i) \equiv e_i \pmod{\beta} \] and hence, since $\adamspsi^3$ is a ring map and $\kw_*\MSp$ is generated by the $e_i$ (over $\kw_*$), we get \[ \adamspsi^3(x) \equiv x \pmod{\beta}\text{ for all } x \in \kw_*\MSp. \]
Via Lemma \ref{lemm:almost-derivation} this implies that \[ \adamsphi(ab) \equiv a\adamsphi(b) + \adamsphi(a)b \pmod{\beta} \] for all $a,b \in \kw_*\MSp$ and hence \begin{equation}\label{eq:derivation} \adamsphi(ab) = a\adamsphi(b) + \adamsphi(a)b \text{ for all } a,b \in \gr^\bullet \kw_* \MSp. \end{equation}
We also have $\adamsphi(\beta) = 8\beta = 0$ and hence \begin{equation}\label{eq:beta'} \adamsphi(\beta') = 0 \end{equation}

We have the decomposition \[ \gr^\bullet \kw_*\MSp \wequi \F_2[\beta', e_1', e_2', \dots] \wequi \F_2[e_2', e_4', \dots] \otimes_{\F_2} \F_2[\beta', e_1', e_3', \dots] =: A \otimes_{\F_2} B. \]
Using that $\adamsphi$ is a derivation (i.e. \eqref{eq:derivation}) and the action on the generators (i.e. \eqref{eq:bibi-2}, \eqref{eq:beta'}) we see that $\adamsphi = \adamsphi|_A \otimes_{\F_2} \id_B$.
Since $\otimes_{\F_2} B$ is an exact functor, it is thus sufficient to prove that $\adamsphi|_A$ is surjective with kernel $\F_2[y_4, y_6, \dots]$.
This is established in Lemma \ref{lemm:comuputation} below (put $x_i = e_{2i}'$, $A_0=\F_2$).
\end{proof}

\begin{lemma} \label{lemm:comuputation}
Let $A_0$ be a (commutative) ring and consider the commutative graded ring \[ A = A_0[x_1, x_2, \dots], \] where $|x_i| = i$.
Give it the derivation $\adamsphi$ with $\adamsphi(x_i) = x_{i-1}$ (and $\adamsphi(x_1) = 1$, $\adamsphi(A_0) = 0$).\NB{In other words $\adamsphi = \sum_{i \ge 0} x_i \partial/\partial x_{i+1}$.}
Then $\adamsphi$ is surjective with kernel a polynomial ring \[ A_0[y_2, y_3, \dots] \] where $|y_i| = i$.
\end{lemma}
\begin{proof}
If $\adamsphi$ is surjective (for some $A_0$), then the exact sequence of $A_0$-modules \[ 0 \to \ker(\adamsphi) \to A \xrightarrow{\adamsphi} A \to 0 \] is stable under base change, $A$ being free as an $A_0$-module.
We may thus assume that $A_0=\Z$.
Recall that \[ H^*(\BU) = \Z[c_1, c_2, \dots] \] and that the canonical map $\BSU \to \BU$ induces $H^*(\BSU) \wequi H^*(\BU)/c_1$.
In other words we have a short exact sequence \[ 0 \to H^*(\BU) \xrightarrow{c_1} H^*(\BU) \to H^*(\BSU) \to 0. \]
Passing to duals, we get \[ 0 \to H_*(\BSU) \to H_*(\BU) \xrightarrow{f} H_*(\BU) \to 0, \] where $f$ is dual to multiplication by $c_1$.
The homology of $\BU$ a polynomial ring, free on the generators of $H_*(\BU(1))$ (compare the proof of \ref{thm:cobordism-MSL-summary}).
Together with the description of the Hopf algebra structure on $H_*(\BU)$ (see e.g. \cite[\S24.1]{may2011more}), we find that $H_*(\BU) \wequi A$ (with degrees doubled) and that under this isomorphism $f$ corresponds to $\adamsphi$.
The result follows since $H_*(\BSU)$ is known to be a polynomial ring on generators in degrees $4, 6, 8, \dots$ \cite[Lemma 2.4]{adams1976primitive}.
\end{proof}

\begin{corollary} \label{cor:MSp-intermediate}
Let $k$ be any field of characteristic $\ne 2$.
Then \[ \ul{\pi}_* \MSp_{(2)}[\eta^{-1}] \wequi \ul{W}_{(2)}[y_1, y_2, \dots]. \]
\end{corollary}
\begin{proof}
To ease notation, we implicitly invert $\eta$ throughout this proof.

We first show the claim about $\pi_*$ instead of $\ul{\pi}_*$.
Note that $\W(k)_{(2)}$ is a local ring (Lemma \ref{lemm:witt-units}).
We need to show that \[ \adamsphi: \kw_* \MSp_{(2)} \to \kw_{*-4} \MSp_{(2)} \] is surjective with kernel as indicated.
The map $\adamsphi$ is a map of $\W(k)_{(2)}$-modules which are degreewise finitely generated free.
Note that $\kw_* \MSp_{(2)}/\I$ is independent of $k$ and hence the same holds for $\adamsphi/\I$.
Thus by Lemma \ref{lemm:MSp-htpy-quad-closed} the claim holds modulo $\I$.
It follows (using Nakayama's lemma) that $\adamsphi$ is split surjective: we may choose $C \subset \kw_* \MSp_{(2)}$ such that $\kw_* \MSp_{(2)} = \ker(\adamsphi) \oplus C$ and $\adamsphi: C \to \kw_{*-4} \MSp_{(2)}$ is an isomorphism.
Thus $\ker(\adamsphi/\I) \wequi \ker(\adamsphi)/\I$.
This implies in particular that any element in $\ker(\adamsphi/\I)$ lifts to $\ker(\adamsphi)$, and that any family of elements of $\ker(\adamsphi)$ which form a basis of $\ker(\adamsphi/\I)$ form a basis of $\ker(\adamsphi)$ (again using Nakayama's lemma).
Lifting the polynomial generators $\bar y_i \in \ker(\adamsphi/\I)$ arbitrarily to $y_i \in \ker(\adamsphi)$ we deduce that monomials in the $y_i$ form a basis of $\ker(\adamsphi)$; the result about $\pi_*$ follows.

Since $\ul{\pi}_* \MSp_{(2)}$ is a $\ul{W}_{(2)}$-algebra, we obtain a map $\ul{W}_{(2)}[y_1, y_2, \dots] \to \ul{\pi}_* \MSp_{(2)}$.
We shall show this is an isomorphism.
To do so, we need to see that the map induces an isomorphism on fields, or equivalently that our generators $y_i$ are stable under base change.
The above proof shows that a family $\{y_i\}$ will generate (over some field $K$) if and only if it generates modulo $\I$.
Since $\W(K)/\I$ is independent of $K$, the result follows.
\end{proof}

\begin{theorem} \label{thm:MSp-htpy}
Let $k$ be any field of characteristic $\ne 2$.
Then \[ \ul{\pi}_* \MSp[\eta^{-1}] \wequi \ul{W}[y_1, y_2, \dots], \] where $|y_i| = 2i$.
\end{theorem}
\begin{proof}
To ease notation, we implicitly invert $\eta$ throughout this proof.

If $\chara(k) > 0$ then $\W(k) = \W(k)_{(2)}$ \cite[Theorem III.3.6]{milnor1973symmetric}, and hence the result follows from Corollary \ref{cor:MSp-intermediate}.
We may thus assume that $\chara(k) = 0$, and by essentially smooth base change \cite[Lemma B.1]{bachmann-norms} that $k=\Q$.
Write $J$ for the augmentation ideal $\ker(\pi_* \MSp \to \pi_* \HW)$.
Let $n > 0$ and put $M = (J/J^2)_{2n}$.
We claim that $M \wequi W(k)$.
Assuming this for now, let $z_n$ be a generator of $M$, and $y_n \in \pi_{2n} \MSp$ a lift of $z_n$.
We obtain a map $\alpha: \ul{W}[y_1, y_2, \dots] \to \ul{\pi}_*\MSp$, which we shall show is an isomorphism.
It suffices to show that $\alpha_{(2)}$ and $\alpha[1/2]$ are isomorphisms (see e.g. Lemma \ref{lemm:inv-compl}).
It follows from Corollary \ref{cor:MSp-intermediate} that \[ M_{(2)} \wequi \W(k)_{(2)}\{y_n'\}, \] where the $y_n'$ form a family of polynomial generators of $\ul{\pi}_*\MSp_{(2)}$.
Thus $\alpha_{(2)}$ is an isomorphism (Lemma \ref{lemm:indecomposables}(2)).
To show that $\alpha[1/2]$ is an isomorphism, since we are dealing with $\rho$-periodic spectra, it suffices to show that there is an isomorphism on global sections (see \S\ref{subsec:r-R}; this is where we use $k=\Q$).
Moreover, since $r_\R(\MSp) \wequi \MU$ (see Lemma \ref{lemm:cobordism-realisation}), we know that \cite[Theorems 4.1.6 and A2.1.10]{ravenel1986complex} \[ \pi_*\MSp[1/2] \wequi \pi_* \MU[1/2] \wequi \Z[1/2, t_1, t_2, \dots] \] with $|t_i| = 2i$.
This implies that \[ M[1/2] \wequi \Z[1/2]\{t_n\} \] and so $\alpha[1/2]$ is an isomorphism (Lemma \ref{lemm:indecomposables}(2) again).

It remains to prove the claim that $M \wequi \W(\Q)$.
As we have seen, $M_{(2)} \wequi \W(\Q)_{(2)}$ and $M[1/2] \wequi \Z[1/2] \wequi \W(\Q)[1/2]$ (see e.g. \cite[Theorem III.3.10]{milnor1973symmetric} for the latter isomorphism).
We first show that $M$ is finitely generated as a $\W(\Q)$-module.
Indeed let $\frac{x}{2^m} \in M[1/2]$ and $\frac{y}{a} \in M_{(2)}$ (with $a \in \Z$ odd and $x, y \in M$) generate $M[1/2]$ and $M_{(2)}$ respectively; then $x$ and $y$ generate $M$.\NB{must be standard??}
By \cite[\S II.5.2, Theorem 1]{bourbakiCA}, a finitely generated module is invertible (by which we mean locally free of rank $1$) if and only if it is stalkwise invertible.
These properties hold for $M$, so it is an invertible $\W(\Q)$-module.
The result will follow if we show that $\Pic(\W(\Q))$ is trivial.
It follows from idempotent lifting \cite[Exercise I.2.2]{weibel-k-book}\NB{better reference} that for any ring $R$ we have $\Pic(R) \wequi \Pic(R_\red)$.
It thus suffices to show that $\W(\Q)_\red \wequi \Z$.
Since $\Q$ is uniquely orderable, this follows from \cite[Theorem III.3.8]{milnor1973symmetric}.
\end{proof}

\todo{say something about hurewicz image?}

\subsection{Homotopy groups of \texorpdfstring{$\MSL[\eta^{-1}]$}{MSL}}
We can easily adapt the above arguments to $\MSL$ as well.
\begin{theorem} \label{thm:homotopy-msl}
The canonical map $\MSp \to \MSL$ induces \[ \ul{\pi}_* \MSL[\eta^{-1}] \wequi \ul{W}[y_2, y_4, \dots]. \]
\end{theorem}
\begin{proof}
We have a map $\alpha: \ul{W}[y_2, y_4, \dots] \to \ul{\pi}_* \MSL[\eta^{-1}]$ which we need to show is an equivalence; it suffices to do this for $\alpha[1/2]$ and $\alpha_{(2)}$.
For $\alpha[1/2]$ the claim reduces to the analogous result in topology (using Lemma \ref{lemm:cobordism-realisation} and \S\ref{subsub:plus-minus}): the map $\MU[1/2] \to \MSO[1/2]$ induces $\pi_*\MSO[1/2] \wequi \Z[1/2, t_2, t_4, \dots]$ \cite[\S IX, Proposition p. 178, Theorem p. 180]{stong2015notes}.
For $\alpha_{(2)}$ we use the resolution \[ \eta^{-1}\MSL_{(2)} \to \kw \wedge \MSL_{(2)} \xrightarrow{\adamsphi} \Sigma^4 \kw \wedge \MSL_{(2)}. \]
It suffices to show that $\pi_* \alpha_{(2)}$ is an isomorphism (since the base field is arbitrary and the formation of $\alpha$ is compatible with base change).
We hence need to show that $\adamsphi$ is surjective with kernel as indicated.
Using that $\W(k)_{(2)}$ is a local ring (Lemma \ref{lemm:witt-units}), this reduces to checking modulo $\I$, i.e. we may base change to a field with $\W(k) = \F_2$.
Examining the proof of Lemma \ref{lemm:MSp-htpy-quad-closed}, this is easily seen to hold.\NB{Better separation?}
\end{proof}

\subsection{$\MSp$, $\MSL$ and $\kw$}
The map $\MSp[\eta^{-1}] \to \MSL[\eta^{-1}]$ is an $\scr E_\infty$-ring map (see \eqref{eq:thom-spectra}) which annihilates $y_i$ for $i$ odd, for degree reasons.
It hence induces for $i$ odd an $\MSp[\eta^{-1}]$-module map $\MSp[\eta^{-1}]/y_i \to \MSL[\eta^{-1}]$.
Put \[ \MSp[\eta^{-1}]/(y_1, y_3, \dots, y_{2n+1}) = \bigotimes_{i=0}^n \MSp[\eta^{-1}]/y_{2i+1} \in \MSp[\eta^{-1}]\Mod. \] 
The canonical map $\MSp[\eta^{-1}] \to \MSp[\eta^{-1}]/y_{2n+3}$ induces \[ \MSp[\eta^{-1}]/(y_1, y_3, \dots, y_{2n+1}) \to \MSp[\eta^{-1}]/(y_1, y_3, \dots, y_{2n+3}), \] and we put \[ \MSp[\eta^{-1}]/(y_1, y_3, \dots) = \colim_n \MSp[\eta^{-1}]/(y_1, \dots, y_{2n+1}). \]
\begin{corollary} \label{cor:MSp-MSL}
There is a canonical equivalence \[ \MSp[\eta^{-1}]/(y_1, y_3, \dots) \wequi \MSL[\eta^{-1}]. \]
\end{corollary}
\begin{proof}
We implicitly invert $\eta$ throughout this proof.
Since $y_i$ maps to $0$ in $\MSL$ (for $i$ odd), the map $\MSp \to \MSL$ factors over $\MSp/y_i$.
We can thus form the composite \[ \MSp/(y_1, \dots, y_{2n+1}) \wequi \bigotimes_{i=0}^n \MSp/y_{2i+1} \to \MSL^{\otimes n} \to \MSL, \] with the last map being multiplication.
Taking the colimit we obtain $\MSp/(y_1, y_3, \dots) \to \MSL$.
To see that this is an equivalence we can compute the effect on homotopy sheaves.
Since $(y_1, y_3, \dots)$ is a regular sequence in $\ul{\pi}_*\MSp \wequi \ul{W}[y_1, y_2, \dots]$ we find \[ \ul{\pi}_*(\MSp/(y_1, y_3, \dots)) \wequi \ul{W}[y_2, y_4, \dots]; \] the result thus follows from Theorem \ref{thm:homotopy-msl}.
\end{proof}

There is an $\scr E_\infty$-map $\MSL \to \kw$ \cite[Corollary B.3]{bachmann-euler}.
\begin{lemma} \label{lemm:MSL-kw-surj-pi4}
The induced map $\pi_4(\MSL[\eta^{-1}]) \to \pi_4(\kw)$ is surjective.
\end{lemma}
\begin{proof}
It suffices to show surjectivity after tensoring with $\Z[1/2]$ and $\Z_{(2)}$.

\discuss{suggestions?}
For $\Z[1/2]$ this reduces (via Lemmas \ref{lemm:KO-real-realization} and \ref{lemm:cobordism-realisation}) to the topological result that the map $\MSO[1/2] \xrightarrow{\alpha^\topsup} \ko[1/2]$ (obtained by applying $r_\R$ to $\MSL[\eta^{-1}] \to \kw$) is surjective on $\pi_4$.
%Since $\pi_4(\ko[1/2]) \wequi \pi_4(\ku[1/2]) \wequi \pi_4(\KU[1/2])$, this follows from the fact that $\pi_4 \MU \to \pi_4 \KU$ is surjective; for this well-known result see e.g. \cite{conner-floyd}.
We shall make use of some of the theory of complex orientations and formal group laws, see e.g. \cite[\S\S4.1, A.2]{ravenel1986complex}.
Consider the formal group law $F$ induced by $\MU[1/2] \to \MSO[1/2] \xrightarrow{\alpha^\topsup} \ko[1/2] \to \ku[1/2]$.
Since the ring structure on $\ku[1/2]$ is the usual one, the associated formal group is the multiplicative one, and hence the formal group law must be \emph{isomorphic} to the multiplicative one\NB{Clearly it can't be the multiplicative one, since $t \not\in \pi_*\ko[1/2]$. Which one is it?}.
Thus there exists a formal power series $f(x) =b_0x + b_1 x^2 + \dots$ such that $F(x,y) = f^{-1}(f(x) + f(y)+tf(x)f(y))$, where $t \in \pi_2\ku[1/2]$ is the generator.
It suffices to show that $t^2$ can be obtained as a $\Z[1/2]$-linear combination of coefficients of the formal group law $F$.
Applying the isomorphism $x \mapsto ux$ (where $u$ is a unit in $\pi_* \ku[1/2]$, i.e. $u = \pm2^n$) does not change this property, so we may assume that $b_0=1$.
Noting that the coefficients of $F$ lie in $\pi_* \ko[1/2] = \Z[1/2,t^2]$, one finds that $b_1 = t/2$.
A tedious but straightforward verification shows that the coefficient of $xy^2$ is $-t^2$; hence the desired result.\NB{better argument?}

For $\Z_{(2)}$, since $\W(k)_{(2)}$ is a local ring we may base change to a field with $\W(k) = \F_2$.
Consider the commutative diagram
\begin{equation*}
\begin{CD}
\pi_4 \MSL_{(2)}[\eta^{-1}] @>>> \pi_4(\MSL \wedge \kw)_{(2)} \\
@VVV                               @VVV \\
\F_2 \wequi \pi_4 \kw_{(2)} @>>> \pi_4(\kw \wedge \kw)_{(2)}.
\end{CD}
\end{equation*}
We wish to show that the left hand vertical map is non-zero.
By Lemma \ref{lemm:MSp-htpy-quad-closed}, the image of the top map is spanned by $\beta$, which is mapped to $\beta_R$ (the image of $\beta$ under the right unit $u_R: \kw \to \kw \wedge \kw$) in the bottom right hand corner.
Since $\kw \wedge \kw$ is a ring, the right unit $\kw \to \kw \wedge \kw$ has a retraction and so induces an injection on homotopy groups; in particular $\beta_R \ne 0$.
The result follows.
\end{proof}

\begin{corollary} \label{cor:MSL-kw}
There exist generators $y_2, y_4, y_6, \dots \in \pi_* \MSL[\eta^{-1}]$ such that \[ \MSL[\eta^{-1}]/(y_4, y_6, \dots) \wequi \kw. \]
\end{corollary}
\begin{proof}
Lemma \ref{lemm:MSL-kw-surj-pi4} shows that we may choose $y_2$ such that $\alpha(y_2) = \beta$, where $\alpha: \MSL \to \kw$ is the canonical $\scr E_\infty$-map.
Now let $n>1$.
We have $\alpha(y_{2n}) = a \beta^n$ for some $a \in \W(k)$; replacing $y_{2n}$ by $y_{2n} - ay_2^n$ ensures that $\alpha(y_{2n}) = 0$.
Arguing as in the proof of Corollary \ref{cor:MSp-MSL} can thus form a map \[ \MSL[\eta^{-1}]/(y_4, y_6, \dots) \to \kw \in \MSL[\eta^{-1}]\Mod \] which induces an isomorphism on $\ul{\pi}_*$.
This concludes the proof.
\end{proof}

\subsection{Cellularity results}
It is well-known that the spectra $\KO$ and $\KW$ are cellular (i.e. in the subcategory of $\SH(k)$ generated under colimits and desuspensions by $S^{p,q}$) \cite{rondigs2016cellularity}.
Unfortunately in general if $E$ is cellular there is little reason to believe that truncations like $E_{\ge 0}$ or $\ul{\pi}_0 E$ are cellular.
Our main theorem allows us to make some deductions of this form.

\begin{proposition} \label{prop:cellularity}
Let $k$ have exponential characteristic $e \ne 2$.
The spectra \[ \kw, \HW, \ko[1/e], \kgl[1/e], H\tilde \Z[1/e], f_0(\ul{K}^W) \in \SH(k) \] are cellular.
\end{proposition}
\begin{proof}
By Lemma \ref{lemm:cellular-pp} below, to prove that $E$ is cellular, it suffices to show that $E[1/\eta]$ and $E/\eta$ are cellular.
Since $\ko/\eta \wequi \kgl$ \cite[Proposition 2.11]{ananyevskiy2017very} and $\ko[1/\eta] \wequi \kw$, we may remove $\ko$ from the list.
The argument in \cite[Proposition 5.12]{spitzweck2012motivic}\NB{markus agrees} (employing \cite[Proposition B.1]{levine2013algebraic}) shows that $\kgl[1/e] \wequi \MGL[1/e]/(x_2, x_3, \dots)$ is cellular (since $\MGL$ is).
Put $H_W \Z := f_0(\ul{K}^W)$.
By \cite[Proposition 23]{bachmann-very-effective} we have $H_W \Z/\eta \wequi \HZ/2 \vee \Sigma^2 \HZ/2$, which is cellular, and by \cite[Theorem 17]{bachmann-very-effective} we have $H_W \Z[1/\eta] \wequi \HW$; hence we may remove $H_W \Z$ from the list.
Again by \cite[Proposition 23]{bachmann-very-effective} we have a cofiber sequence $\Sigma^{1,1} H_W \Z \to H\tilde \Z \to \HZ \vee \Sigma^2 \HZ/2$.
Since $\HZ[1/e]$ is cellular, we may also remove $H\tilde \Z$ from the list.

It remains to deal with $\kw$ and $\HW$.
Since $\MSp$ is cellular (see \cite[Proposition 3.1]{rondigs2016cellularity}), cellularity of $\kw$ is immediate from Corollaries \ref{cor:MSp-MSL} and \ref{cor:MSL-kw}.
Finally $\HW$ is cellular since $\HW \wequi \kw/\beta$.
\end{proof}

\begin{lemma} \label{lemm:cellular-pp}
Let $E \in \SH(S)$ and $x \in \pi_{**}(\1)$.
Then $E$ is cellular if and only if both $E/x$ and $E[1/x]$ are.
\end{lemma}
\begin{proof}
Since cellular spectra are closed under colimits, necessity is clear.
We show sufficiency.
Let $\scr C \subset \SH(S)$ denote the subcategory of cellular spectra.
Since $\scr C$ is closed under colimits, the inclusion has a right adjoint $r$, since $\scr C$ is generated by a set of compact objects from $\SH(S)$, $r$ preserves colimits, and since $\scr C$ is stable under desuspension the functor $r$ is stable.
Moreover we have $r(\Sigma^{p,q}E) \wequi \Sigma^{p,q}r(E)$, and $r(E \xrightarrow{x} \Sigma^{**}E) \wequi (r(E) \xrightarrow{x} \Sigma^{**}r(E))$.
It follows $(*)$ that $r$ commutes with formation of $(\ph)/x$ and $(\ph)[1/x]$.

We seek to show that $r(E) \to E$ is an equivalence.
By Lemma \ref{lemm:inv-compl} it suffices to show that $r(E)/x \to E/x$ and $r(E)[1/x] \to E[1/x]$ are equivalences.
Since $r(E)/x \wequi r(E/x)$ and $r(E)[1/x] \wequi r(E[1/x])$ by $(*)$, the result follows (from the assumption that $E/x$ and $E[1/x]$ are cellular).
\end{proof}

\subsection{$\kw^*\kw_{(2)}$}
\begin{lemma} \label{lemm:phi-gen}
Write $\adamsphi^n$ for the $n$-fold iteration of $\adamsphi$.
The map \[ \kw \wedge \kw_{(2)} \xrightarrow{\id \wedge \adamsphi^\bullet} \prod_{n \ge 0} \kw \wedge \Sigma^{4n} \kw_{(2)} \wequi \bigoplus_{n \ge 0} \Sigma^{4n} \kw \wedge \kw_{(2)} \to \bigoplus_{n \ge 0} \Sigma^{4n} \kw_{(2)} \] is an equivalence of left $\kw$-modules (where the last map is multiplication).
\end{lemma}
\begin{proof}
We have $\prod \wequi \bigoplus$ for connectivity reasons.\NB{details?}
Since the base field is arbitrary and the map is canonical, it suffices to show that we have an isomorphism on $\pi_*$.
By Proposition \ref{prop:kw-HW-better}, this is a map of degreewise finite free left $\kw^*_{(2)}$-modules.
We may thus show that there is an isomorphism modulo $\I$, i.e. we may assume that $\W(k) \wequi \F_2$, and it suffices to show that the (left module) generators are preserved.
By Proposition \ref{prop:kw-HW-better}, generators of the source are obtained as the images of the $e_{2i}$ under $\MSL \wedge \kw \to \kw \wedge \kw$.
It hence suffices to show that $\adamsphi^n(e_{2n}) = 1$.
This is Lemma \ref{lemm:phi-MSL}(2).
\end{proof}

\begin{corollary} \label{cor:kw-ops}
We have \[ \kw^*_{(2)}\kw \wequi \kw^*_{(2)}\fpsr{'\adamsphi}, \] in the sense that the underlying $\kw^*_{(2)}$-module is $\kw^*_{(2)}\fpsr{\adamsphi}$ but the composition product is determined by $\adamsphi\beta = 9\beta\adamsphi + 8$.
\end{corollary}
\begin{proof}
By adjunction, we have \[ [\kw, \Sigma^* \kw_{(2)}] \wequi [\kw \wedge \kw_{(2)}, \Sigma^* \kw_{(2)}]_{\kw_{(2)}}, \] where $[\ph,\ph]_{\kw_{(2)}}$ denotes homotopy classes of maps of (strong) $\kw_{(2)}$-modules.
Now $\kw \wedge \kw_{(2)} \wequi \bigoplus_n \Sigma^{4n} \kw_{(2)}$ as $\kw_{(2)}$-modules, so that \[ [\kw \wedge \kw_{(2)}, \Sigma^* \kw_{(2)}]_{\kw_{(2)}} \wequi \prod_n \kw^*_{(2)}\{q_n\}, \] where $q_n: \kw \wedge \kw_{(2)} \to \Sigma^{4n}\kw_{(2)}$ is the projection.
By Lemma \ref{lemm:phi-gen}, $q_n$ is (or rather may be chosen to be) adjoint to $\adamsphi^n$.
The additive structure follows.

Since $\adamsphi$ commutes with multiplication by $\kw^0_{(2)}$, the multiplicative structure is determined once we know $\adamsphi\beta$.
Since $\kw^*\kw_{(2)}$ is $\beta$-torsion free, it suffices to know $\beta\adamsphi\beta$.
We compute\NB{beware this is an equation in endomorphisms of $\kw$, not $\kw_*$, so $\adamspsi\beta = \adamspsi(\beta)\adamspsi$, $\adamspsi$ being a ring map} \[ \beta\adamsphi\beta = (\adamspsi-1)\beta = 9\beta\adamspsi-\beta = 9\beta(\beta\adamsphi+1) - \beta = \beta\cdot(9\beta\adamsphi+8). \]
The result follows.
\end{proof}

\subsection{$\kw_*\kw_{(2)}$ and $\kw_*\HW$}\discuss{Suggestions?}
Let $A = \kw_{(2)}$.
Since $A_*A$ is flat as an $A_*$-module, it forms part of a Hopf algebroid $(A_*, A_*A)$.
We determine some of its structure.
Recall that $A_* = \W(k)_{(2)}[\beta]$.
\begin{proposition} \label{prop:kw-coops}
There exist elements $x_i \in A_{4i}A$ (with $x_0 = 1$) such that $A_*A$ is the free left $A_*$-module on the $x_i$, and such that the following hold.
\begin{enumerate}
\item The right unit $\eta_R: A_* \to A_*A$ is the $\W(k)_{(2)}$-algebra map with $\eta_R(\beta) = \beta + 8x_1$.
\item The comultiplication $\Delta: A_*A \to A_*A \otimes_{A_*} A_*A$ is determined by $\Delta(x_i) = \sum_{i=j+k} x_j \otimes x_k$.
\item The augmentation $\epsilon: A_*A \to A_*$ is determined by $\epsilon(x_i) = 0$ for $i > 0$.
\item The multiplication satisfies \[ x_i x_j = {i + j \choose i} x_{i+j} + O(\beta,8). \]
\end{enumerate}
\end{proposition}
\begin{proof}
We know that $A_*A$ is the topological dual of $A^*A$, which was determined in Corollary \ref{cor:kw-ops}.
Define $x_i: A^*A \to A_*$ via $x_i(\adamsphi^j) = \delta_{ij}$.
\todo{details for the following?}
Then $x_i$ vanishes on sufficiently high powers of $\adamsphi$ and hence is continuous; thus $x_i$ defines an element of $A_*A$.
This is the dual of the topological basis $\{\adamsphi^i\}_{i \ge 0}$ of $A^*A$ and hence forms a left $A_*$-basis of $A_*A$.
Since $\adamsphi^i(1) = \delta_{i0}$ we have $x_0=1$.

(1) The right unit is given by $\eta_r(x) = \sum_i x_i \adamsphi^i(x)$, whence the claim follows from $\adamsphi^0(\beta) = \beta$, $\adamsphi(\beta) = 8$ and $\adamsphi^n(\beta) = 0$ for $n>1$.

(2) This follows by dualization from $\adamsphi^i \circ \adamsphi^j = \adamsphi^{i+j}$.

(3) This is just the dual of the unit map $A^* \to A^*A$. \todo{really?}

(4) For degree reasons we can write $x_i x_j = a_{ij} x_{i+j} + O(\beta)$, with $a_{ij} \in \W(k)_{(2)}$ uniquely determined.
Using that $\Delta$ is a ring map and computing $\Delta(x_i x_j)$ in two ways we find that \[ a_{ij} \sum_{i+j = p+q} x_p \otimes x_q \equiv \sum_{p+q = i, r+s=j} a_{pr}a_{qs} (x_{p+r} \otimes x_{q+s}) \pmod{8}; \] here we use that $\eta_R(\beta) \in O(\beta, 8)$.\todo{?!?}
Comparing the coefficient of $x_i \otimes x_j$ yields \begin{equation}\label{eq:ind-tech} a_{ij} \equiv \sum_p a_{p,i-p}a_{i-p,j+p-i} \pmod{8}. \end{equation}
Since $x_0 = 1$ we have \begin{equation}\label{eq:ind-triv} a_{i0} = 1 = a_{0i} \text{ for all } i \ge 0. \end{equation}
We now prove that $a_{ij} \equiv {i+j \choose i} \pmod{8}$, by induction on $d=i+j$.
The case $d=0$ follows from \eqref{eq:ind-triv}, as do the cases $i=0$ or $j=0$.
Consequently we may assume that $i>0$ and $j>0$.
In this case in \eqref{eq:ind-tech} the right hand side only involves $a_{rs}$ with $r+s < d$, and so (see e.g. \cite[Theorem 5.2.3]{moll2012numbers}) \[ a_{ij} \equiv \sum_p {i \choose p} {j \choose i-p} = {i+j \choose i}. \]
\end{proof}

\begin{remark}
The Hopf algebras $\ku_* \C\P^\infty \otimes_\Z \F_2$ and $\HZ_* \C\P^\infty \otimes_\Z \F_2[\beta]$ are isomorphic to $\kw_*\kw_\C$ as coalgebras, but are not equal (their multiplications encode the additive and multiplicative formal groups, respectively).
It follows that the multiplicative structure on $A_*A$ cannot be determined purely formally.\todo{work it out anyway?}
\end{remark}

\begin{proposition} \label{prop:typical-generators}
There exist generators $x_i \in \pi_{4i} (\kw \wedge \HW_{(2)})$ such that \[ x_m x_n = {m+n \choose n} x_{m+n}. \]
\end{proposition}
In other words the $x_i$ satisfy the identities of a divided power algebra.
\begin{proof}
Since $\HW = \kw/\beta$ we have $\pi_*(\HW \wedge \kw_{(2)}) \wequi A_*A/\beta$.
Thus by Proposition \ref{prop:kw-coops}(4) there exist generators $x_i'$ with $x_i' x_j' = \left({i + j \choose j} + 8b_{ij}\right)x_{i+j}'$.

Put $t_i = x_{2^i}'$.
Noting that $\nu_2({2^{i+1} \choose 2^i}) = 1$ (by Kummer's theorem \cite[Theorem 2.6.7]{moll2012numbers}) we can write \[ t_i^2 = \left({2^{i+1} \choose 2^i} + 8b_{2^i,2^i}\right) t_{i+1} = 2w_i t_{i+1}, \] where $w_i = {2^{i+1} \choose 2^i}/2 + 4b_{2^i,2^i}$ is a unit (since it is so modulo $I$; use Lemma \ref{lemm:witt-units}).
We can write $t_i^2 = b_i t_{i+1}$.
%By Theorem \ref{thm:HW-kw-summary}(2,3) we have $b_i \in \I(k)$ and $b_i \equiv 2u \pmod{\I^2}$ for some $u \in (\W(k)_{\I}^\comp)^\times$.
%Note that $1-u \in I$ and $2(1-u) \in \I^2$, so that $b_i \equiv 2 \pmod{\I^2}$.
%We may thus write $b_i = 2 + c_i$ for some $c_i \in \I^2$.
%If $\vcd_2(k) \le 1$ then $\I^2 = 2\I$ \cite[last Theorem]{elman-lum-2cohom}, whereas if $k = \Q$ we have $\I^3 = 8\Z\lra{1}$ \cite[III (5.9)]{milnor1973symmetric}.
%Hence in either case $fc_i = 2fd_i$ for some $d_i \in \I(k)$.
%It follows that \[ f t_i^2 = f \cdot 2w_i \cdot t_{i+1}, \] where $w_i = 1+d_i \in \W(k)_{(2)}$ is a unit (Lemma \ref{lemm:witt-units}).

Now we inductively define $s_i$ generating $\pi_{4\cdot 2^i}$ such that \begin{equation} \label{eq:divided-power-special} s_i^2 = {2^{i+1} \choose 2^i}s_{i+1}. \end{equation}
Indeed we put $s_1=t_1$ and assuming that $s_n$ has been chosen we get $s_n = a_n t_n$ for some unit $a_n$, hence \[ s_n^2 = (a_nt_n)^2 = a_n^2 \cdot 2w_n t_{n+1} = 2a_n^2w_n t_{n+1}. \]
Noting that ${2^{n+1} \choose 2^n} = 2v_n$ where $v_n$ is odd by Kummer's theorem \cite[Theorem 2.6.7]{moll2012numbers}, we may put $s_{n+1} = (v_n^{-1}a_n^2w_n)t_{n+1}$.

Let $y_i$ generate $\pi_{4i}$.
For $n = \sum_i \epsilon_i 2^i$ put \[ \delta_n = 1/n! \prod_i (2^{\epsilon_i i}!); \] then $\nu_2(\delta_n)=0$ by Legendre's formula \cite[Theorem 2.6.4]{moll2012numbers}.
We can write \[ x_n := \delta_n \prod_i s_i^{\epsilon_i} = e_n y_n. \]
By Theorem \ref{thm:HW-kw-summary}(3) $e_n$ is a unit modulo $\I$, and hence a unit in $\W(k)_{(2)}$ by Lemma \ref{lemm:witt-units}.
Hence the $x_i$ are generators.

We can verify the divided power relations as follows.\NB{ref?}
Write $n= \sum_i \epsilon_i(n) 2^i$.
Then \[ x_n x_m = \frac{(n+m)!}{n!m!} \cdot \frac{\prod_i 2^{\epsilon_i(n) i}! 2^{\epsilon_i(m) i}! s_i^{\epsilon_i(n) + \epsilon_i(m)}}{(n+m)!}; \] these expressions make sense because we are working in a $\Z_{(2)}$-algebra.\footnote{To be precise, the denominator $(n+m)!$ has to be cancelled with the various factors of $2^r!$, which works by Kummer's theorem.}
The product on the right hand side consists of factors of the form $2^r! s_r$, possibly repeated.
We can simplify it by repeatedly applying the relation $(2^r!s_r)^2 = 2^{r+1}!s_{r+1}$ coming from \eqref{eq:divided-power-special}.
In the end we will thus have transformed it into a product of the same form, but with no repeated factors.
The $s_i$ occurring correspond to sum of the binary expansions of $n$ and $m$, i.e. the binary expansion of $n+m$.
Hence the right hand side is \[ {n+m \choose m} x_{n+m}, \] as desired.
\end{proof}

\subsection{$\HW \wedge \HW_{(2)}$}

\begin{lemma} \label{lemm:hurewicz-beta}
Let $h: \kw_{(2)} \to (\kw \wedge \HW)_{(2)}$ be the Hurewicz map.
Then $h(\beta) = 8ux_1$, for some unit $u \in \W(k)_{(2)}$.
\end{lemma}
\begin{proof}
We have $h(\beta) = a x_1$ and $\adamsphi(\beta) = 8$.
By Proposition \ref{prop:kw-HW-better}(2), $\adamsphi(x_1) = v$ for some unit $v$.
Hence \[ 8 = h(\adamsphi(\beta)) = \adamsphi(h(\beta)) = \adamsphi(ax_1) = av. \]
The result follows.
\end{proof}

\begin{corollary} \label{cor:HW-HW}
We have \[ \HW \wedge \HW_{(2)} \wequi \bigvee_{n \ge 0} \Sigma^{4n} \HW_{(2)}/8n. \]
\end{corollary}
\begin{proof}
Since $\HW \wequi \kw/\beta$ we have $\HW \wedge \HW_{(2)} \wequi (\kw \wedge \HW_{(2)})/h(\beta)$.
By Lemma \ref{lemm:hurewicz-beta}, up to a unit multiple we have $h(\beta) = 8x_1$.
Thus by Proposition \ref{prop:kw-HW-better}(1), our result will hold if we show that the map \[ \W(k)_{(2)} \wequi \pi_{4n-4} \kw \wedge \HW_{(2)} \xrightarrow{\times 8x_1} \pi_{4n} \kw \wedge \HW_{(2)} \wequi \W(k)_{(2)} \] is (up to a unit) given by multiplication by $8n$.
This follows from Proposition \ref{prop:typical-generators}.\NB{which says that $8x_1x_{n-1} = 8{n \choose 1}x_n$}
\end{proof}

\begin{remark}
Consider the commutative diagram
\begin{equation*}
\begin{CD}
\kw_{(2)} @>{\adamsphi}>> \Sigma^4 \kw_{(2)} @>{\beta}>> \kw_{(2)} \\
@VVV                         @VVV                         @VVV     \\
\kw \wedge \HW_{(2)} @>{\adamsphi}>> \Sigma^4 \kw \wedge \HW_{(2)} @>{\beta}>> \kw \wedge \HW_{(2)}.
\end{CD}
\end{equation*}
Theorem \ref{thm:main} implies that the lower map denoted $\adamsphi$ induces an isomorphism on positive homotopy groups.
Thus in order to determine the effect (up to unit) of the lower map $\beta$ on homotopy groups (in positive degrees), it suffices to determine the effect of the lower composite.
If $\W(k)$ has no torsion (e.g. $k=\R$), then it suffices to determine this effect rationally.
But rationally the vertical maps are isomorphisms, so it suffices to determine the effect of the top composite on homotopy.
This is given on $\pi_n$ by multiplication by $9^n-1$ (see Example \ref{ex:phi-beta}), which up to a $2$-adic unit is the same as $8n$ (see Lemma \ref{lemm:im-j-congruence}).
This provides an alternative proof of Corollary \ref{cor:HW-HW} over such fields.\NB{And in particular explains that the $8n$ is ``the same'' as in Theorem \ref{thm:stable-stems}.}
\end{remark}

\appendix
\section{The homotopy fixed point theorem} \label{sec:htpy-fixed}
We shall supply an alternative proof of the homotopy fixed point theorem (also known as homotopy limit problem) for hermitian $K$-theory of (certain) fields.
The original reference is \cite{hu2011homotopy} and uses a delicate analysis of some problems in equivariant motivic stable homotopy theory.
We shall instead use the improved version of Levine's slice converges theorem from \cite[\S5]{bachmann-bott} and the computation of the slice spectral sequence of $\KW$ \cite{rondigs2013slices} (see also Remark \ref{rmk:non-circular}).
We fix throughout a base field $k$.

Let $E \in \SH(k)$.
We denote by \[ E \to (f^\bullet E = \dots \to f^2 E \to f^1 E \to \dots) \] the \emph{slice tower}; here $f^n(E) := \cof(f_n E \to E)$.
We have $\colim_n f^\bullet E \wequi 0$ (see e.g. \cite[Lemma 3.1]{rso-solves}) and $\lim_n f^nE = \mathrm{sc} E$.
\begin{lemma} \label{lemm:KO-slice-complete} \NB{presumably could soup this up to say that $\ul{\pi}_{i,0}(f_\bullet(\dots)/(2,\rho))$ is locally nilpotent, whence $f_0(sc(\KO)/(2,\rho)) \wequi f_0(\KO/(2,\rho))$...}
Let $\chara(k) \ne 2$, $\vcd_2(k) < \infty$.
\begin{enumerate}
\item $\lim_n \map(\1, f^n \KW/2) \wequi \map(\1, \KW/2)$
\item $\lim_n \map(\1, f^n(\KO)/(2, \rho)) \wequi \map(\1, \KO/(2,\rho))$
\end{enumerate}
\end{lemma}
\begin{proof}
(1) By \cite[Theorem 6.12]{rondigs2013slices} the morphism $\map(\1, \KW) \to \lim_n \map(\1, f^n \KW)$ induces an isomorphism on $\pi_i$ for $i \not\equiv 0 \pmod{4}$ (both sides are zero in this case), and induces the $\I$-adic completion map $\W(k) \to \W(k)_\I^\comp$ for $i \equiv 0 \pmod{4}$.
Since $\vcd_2(k) < \infty$, this is thus a $2$-adic equivalence (see e.g. Lemmas \ref{lemm:witt-completion} (showing that $\W(k)_\I^\comp \wequi \W(k)_2^\comp$), \ref{lemm:witt-torsion} and \ref{lemm:derived-p-comp} (showing that $\W(k)_2^\comp \wequi L_2^\comp \W(k)$) and \ref{lemm:completion-homotopy} (allowing us to compute the homotopy groups of $\map(\1, \KW)_2^\comp$)), whence the claim.

(2) Consider the cofiber sequence $\ko \to \KO \to E$.\NB{Could just as well use $\KO_{\ge 0} \to \KO \to E' = \KO_{< 0}$.}
By \cite[Corollary 5.13]{bachmann-bott} we have \[ \lim_n \map(\1, f^n(\ko)/(2, \rho)) \wequi \map(\1, \ko/(2,\rho)); \] hence it suffices to prove the analogous claim about $E$.
By \cite[Corollary 5.13]{bachmann-bott} again we have \[ \lim_n \map(\1, f^n(\kw)/(2,\rho)) \wequi \map(\1, \kw/(2,\rho)). \]
Since $\rho = -2$ on $\eta$-periodic spectra, (1) implies that also \[ \lim_n \map(\1, f^n(\KW)/(2,\rho)) \wequi \map(\1, \KW/(2,\rho)), \] and hence $(*)$ the same holds for $E[\eta^{-1}]$ (note that $\kw = \ko[\eta^{-1}]$ and $\KW = \KO[\eta^{-1}]$).
We know the homotopy sheaves $\ul{\pi}_i(\KW)_0$ (given by $\ul{W}$ in degrees divisible by $4$, else $0$) and also $\ul{\pi}_i(\KO)_0$ for $i < 0$: namely they are the same (see e.g. \cite[Proposition 6.3]{schlichting2016hermitian}).
It follows that $(**)$ the map $E \to E[\eta^{-1}]$ induces an isomorphism on $\ul{\pi}_i(E)_0$ for all $i$, and hence on $f_0$.
It hence suffices to prove the following claim: \emph{for any spectrum $F \in \SH(k)$, the map $f_0 F \to F$ induces equivalences \[ \map(\Gmp{q}, f^n(f_0F)/(2,\rho)) \stackrel{(a)}{\wequi} \map(\Gmp{q}, f^n(F)/(2,\rho)) \] and \[ \map(\Gmp{q}, f_0(F)/(2,\rho)) \stackrel{(b)}{\wequi} \map(\Gmp{q}, F/(2,\rho)), \] for any $q \ge 0, n \in \Z$}.
Indeed then we find
\begin{align*} \lim_n \map(\1, f^n(E)/(2, \rho)) &\stackrel{(a)}{\wequi} \lim_n \map(\1, f^nf_0(E)/(2, \rho)) \\
           &\stackrel{(**)}{\wequi} \lim_n \map(\1, f^nf_0(E[\eta^{-1}])/(2, \rho)) \\
           &\stackrel{(a)}{\wequi} \lim_n \map(\1, f^n(E[\eta^{-1}])/(2, \rho)) \\
           &\stackrel{(*)}{\wequi} \map(\1, E[\eta^{-1}]/(2,\rho)) \\
           &\stackrel{(b)}{\wequi} \map(\1, f_0(E[\eta^{-1}])/(2,\rho)) \\
           &\stackrel{(**)}{\wequi} \map(\1, f_0(E)/(2,\rho)) \\
           &\stackrel{(b)}{\wequi} \map(\1, E/(2,\rho)).
\end{align*}

Since all functors in sight commute with taking the cofiber of $2$, we may replace $F$ by $F/2$ and ignore the modding out by $2$ part. (This is mainly for notational convenience.)
We have \[ \map(\Gmp{q}, G/\rho) \wequi \cof(\map(\Gmp{q}, \Gmp{-1} \wedge G) \xrightarrow{\rho} \map(\Gmp{q}, G)) \wequi \cof(\map(\Gmp{q+1}, G) \xrightarrow{\rho} \map(\Gmp{q}, G)). \]
Thus we may ignore taking the cofiber of $\rho$ as well.
Then $\Gmp{q} \in \SH(k)^\eff$ implies that the $(b)$ holds, and that $(a)$ holds for $n<0$ (since both sides are zero).
If $n \ge 0$ we additionally use that $f_0f^n \wequi f^nf_0$ (since $f_0 f_n \wequi f_n \wequi f_nf_0$).
\end{proof}
\begin{remark} \label{rmk:non-circular}
In the above proof, we have referred to the paper \cite{rondigs2013slices} for a certain spectral sequence computation.
That paper references the homotopy fixed point theorem, which may seem to lead to circular reasoning in the proof of Theorem \ref{thm:htpy-fixed} below.
However, the only reason why \cite{rondigs2013slices} uses the homotopy fixed point theorem is to determine the slices of $\KO$.
The paper \cite{ananyevskiy2017very} determines $s_*(\ko)$ independently, and from this we can deduce the slices of $s_*(\KO)$ via $s_*(\KO) \wequi s_*(\ko)[\beta^{-1}]$.
Thus there is no circularity.
\end{remark}

\begin{theorem}[homotopy fixed point theorem] \label{thm:htpy-fixed}
Let $\chara(k) \ne 2$, $\vcd_2(k) < \infty$.
The canonical map \[ \KO/2 \to \KGL^{hC_2}/2 \in \SH(k) \] induces an isomorphism on $\pi_{**}$.
\end{theorem}
\begin{proof}
By \cite[Corollary 3.9]{heard2017homotopy} (see also Lemma \ref{lemm:heard}), the map $\KO \to \KGL^{hC_2}$ is an $\eta$-equivalence.
Noting that $\KGL$ is $\eta$-complete and hence so is $\KGL^{hC_2}$, it hence suffices to show that $\pi_{**}(\KO/2) \wequi \pi_{**}(\KO_\eta^\comp/2)$.
Consider the fiber sequence $F \to \KO \to \KO_\eta^\comp$; we need to show that $\pi_{**}(F/2) = 0$.
Since $F$ is $\eta$-periodic, $\rho=-2$ on $F$ and we may as well show that $\pi_{**}(F/(2,\rho)) = 0$.
Again by $\eta$-periodicity, it suffices to show that $\map(\1, F/(2,\rho)) = 0$.
We have \[ F \wequi \lim \left[ \dots \xrightarrow{\eta} \Sigma^{2,2} \KO \xrightarrow{\eta} \Sigma^{1,1} \KO \xrightarrow{\eta} \KO \xrightarrow{\eta} \dots \right]. \]
Passing to the final subsystem of multiplication by $\eta^4$, and using the $\beta$-periodicity $\KO \wequi \Sigma^{8,4} \KO$, we can rewrite this as \[ F \wequi \lim \left[ \dots \xrightarrow{\eta^4\beta^{-1}} \Sigma^{-8} \KO \xrightarrow{\eta^4\beta^{-1}} \Sigma^{-4} \KO \xrightarrow{\eta^4\beta^{-1}} \KO \xrightarrow{\eta^4\beta^{-1}} \dots \right]. \]
We deduce that
\begin{align*} \map(\1, F/(2,\rho))
  \wequi &\lim_t \map(\1, \Sigma^{-4t} \KO/(2,\rho)) \\
  \stackrel{L.\ref{lemm:KO-slice-complete}(2)}{\wequi} &\lim_t \lim_n \map(\1, \Sigma^{-4t} f^n(\KO)/(2,\rho)) \\
  \wequi &\lim_n \lim_t \map(\1, \Sigma^{-4t} f^n(\KO)/(2,\rho)) \\
  =: &\lim_n F_n.
\end{align*}
Noting that $F_{-1} = 0$, it suffices to show that $S_n := \fib(F_{n} \to F_{n-1}) \wequi 0$ for all $n$.
Unwinding the definitions, we have \[ S_n = \lim_t \map(\1, \Sigma^{-4t} s_n(\KO)/(2,\rho)). \]
It follows from \cite[p.9]{ananyevskiy2017very} (see also \cite[Theorems 4.18 and 4.27]{rondigs2013slices}) that $s_n(\KO)/(2,\rho)$ is a sum of spectra of the form \[ \Sigma^{n-i}\Sigma^{n,n}\HZ/(2,\rho), \] for various $i \ge 0$.
Since (see e.g. \S\ref{subsub:HZ2}) \[ \pi_{*,0} \Sigma^{n-i}\Sigma^{n,n}\HZ/2 \wequi \begin{cases} k^M_{2n-i-*}(k) & *\ge n-i \\ 0 & \text{else} \end{cases} \] we deduce that $\map(\1, s_n(\KO)/(2,\rho)) \in \SH_{\le 2n+3}$.
This implies that \[ S_n = \lim_t \Sigma^{-4t} \map(\1, s_n(\KO)/(2,\rho)) \in \SH_{\le 2n+3-4t_0}, \] for any $t_0$, and thus $S_n = 0$.
This concludes the proof.
\end{proof}

\bibliographystyle{alpha}
\bibliography{bibliography}

\end{document}